\documentclass{amsart}
\usepackage[english]{babel}
\usepackage{amsmath,amscd,amssymb,latexsym,graphicx,stmaryrd}

\usepackage[all]{xy}
\usepackage{enumerate}
\usepackage{color}

\newtheorem*{acknowledgement}{Aknowledgement}

\reversemarginpar

\newcommand{\cA}{\mathcal{A}}
\newcommand{\cB}{\mathcal{B}}
\newcommand{\cC}{\mathcal{C}}
\newcommand{\cD}{\mathcal {D}}

\newcommand{\cF}{\mathcal{F}}
\newcommand{\cG}{\mathcal{G}}
\newcommand{\cH}{\mathcal{H}}

\newcommand{\cK}{\mathcal{K}}
\newcommand{\cL}{\mathcal{L}}

\newcommand{\cN}{\mathcal{N}}

\newcommand{\cP}{\mathcal{P}}
\newcommand{\cQ}{\mathcal{Q}}

\newcommand{\cS}{\mathcal{S}}
\newcommand{\cT}{\mathcal{T}}
\newcommand{\cU}{\mathcal{U}}
\newcommand{\cV}{\mathcal{V}}

\newcommand{\fS}{ \mathsf{S}}

\newcommand{\RR}{\mathbb R}
\newcommand{\CC}{\mathbb C}

\newcommand{\id}{\mathop{\mathrm{Id}}\nolimits}

\newcommand{\pb}[1]{{}^*#1^*}

\newcommand\smx{X^\bullet}
\newcommand\smX{X^\bullet}

\newcommand\pa{\partial}

\newcommand{\interior}[1]{\overset{\circ}{#1}}

\newtheorem{theorem}{Theorem}[section]
\newtheorem{proposition}[theorem]{Proposition}
\newtheorem{corollary}[theorem]{Corollary}
\newtheorem{lemma}[theorem]{Lemma}
\newtheorem{claim}{Claim}

\theoremstyle{definition}
\newtheorem{definition}[theorem]{Definition}

\theoremstyle{definition}

\theoremstyle{definition}
\newtheorem{remark}[theorem]{Remark}

\newcommand\CI{\mathcal{C}^{\infty}}
\newcommand\ff{\operatorname{ff}}
\newcommand\Hom{\operatorname{Hom}}
\newcommand\bbR{\mathbb{R}}
\newcommand\bbC{\mathbb{C}}
\newcommand\bbN{\mathbb{N}}

\newcommand\Id{\operatorname{Id}}
\newcommand\ad{\operatorname{sl}}

\newcommand\pr{\operatorname{pr}}
\newcommand\bbS{\mathbb{S}}
\newcommand\sus{\operatorname{sus}}
\newcommand\Diff{\operatorname{Diff}}
\newcommand\cf{cf\@. }

\newcommand\phg{\operatorname{ph}}
\newcommand\ed{\operatorname{ed}}
\newcommand\Con{\operatorname{C}}
\newcommand\nc{\operatorname{nc}}
\newcommand\quan{\operatorname{quan}}
\newcommand\PD{\operatorname{PD}}
\newcommand\fc{\operatorname{fc}}
\newcommand\fe{\operatorname{FE}_{\fS}}
\newcommand\ifc{\operatorname{ifc}}
\newcommand\fb{\operatorname{fb}}
\newcommand\Gr[1]{\mathcal{G}^{(#1)}}
\newcommand\SX{^\mathsf{S}\!X}
\newcommand\SN{^\mathsf{S}\!\mathcal{N}}

\newcommand\ev{\operatorname{ev}}

\newcommand\FCX{^\mathsf{FC}\!X}
\newcommand\FC{\mathsf{FC}}
\newcommand\MC{\mathsf{MC}}

\title{Pseudodifferential operators on manifolds with fibred corners} 
 
\author{Claire Debord}
\address{Laboratoire de Math\'ematiques, Universit\'e Blaise Pascal}
\email{debord@math.univ-bpclermont.fr}
\author{Jean-Marie Lescure}  
\address{Laboratoire de Math\'ematiques, Universit\'e Blaise Pascal}
\email{lescure@math.univ-bpclermont.fr}
\author{Fr\'ed\'eric Rochon}
\address{Department of mathematics, Australian National University}
\email{Frederic.Rochon@anu.edu.au} 
 
\begin{document} 
 
\maketitle

\newcommand\cd[1]{\marginpar{{\bf CD}: #1}}
\begin{abstract}
One way to geometrically encode  the singularities of a stratified pseudomanifold is to endow its interior with an iterated fibred cusp metric.  For such a metric, we develop and study a pseudodifferential calculus generalizing the $\Phi$-calculus of Mazzeo and Melrose.  Our starting point is the observation, going back to Melrose, that a stratified pseudomanifold can be `resolved' into a manifold with fibred corners.  This allows us to define pseudodifferential operators as conormal distributions on a suitably blown-up double space.  Various symbol maps are introduced, leading to the notion of full ellipticity.  This is used to construct refined parametrices and to provide criteria for the mapping properties of operators such as Fredholmness or compactness.  We also introduce a semiclassical version of the calculus and use it to establish a Poincar\'e duality between the $K$-homology of the stratified pseudomanifold and the $K$-group of fully elliptic operators.    
\end{abstract}

\tableofcontents


\section*{Introduction}

To study linear elliptic equations on singular spaces, it is very helpful to have a pseudodifferential calculus adapted to the geometry of the singularities.  Indeed, such a tool allows one to construct refined parametrices to geometric operators like the Laplacian, leading to a precise description of the space of solutions and typically having important consequences and applications in spectral theory, scattering theory, index theory and regularity theory.  This has also applications to study certain non-linear elliptic equations, see for instance \cite{Mazzeo-Montcouquiol},\cite{Jeffres-Mazzeo-Rubinstein}, \cite{Rochon-Zhang} for recent works in that direction.  Over the years, various types of pseudodifferential calculi have been introduced on non-compact and singular spaces, see for instance \cite{Mazzeo-Melrose0}, \cite{EMM}, \cite{MelroseAPS}, \cite{Schulze_book}, \cite{MazzeoEdge}, \cite{Mazzeo-Melrose} \cite{Lauter-Moroianu2}, \cite{Krainer}, \cite{Grieser-Hunsicker} and \cite{pomfb}.  Such a diversity of calculi comes from the fact that different types of singularities usually require quite different treatments.

Still, many of the examples above are concerned with a particular class of singular spaces:  stratified pseudomanifolds.  The notion of stratified pseudomanifold is relatively easy to describe and has the advantage of including many important examples of singular spaces, going from manifolds with corners to algebraic varieties.  One could therefore hope for a relatively uniform treatment of pseudodifferential operators in this context.  However, it is necessary to first choose a Riemannian metric geometrically encoding the singularities.  There are two natural choices.  To present these two choices, let us first consider a stratified pseudomanifold of depth one, that is, with only one singular stratum, see Figure~\ref{edge_metric} and Figure~\ref{fibred_cusp_metric}. 
\begin{figure}[h]
\setlength{\unitlength}{1cm}
\begin{minipage}[s]{5cm}
\begin{picture}(5,6)
\thicklines
\put(1,5){\line(1,-4){1}}
\put(0.9,3.6){$E$}
\put(1.1,4.6){\line(4,1){2}}
\put(1.1,4.6){\line(4,-1){2}}
\put(1.5,3){\line(4,1){2}}
\put(1.5,3){\line(4,-1){2}}
\put(1.9,1.4){\line(4,1){2}}
\put(1.9,1.4){\line(4,-1){2}}
\put(1.5,0.5){\vector(1,0){3}}
\put(4.5, 0.2){$r$}
\thinlines
\qbezier(3.1,5.1)(3,4.6)(3.1,4.1)
\qbezier(3.1,5.1)(3.2,4.6)(3.1,4.1)
\put(3.3,4.6){$L$}
\qbezier(3.5,3.5)(3.4,3)(3.5,2.5)
\qbezier(3.5,3.5)(3.6,3)(3.5,2.5)
\qbezier(3.9,1.9)(3.8,1.4)(3.9,0.9)
\qbezier(3.9,1.9)(4,1.4)(3.9,0.9)\end{picture} 
\caption{$g_{\ed}$}\label{edge_metric}
\end{minipage}
\hfill
\begin{minipage}[s]{5cm}
\begin{picture}(5,6)
\thicklines
\put(1,5){\line(1,-4){1}}
\put(0.9,3.6){$E$}
\qbezier(1.1,4.6)(3.1,4.6)(3.1,5.1)
\qbezier(1.1,4.6)(3.1,4.6)(3.1,4.1)
\put(3.3,4.6){$L$}
\qbezier(1.5,3)(3.5,3)(3.5,3.5)
\qbezier(1.5,3)(3.5,3)(3.5,2.5)
\qbezier(1.9,1.4)(3.9,1.4)(3.9,1.9)
\qbezier(1.9,1.4)(3.9,1.4)(3.9,0.9)
\put(1.5,0.5){\vector(1,0){3}}
\put(4.5, 0.2){$r$}
\thinlines
\qbezier(3.1,5.1)(3,4.6)(3.1,4.1)
\qbezier(3.1,5.1)(3.2,4.6)(3.1,4.1)
\qbezier(3.5,3.5)(3.4,3)(3.5,2.5)
\qbezier(3.5,3.5)(3.6,3)(3.5,2.5)
\qbezier(3.9,1.9)(3.8,1.4)(3.9,0.9)
\qbezier(3.9,1.9)(4,1.4)(3.9,0.9)
\end{picture}
\caption{$g_{\fc}$}\label{fibred_cusp_metric}
\end{minipage}
\end{figure}

The first choice, going back to Cheeger \cite{Cheeger}, is to consider an incomplete edge metric, a prototypical example being a metric which in a neighborhood of the singular stratum takes the form
\begin{equation}
    g_{\ed} = dr^2 + g_E + r^2 g_L,
\label{intro.1}\end{equation}        
where $r$ is the distance to the singular stratum, $g_E$ is a Riemannian metric on the singular stratum (the edge) and $g_L$ is a choice of metric on the link.  In this setting, a pseudodifferential calculus was developed independently by Mazzeo \cite{MazzeoEdge} and Schulze \cite{Schulze_book}.  In \cite{MazzeoEdge}, the metric which is really used as a starting point is in fact the conformally related metric 
\begin{equation}
     \widetilde{g}_{\ed}= \frac{g_{\ed}}{r^2}= \frac{dr^2}{r^2} + \frac{g_E}{r^2} + g_L,
\label{intro.1b}\end{equation}
a complete edge metric, which has the virtue of defining a Lie algebra of vector fields `generating' the pseudodifferential calculus.

Alternatively, one can consider a fibred cusp metric to encode the singularity, which is a certain type of complete Riemannian metric of finite volume on the regular stratum.  A prototypical example of such metric is one which near the singular stratum takes the form
\begin{equation}
    g_{\fc}= \frac{dr^2}{r^2}+ g_E + r^2 g_{L}.
\label{intro.2}\end{equation} 
For such a metric, a pseudodifferential calculus was introduced by Mazzeo and Melrose \cite{Mazzeo-Melrose} starting with a Lie algebra of smooth vector fields associated to the conformally related metric
\begin{equation}
 g_{\fb} = \frac{g_{\fc}}{r^2} = \frac{dr^2}{r^4} + \frac{g_E}{r^2} + g_L.  
\label{intro.2b}\end{equation}

Both \eqref{intro.1} and \eqref{intro.2} have analogs on general stratified pseudomanifolds by iterating the definition.  \  The generalization of \eqref{intro.1} is called an iterated edge metric \cite{Albin-Leichtnam-Mazzeo-Piazza} .  An important source of examples of iterated edge metrics is given by certain constant curvature metrics \cite{Mazzeo-Montcouquiol} and by K\"ahler-Einstein metrics singular along a divisor \cite{Jeffres-Mazzeo-Rubinstein}.   On the other hand, we call the analog of \eqref{intro.2} for a general stratified pseudomanifold an \textit{iterated fibred cusp metric}, see Definition~\ref{fc.1} below.  For stratified pseudomanifolds of depth one, natural examples of such metrics are given by certain fibred cusp K\"ahler-Einstein metrics, see \cite{Rochon-Zhang}.   

For iterated edge metrics, an associated pseudodifferential calculus has been introduced in \cite{NSS07a} and \cite{NSS07b} for operators of order zero and was used in \cite{NSS07c}.  There is also a recent survey \cite{Schulze09} by Schulze giving a nice description of how his methods can be adapted to stratified pseudomanifolds of higher depth.  Adopting a Lie groupoid point of view, one can obtain a pseudodifferential calculus by applying the general method of \cite{NWX} and \cite{ALN}, which works for both iterated edge metrics and iterated fibred cusp metrics.  This latter approach is suitable for certain applications in index theory, but the properness condition on the support of the operators makes it less appealing for the construction of refined parametrices.  Still, in certain cases, this can be avoided by introducing a length function, see \cite{LMN}. 

In this paper, we propose to systematically develop and study a calculus of pseudodifferential operators on stratified pseudomanifolds equipped with an iterated fibred cusp metric.  We call it the $\fS$-calculus.  Our approach takes its inspiration from \cite{Mazzeo-Melrose}, which deals with the case of a stratified pseudomanifold of depth 1.  In particular, we start with a Lie algebra of smooth vector fields associated to \textbf{iterated fibred corner metrics}, a type of metrics conformally related to iterated fibred cusp metrics.  To be able to consider stratified pseudomanifolds of arbitrary depth, our starting point is the idea, going back to Melrose (see \cite{Albin-Leichtnam-Mazzeo-Piazza}), that a stratified pseudomanifold can be resolved by a manifold with fibred corners.  This allows us to use blow-up techniques in a systematic way to construct the double space on which the Schwartz kernels of the operators can naturally be defined.  

To prove that this pseudodifferential calculus is closed under composition, we diverge from \cite{Mazzeo-Melrose} and follow an approach closer in spirit to \cite{Krainer}.  Beside the `usual' principal symbol, we introduce a `noncommutative' symbol for each singular stratum of the stratified pseudomanifold by restricting on a corresponding front face in the double space.  This lead to a simple Fredholm criterion for polyhomogeneous pseudodifferential operators:  an operator is Fredholm when acting on suitable Sobolev spaces if and only if              it is elliptic and its `noncommutative' symbols are invertible for each stratum.  We say such operators are fully elliptic.  For fully elliptic operators, we are able to construct a refined parametrix giving rise to a corresponding regularity result.   
This refined parametrix can also be used to show that our calculus is spectrally invariant, namely, that an invertible operator (when acting on suitable Sobolev spaces) has its inverse also contained in the calculus.  

Along the way, we have a parallel discussion that keeps track of the underlying Lie groupoid and relates our approach with the one of \cite{NWX} and \cite{ALN}.  This becomes particularly useful at the end of the paper, where we establish a Poincar\'e duality between  the fully elliptic $\fS$-operators and the K-homology of the stratified pseudomanifold.  In \cite{NSS07c}, such a Poincar\'e duality was obtained using the pseudodifferential operators of \cite{NSS07a} and \cite{NSS07b}.  Using instead groupoids, the first two authors introduced in \cite{Debord-Lescure} the noncommutative tangent space of a stratified pseudomanifold and showed its K-theory is Poincar\'e dual to the K-homology of the stratified pseudomanifold.  In fact, they showed more generally that the $C^*$-algebra of the noncommutative tangent space is dual in the sense of KK-theory to the $C^*$-algebra of continuous functions of the underlying stratified pseudomanifold.  

A key feature of our approach is the introduction of the semiclassical $\fS$-double space and its associated semiclassical $\fS$-calculus.  This allows us to define a continuous family groupoid $T\FCX$ playing the role in our context of the noncommutative tangent space of \cite{Debord-Lescure}.  By looking at the associated algebra of pseudodifferential operators, we can then provide a simple way of relating classes of fully elliptic $\fS$-operators with elements of the $K$-theory of $T\FCX$ (see Theorem~\ref{hcep.6} below).  This allows us to use a hybrid combination of the operator theoretic methods of \cite{NSS07c} (see also \cite{Melrose-Rochon06}) and the groupoid approach of \cite{Debord-Lescure} to obtain our Poincar\'e duality result in $KK$-theory, Theorem~\ref{dualite-Poincare-FC}.  In Theorem~\ref{interpretation-dp}, we also provide an interpretation of this result in terms of a quantization map for full symbols of fully elliptic $\fS$-operators,  relating in this way the points of view of \cite{Debord-Lescure} and \cite{NSS07c}.

The paper is organized as follows.  In Section~\ref{mwfc.0}, we introduce the definition of manifolds with fibered corners and recall from \cite{Albin-Leichtnam-Mazzeo-Piazza}  how they can be used to `resolve' stratified pseudomanifolds.  In Section~\ref{vfomfc}, we introduce a natural class of vector fields defined on a manifold with fibered corners $X$.  This leads to the notion of $\fS$-pseudodifferential operators in Section~\ref{def.S.op}. 
In Section~\ref{grpd}, we review the definition of groupoid and explain its relevance to the present context.  
In Section~\ref{asp.0}, we describe how $\fS$-pseudodifferential operators act on smooth functions.  
Section \ref{sus.0} is about suspended operators, which are used in Section~\ref{sm.0} to introduce various symbol maps for $\fS$-operators.  In section \ref{com.0}, we prove that the composition of two $\fS$-operators is again a $\fS$-operator.  In Section~\ref{mp.0}, we introduce natural Sobolev spaces on which $\fS$-operators act and provide criteria to determine when a $\fS$-operator is bounded, compact or Fredholm.  In Section~\ref{ac.0}, we introduce the semiclassical $\fS$-double space and the associated semiclassical $\fS$-calculus, as well as the Lie groupoid $T\FCX$.  This is used to obtain a relationship between classes of fully elliptic $\fS$-operators and elements of the K-theory of $T\FCX$.  Finally, in Section~\ref{pd.0}, we establish a Poincar\'e duality in $KK$-theory between $T\FCX$ and the stratified pseudomanifold ${}^\fS X$ associated to $X$ and interpret it in terms of a quantization map.

\begin{acknowledgement}
The authors are very grateful to Thomas Krainer for many helpful conversations.  
\end{acknowledgement}




\numberwithin{equation}{section}

\section{Manifolds with fibered corners and stratified pseudomanifolds}  \label{mwfc.0}

Let $X$ be a manifold with corners as defined in \cite{MelroseMWC}.  In particular, we are assuming that each boundary hypersurface $H\subset X$ is
embedded in $X$.  This means that there exists a boundary defining function $x_{H}\in \CI(X)$ such that $x_{H}^{-1}(0)=H$, $x_{H}$ is positive on
$X\setminus H$ and the differential $dx_{H}$ is nowhere zero on $H$.    
 In such a situation, one can choose 
 a smooth retraction $r_{H}:\cN_H \rightarrow
H \ ,$ where $\cN_H$ is a (tubular) neighborhood of $H$ in $X$ such
that $(r_H,x_H):\cN_H\rightarrow H\times [0,\infty)$ is a diffeomorphism on its
image. We call $(\cN_H, r_H, x_H)$ a \textbf{tube system} for $H$. 
A smooth
map $\phi:X\to Y$ between manifolds with corners is said to be a \textbf{fibration}\footnote{A more standard terminology would be \emph{smooth fibre bundle.}} if it is a locally trivial surjective submersion.

\begin{definition} \label{defmwfc}
Let $X$ be a compact manifold with corners and $H_{1},\ldots, H_{k}$ an exhaustive list of its set of boundary hypersurfaces $M_{1}X$.  Suppose that
each boundary hypersurface $H_{i}$ is the total space of a smooth fibration $\pi_{i}:H_{i}\to S_{i}$ where the base $S_{i}$ is also a compact manifold
with corners.   The collection of fibrations $\pi= (\pi_{1},\ldots, \pi_{k})$ is said to be an \textbf{iterated fibration structure} if there is a
partial order on the set of hypersurfaces such that  
  \begin{itemize} \item[(i)]  for any subset $I\subset \{1,\ldots,k\}$ with $\underset{i\in I}{\cap} H_i \not= \emptyset$, the set $\{H_i\ \vert \
i\in I\}$ is totally ordered. 
  \item[(ii)] If $H_i<H_j$, then $H_{i}\cap H_{j}\ne \emptyset$, $\pi_{i}(H_{i}\cap H_{j})=S_{i}$ 
  with $\pi_{i}: H_{i}\cap H_{j}\to S_{i}$  a surjective submersion and 
  $S_{ji}:= \pi_{j}(H_{i}\cap H_{j})\subset S_{j}$ is one of the boundary hypersurfaces of the 
  manifold with corners $S_{j}$.  Moreover,  there is a surjective submersion
    $\pi_{ji} : S_{ji}\rightarrow S_i$ such that on $H_i\cap H_j$ we have
    $\pi_{ji}\circ \pi_{j}=\pi_i$.  
   \item[(iii)] The boundary hypersurfaces of $S_j$ are exactly the $S_{ji}$ with $H_i<H_j$. In particular if $H_{i}$ is minimal, then $S_{i}$ is a
closed manifold.  
  \end{itemize}
 A \textbf{manifold with fibred corners} is a manifold with corners $X$ together with an iterated fibration structure $\pi$.  A smooth map $\psi: X\to X'$ between two manifolds with fibred corners $(X,\pi)$ and $(X',\pi')$ is said to be a \textbf{diffeomorphism of manifolds with fibred corners} if it is a diffeomorphism of manifolds with corners and if for each $H_i\in M_1 X$, there is $H_{\mu(i)}'\in M_1 X'$ and a diffeomorphism $\psi_i: S_i \to S_{\mu(i)}'$ inducing a commutative diagram
 \[
 \xymatrix{   H_i  \ar[r]^{\psi}   \ar[d]^{\pi_i} & H_{\mu(i)}' \ar[d]^{\pi_{\mu(i)}'} \\
                   S_i \ar[r]^{\psi_i} & S_{\mu(i)}'.
 }
 \]

\end{definition}

\begin{remark} With the previous notation, for any $j$, $S_j$ is naturally a manifold with fibered corners. The hypersurfaces are the $S_{ji}$ with
fibration $\pi_{ji} :S_{ji} \rightarrow S_i$ for any $i$ such that $H_i<H_j$. The same goes for the fibres of $\pi_i$. Precisely, if $x$ belongs to
$S_i$ let $L_i^x:=\pi_i^{-1}(x)$. Then $L^x_i$ is a manifold with fibered corners, where the boundary hypersurfaces are the $L^x_i\cap H_j$ with
$H_i <H_j$ and the corresponding fibration comes from the restriction of $\pi_j$. Notice that in the special case where $H_i$ is maximal,  the fibre $L^x_i$ is a closed manifold.
\end{remark}

\begin{definition} A family of tube system $(\cN_i,r_i,x_i)$ for $H_i$, $i=1,\ldots,k$ is an \textbf{iterated fibred tube system} of the manifold with 
fibred corners $X$ if the following condition holds for $H_i<H_j$, $$r_j(\cN_i\cap \cN_j)\subset \cN_i ,\ x_i\circ r_j=x_i,\  \pi_i\circ r_i\circ
r_j=\pi_i\circ r_i \mbox{ on } \cN_i\cap \cN_j ,$$ and the restriction to $H_j$ of the function $x_i$ is constant on the fibres of
$\pi_j$.
\end{definition}
If $X$ is equipped with an iterated fibred tube system, then for each $H_i\in M_1 X$, we have an induced iterated fibred tube system on each fibre of $\pi_i:H_i\to S_i$.  Similarly, there is an induced iterated fibred tube system on the base $S_i$.

To see that manifolds with fibred corners always admit iterated fibred tube systems, it is useful to describe tube systems in terms of vector fields.  Given a tube system $(\cN_H, r_H,x_H)$ for the boundary hypersurface $H$, one can naturally associate to it a vector field $\xi_H\in \CI(\cN_H;TX)$ such that
\[
          (r_H,x_H)_* \xi_H= \frac{\pa}{\pa x_H}.
\]
Clearly, the tube system can be recovered from this vector field by considering its flow.  More generally, if $\eta_H\in \CI(X;TX)$ is a vector field which is inner pointing an nowhere vanishing on $H$, but tangent to all other boundary hypersurfaces, we can construct a tube system $(\cN_H',r_H',x_H')$ such that 
\[
         (r_H',x_H')_* \eta_H= \frac{\pa}{\pa x_H'}
\]  
by considering the flow of $\eta_H$ for some short period of time $E_H$.  Thus, to obtain an iterated fibred tube system, it suffices to associate a vector field $\xi_{H_i}$ to each boundary hypersurface $H_i$ in such a way that,
\begin{itemize}
\item[(i)] The restriction $\left.\xi_{H_i}\right|_{H_i}$ is inner pointing and nowhere vanishing on $H_i$;
\item[(ii)] If $H_i < H_j$, then $\xi_{H_i}$ is tangent to $H_j$ and there is a vector field $\xi_{ji}\in \CI(S_j,TS_j)$ such that $(\pi_j)_* (\left. \xi_{H_i}\right|_{H_j})= \xi_{ji}$, while $\xi_{H_j}$ is tangent to the fibres of the fibration $\pi_i: H_i\to S_i$ on $H_i$.  Moreover, in a neighborhood of $H_i\cap H_j$, we have that $[\xi_{H_i}, \xi_{H_j}]=0$.  
\end{itemize}
Indeed, the flows of these vector fields generates tube system for each boundary hypersurface.  The condition that $[\xi_{H_i}, \xi_{H_j}]$ insures that the flows of $\xi_{H_i}$ and $\xi_{H_j}$ commute, so that in particular $x_i\circ r_j=x_i$ near $H_i\cap H_j$, while requiring $\xi_{H_j}$ to be tangent to the fibres of $\pi_i: H_i\to S_i$ insures that $\pi_i\circ r_i\circ r_j= \pi_i\circ r_j$ near $H_i\cap H_j$.  On the other hand, the condition 
$(\pi_j)_* (\left. \xi_{H_i}\right|_{H_j})=\xi_{ji}$ insures that $\left. x_i\right|_{H_j}$ is constant on the fibres of $\pi_j: H_j\to S_j$.  Thus, by shrinking the tube systems if necessary, we can insure they form an iterated fibred tube system.  

\begin{lemma}
A manifold with fibred corners always admit an iterated fibred tube system.  
\label{ifts.1}\end{lemma}
\begin{proof}
By the discussion above, it suffices to find  vector fields $\xi_{H_i}\in \CI(X;TX)$ for each boundary hypersurface $H_i\in M_1 X$ in such a way that conditions (i) and (ii) above are satisfied.  This requires to construct the vectors fields $\xi_{ij}$ on $S_j$ as well.  

Recall that the depth of $X$ is the highest codimension of a boundary face of $X$.  If $X$ has depth zero, that is, if $X$ is a smooth manifold, the lemma is trivial.  We can thus proceed by induction on the depth of $X$ to prove the lemma.  In particular, for each $i$, the base $S_i$ of the fibration $\pi_i$ is a manifold with fibred corners of depth smaller than the one of $X$ and we can assume we have vector fields satisfying $(i)$ and $(ii)$ on $S_i$.  We can denote the vector field associated to the boundary face $S_{ji}$ of $S_j$ by $\xi_{ji}$.  Proceeding by recurrence on the partial order of the $S_i$ to construct these vector fields, starting with $S_i$ minimal, we can assume furthermore that for $H_i<H_j<H_k$, we have 
\begin{itemize}
  \item[(iii)]  $(\pi_{kj})_*   \left.  \xi_{ki}\right|_{S_{kj}} =  \xi_{ji}$.
\end{itemize}

To construct the vector fields $\xi_{H_i}$ on $X$, we can proceed by recurrence on $M_1 X$ starting with maximal elements.  For a maximal element $H_i$, we just choose $\xi_{H_i}\in\CI(X;TX)$ such that $\left.\xi_{H_i}\right|_{H_i}$ is inner pointing and nowhere vanishing on $H_i$ and is tangent to the fibres of $\pi_j:H_j\to S_j$ for $H_j\ne H_i$.

Suppose now that $H_i$ is a hypersurface such that for all $H_j \in M_1 X$ with $H_j>H_i$, the corresponding vector field $\xi_{H_j}$ has been constructed in such a way that conditions $(i)$ and $(ii)$ hold.  To construct $\xi_{H_i}$, we first define it on the maximal elements in $\{H\in M_1X\; ; \; H>H_i\}$.  Let $H_j$ be such a maximal element.  We choose $\left.\xi_{H_i}\right|_{H_j}\in \CI(H_j;TH_j)$ in such a way that on $H_j$, we have that 
$(\pi_j)_* (\left. \xi_{H_i} \right|_{H_j})= \xi_{ji}$.  This can be done by  also requiring at the same time that
$[\left.\xi_{H_i}\right|_{H_j}, \left.\xi_{H_k}\right|_{H_j}]=0$ near $H_i\cap H_j\cap H_k$ for $H_k$ such that $H_j>H_k>H_i$.  Indeed,  we can do so by first constructing $\left.\xi_{H_i}\right|_{H_j}$ recursively on the boundary faces of $H_j$, starting with the boundary faces of smallest dimension, and extending the definition at the next level using the flow of  vector fields $\xi_{H_k}$ for $H_j>H_k>H_i$ whenever possible.  Thanks to the fact condition (ii) is satisfied by the vector fields $\xi_{H_k}$, this can be achieved consistently.  In this process, we also require that $\left.\xi_{H_i}\right|_{H_j}$ be tangent to the fibres of $\pi_l: H_l\to S_l$ for $H_l<H_i$.  Using the flow of $\xi_{H_j}$, we can then extend the definition of $\xi_{H_i}$ to a neighborhood of $H_i\cap H_j$.  

Doing this for all maximal elements in $\{H\in M_1X\; ;\; H>H_i\}$, we then proceed recursively to extend the definition of $\xi_{H_i}$ in a neighborhood of the other hypersurfaces of $\{H\in M_1X\; ;\; H>H_i\}$.   Thus, let $H_k>H_i$ be given and suppose that $\xi_{H_i}$ has already been defined in a neighborhood of $H_j$ for $H_j>H_i$ such that $H_j>H_k$.  Then $\left.\xi_{H_i}\right|_{H_k}$ is already defined in a neighborhood of $H_{j}\cap H_k$.  As before, proceeding by recurrence on the boundary faces of $H_k$, we can extend the definition of $\left.\xi_{H_i}\right|_{H_k}$ to all of $H_k$ in such a way that $(\pi_{k})_{*} (\left. \xi_{H_i} \right|_{H_k})= \xi_{ki}$ and 
 $[\left.\xi_{H_i}\right|_{H_k}, \left.\xi_{H_l}\right|_{H_j}]=0$ near $H_k\cap H_l$ for $H_k>H_l>H_i$.  We can also require $\left.\xi_{H_i}\right|_{H_k}$ to be tangent to the fibres of $\pi_l: H_l\to S_l$ for $H_l<H_i$.  Using the flow of $\xi_{H_k}$, we can then extend the definition of $\xi_{H_i}$ to a neighborhood of $H_k\cap H_i$.  Because the already defined vector fields satisfy condition (ii), this extension is consistent with the previous ones.  

Thus, proceeding recursively, we can extend the definition of $\xi_{H_i}$ to a neighborhood of 
\[
     \bigcup_{\{H\in M_1X\; ;\; H>H_i\}} H_i\cup H. 
\]
We can then extend this definition further in such a way that $\xi_{H_i}$ is tangent to the fibres of $\pi_k: H_k\to S_k$ for $H_k< H_i$.  This give us a vector field $\xi_{H_i}$ which together with the already existing vector fields satisfies conditions (i) and (ii), completing the inductive step and the proof.

\end{proof}

As observed by Melrose and subsequently described in \cite{Albin-Leichtnam-Mazzeo-Piazza}, there is a correspondence between manifolds with fibered corners and stratified pseudomanifolds.  For the convenience of the reader and in order to set up the notation, we will review the main features of this correspondence and refer to \cite{Albin-Leichtnam-Mazzeo-Piazza} for further details.   

 Let us first recall what are pseudomanifolds. We will use the notations and equivalent descriptions given by A. Verona in \cite{verona} and used by
the first
two authors in \cite{Debord-Lescure}.

\smallskip Let $\SX$ be a locally compact separable metrizable space. Recall that a $C^{\infty}$-stratification of $\SX$ is a pair
  $(\fS,N)$ such that,
\begin{enumerate}
 \item $\fS=\{s_i\}$ is a locally finite partition of
  $\SX$ into locally closed subsets of $\SX$, called the strata, which are
  smooth manifolds such that
  $$
    s_0\cap \bar{s_1}\not=\emptyset
    \mbox{ if and only if } s_0\subset \bar{s_1} .
  $$
  In that case we will write $s_{0}\leq s_{1}$ and $s_{0}<s_{1}$ if
  moreover $s_{0}\not= s_{1}$.  
\item $N=\{ \SN_s,\pi_s,\rho_s\}_{s\in \fS}$ is the set of control
  data, namely
  $\SN_s$ is an open neighborhood of $s$ in
  $\SX$, $\pi_s:\SN_s \rightarrow s$ is a continuous retraction and
  $\rho_s:\SN_s \rightarrow [0,+\infty[$ is a continuous map such that
  $s=\rho_s^{-1}(0)$. The map $\rho_s$ is either surjective or
  the constant zero function.\\ 
  Moreover if $\SN_{s_0}\cap s_1\not= \emptyset$, then the map 
  $$
    (\pi_{s_0},\rho_{s_0}):\SN_{s_0}\cap s_1
    \rightarrow s_0 \times ]0,+\infty[
  $$
  is a smooth proper submersion.
\item For any strata $s, t$ such that $s<t$, the set
  $\pi_{t}(\SN_{s}\cap\SN_{t})$ is included in  $\SN_{s}$ and the
  equalities
  $$
    \pi_{s}\circ \pi_{t}=\pi_{s} \makebox{ and } 
    \rho_{s}\circ\pi_{t}=\rho_{s}
  $$ 
  hold on $\SN_{s}\cap \SN_{t}$.
\item For any two strata $s_0$ and $s_1$ the following two statements hold,
  $$
    s_0\cap \bar{s_1}\not= \emptyset \mbox{ if and only if }
    \SN_{s_0}\cap s_1 \not= \emptyset; 
  $$ 
  $$
    \SN_{s_0}\cap \SN_{s_1} \not= \emptyset \mbox{ if and only if }
   s_0\subset \bar{s_1},\ s_0=s_1 \mbox{ or } s_1\subset \bar{s_0}.
  $$
\end{enumerate}

\medskip \noindent A stratification gives rise to a filtration.  Indeed, if $\SX_j$ is the union of strata of dimension
$\leq j$,  then,
 $$
  \emptyset \subset \SX_0 \subset \cdots \subset \SX_n =\SX .
 $$ 
We call $n$ the dimension of $\SX$ and $\smX:=\SX\setminus
\SX_{n-1}$ the regular part of $\SX$. The strata included in
$\smX$ are called {\it regular} while strata included in
$\SX\setminus \smx$ are called {\it singular}. The set of singular (resp. regular)
strata is denoted
$\fS_{sing}$ (resp.  $\fS_{reg}$).

\begin{definition}\label{defpseudomanifold}
A {\bf stratified pseudomanifold} is a triple $(\SX,\fS,N)$ where $\SX$ is a
locally compact separable metrizable space, $(\fS,N)$ is a $C^{\infty}$-stratification
on $\SX$ and the regular part $\smx$ is a dense open subset of $\SX$.
\end{definition}

Given a stratified pseudomanifold, notice that the closure of each of its strata is also naturally a stratified pseudomanifold.  Given a manifold with fibred corners $X$, there is a simple way of extracting a stratified pseudomanifold.  Let $H_1,\ldots,H_k$ an exhaustive list of its boundary hypersurfaces and let
$\pi=(\pi_1,\ldots,\pi_k)$ be the iterated fibration structure on $X$. 
On $X$,  consider the equivalence relation
\[   x\sim y \iff x=y \; \mbox{or} \; x,y\in H_i \ \mbox{with} \ \pi_i(x)=\pi_i(y) \; \mbox{for some} \ H_i.  
\]
We denote by $\SX$ the quotient space of $X$ by the previous equivalence relation and $q:X\rightarrow \ {\SX}$ the quotient map.  By construction, the
restriction of $q$ to $X\setminus \partial X$ is a homeomorphism. We claim that $\SX$ is naturally a stratified pseudomanifold.
 Indeed, for any $i\in I$ let $\sigma_i:=\pi_i(H_i\setminus \underset{H_k<H_i}{\bigcup} H_k)$, that is $\sigma_i=S_i\setminus \partial S_i$ . Then
$\sigma_i$ is a manifold and $\fS=\{X\setminus \partial X,\ \sigma_i\}$ is a locally finite partition of $\SX$ into strata. Choose an iterated fibered
tube system  $\{\cN_i,r_i,x_i\}$ for the manifold with fibred corners $X$. On can easily check that the map $(\pi_i\circ
r_i,x_i):r_i^{-1}(H_i\setminus \underset{H_k<H_i}{\bigcup} H_k)) \rightarrow \sigma_i \times [0,\infty[$ factors through a map $\SN_i\rightarrow
\sigma_i \times [0,\infty[$ where $\SN_i$ is the image of $r_i^{-1}(H_i\setminus \underset{H_k<H_i}{\bigcup} H_k))\subset \cN_i$ in $\SX$. This
enables us to define a set of control data on $\SX$ making it a stratified pseudomanifold. We call $\SX$ the stratified pseudomanifold  {\bf associated to} the manifold with fibered corners $X$.

Conversely, to any stratified pseudomanifold, one can associate a natural manifold with corners.  To see this, we need to recall the notion of depth, which is a good measure of complexity of a stratified pseudomanifold.

\begin{definition}
Let $(\SX,\fS,N)$ be a stratified pseudomanifold.  Then the depth $d(s)$ of a stratum $s$ is the biggest
$k$ such that one can find $k$ different strata $s_0,\cdots,s_{k-1}$ such that
$$s_0<s_1<\cdots<s_{k-1}<s_k:=s .$$ 
The depth of the stratification $(\fS,N)$ of $X$ is
$$d(X):= sup \{d(s),\ s\in \fS \} .$$
A stratum whose depth is $0$ will be called minimal.
\end{definition}
\begin{remark}
This definition is consistent with the notion of depth for manifolds with corners, which constitute a particular type of stratified pseudomanifolds.    Moreover, if $X$ is a manifold with fibred corners of depth $k$, then its associated stratified pseudomanifold ${}^{\fS}X$ also has depth $k$.  Notice that different conventions for the depth are also common, see for instance \cite{Albin-Leichtnam-Mazzeo-Piazza}.       
\end{remark}

Let $(\SX,\fS,N)$ be a stratified pseudomanifold. For any singular stratum $s$, set $L_s:=\rho_s^{-1}(1)$.  From
\cite{verona}, we know there is an isomorphism between $\SN_s$ and $L_s\times[0,+\infty[ / \sim_s$ where $\sim_s$ is the equivalence relation induced by
$(x,0)\sim_s (y,0)$ when $\pi_s(x)=\pi_s(y)$. This local triviality around strata enables to make the unfolding process of \cite{BHS} (see also
\cite{Debord-Lescure} for a complete description). If $s$ is minimal, one can construct a pseudomanifold 
$$ (\SX \setminus s ) \cup L_s\times [-1,1] \cup (\SX \setminus s ) \ $$ 
using the gluing coming from the trivialization of the neighborhood $\SN_s$ of $s$.  If $M$ is the set of minimal strata of $\SX$ and $m= \bigcup_{s\in M} s$ is the union of the minimal strata of $\SX$, then one can more generally construct the double stratified pseudomanifold
$$ 2X= (\SX \setminus m ) \cup \left(\bigsqcup_{s\in M}L_s\times [-1,1]\right) \cup (\SX \setminus m ) \ $$
by gluing $L_s\times [-1,1]$ for $s\in M$ via the trivialization of the neighborhood $\SN_s$ of $s$.   Since all the minimal strata of $\SX$ are involved, notice that the depth of $2X$ is one less that the one of $\SX$.  The stratified pseudomanifold $2X$ also comes with an involution $\tau_1$ interchanging the two copies of $\SX\setminus m$ with fixed point set identified with $L_m= \bigsqcup_{s\in M} L_s\times \{0\}$.  This fixed point set is naturally a stratified pseudomanifold and come with a surjective map $L_m\to m$ induced by the retractions in each neighborhood $\SN_s$ for $s\in M$. 

If $\SX$ has depth $k$, we can repeat this process $k$ times to obtain a stratified pseudomanifold $2^k X$ of depth 0, in other words, a smooth manifold.  The manifold $2^kX$ comes with a continuous surjective map $p:2^kX\to \SX$.  At the $j$th step of this unfolding process, we get a stratified pseudomanifold $2^j X$ with an involution $\tau_j$ and a fixed point set  equipped with a surjective map as before.  This lift to $2^k X$ to give $k$ involutions $\tau_1, \ldots,\tau_k$ with $k$ fixed point sets given by smooth hypersurfaces $R_1, \ldots, R_k$ equipped with smooth fibrations on each of their connected components.  The various bases of these fibrations are simply the smooth manifolds corresponding to the unfolded strata of $\SX$.    The complement $2^kX\setminus (\bigcup_{j=1}^{k} R_j)$ consists of $2^k$ copies of $\smx$.  The closure of any one of these copies is naturally a manifold with corners $\FCX$ with boundary hypersurfaces given by parts of the hypersurfaces $R_1,\ldots, R_k$ and have corresponding induced fibrations with bases given by manifolds with corners.  These fibrations give $\FCX$ a structure of manifold with fibred corners.  We call $\FCX$ the manifold with fibered corners {\bf associated to} the pseudomanifold $\SX$.

Up to the identifications described below, the two previous operations are mutually inverse.   
Precisely, starting with
a
stratified pseudomanifold $\SX$ and letting $\FCX$ be the associated manifold with fibered corners, we have for any $x,\ y$ in $\FCX$ that   
$x\sim y$ if and only if $p(x)=p(y)$. In other words, the map $p$ factors through a homeomorphism $\FCX/_{\sim} \longrightarrow \SX$ which  
 is a diffeomorphism in restriction to each strata and with respect to the control data. Conversely, starting with a manifold with fibered corners $X$ and letting $\SX$ be its associated stratified pseudomanifold, it can be seen that $\FCX$ is isomorphic to the original manifold with fibred corners $X$ by noticing that the unfolding process described above has an analog for manifolds with fibred corners obtained by gluing along boundary hypersurfaces and with the same resulting smooth manifold $2^kX$.

\section{Vector fields on manifolds with fibred corners}\label{vfomfc}

Let $X$ be a manifold with corners with $H_1,\ldots,H_k$ an exhaustive list of the boundary hypersurfaces of $X$.  For each $i\in\{1,\ldots,k\}$, let
$x_{i}\in \CI(X)$ be a boundary defining function for $H_{i}$.  Recall that 
\[
   \cV_{b}(X):= \{ \xi\in \Gamma(TX) \; ; \; \xi x_{i}\in x_{i}\CI(X)\; \forall i \}
\] 
is the Lie algebra of $b$-vector fields.  Notice in particular that a $b$-vector field $\xi\in \cV_{b}(X)$ is necessarily tangent to all the boundary
hypersurfaces of $X$.  Suppose that $\pi=(\pi_{1},\ldots,\pi_{k})$ is an iterated fibration structure on $X$.   

\begin{definition} The space $\cV_{\fS}(X)$ of \textbf{$\fS$-vector fields} on the manifold with fibred corners $(X, \pi)$ is  
\begin{multline}
\cV_{\fS}(X):= \{ \xi \in \cV_{b}(X); \left.\xi\right|_{H_{i}} \;\mbox{is tangent to the fibres of }
   \pi_{i}: H_{i}\to S_{i}  \\
   \mbox{and}\; \xi x_{i}\in x_{i}^{2}\mathcal{C}^{\infty}(X) \ \forall i \}.
\end{multline}    
\end{definition}
\begin{remark}
This definition depends on the choice of boundary defining functions.  To lighten the discussion, we might sometime not mention explicitly the choice of boundary defining functions, but the use of $\fS$-vector fields and related concepts always presuppose such a choice has been made to start with.  Moreover, we will usually assume the boundary defining functions are induced by an iterated fibred tube system.    
\end{remark}
As can be seen directly from the definition, $\cV_{\fS}(X)$ is a $\CI(X)$ module and is closed under the Lie bracket.  It is therefore a Lie
subalgebra of $\Gamma(TX)$.  
More generally,  the space $\Diff^{k}_{\fS}(X)$ of \textbf{$\fS$-differential operators of order $k$} is the space of operators on $\CI(X)$ generated
by $\CI(X)$ and product of up to $m$ elements of $\cV_{\fS}(X)$.   

Away from the boundary, a $\fS$-vector field is just like a usual vector field in $\CI(X;TX)$.  On the other hand, near a point $p\in \pa X$,  it is
useful to introduce a system of coordinates in which $\fS$-vector fields admits a simple description.  To this end,  let $H_{i_{1}}, \ldots,
H_{i_{\ell}}$ be the boundary hypersurfaces of $X$ containing the point
$p\in \pa X$.  After relabelling if necessary, we can assume $H_{1}, \ldots, H_{\ell}$ are the  boundary hypersurfaces containing $p$ and that 
\begin{equation}
   H_{1}< H_{2}<\cdots <H_{\ell}.  
\label{tdiff.3}\end{equation}  
Let $x_{i}\in \CI(X)$ be the chosen boundary defining function for $H_{i}$.  Consider a small neighborhood of $p$ where for each
$i\in\{1,\ldots,\ell\}$, the fibration $\pi_{i}: H_{i}\to S_{i}$ restricts to be trivial.  Consider then tuples of functions $y_{i}= (y^{1}_{i},
\ldots, y^{k_{i}}_{i})$ for $1\le i\le \ell$ and $z=(z^{1},\ldots, x^{q})$ such that 
\begin{equation}
(x_{1}, y_{1}, \ldots, x_{\ell},  y_{\ell}, z)
\label{coord.1}\end{equation}
form coordinates near $p$ with the property that on $H_{i}$,  $(x_{1},\ldots,x_{i-1}, y_{1}, \ldots, y_{i})$ induce coordinates on the base
$S_{i}$ of the fibration $\pi_{i}: H_{i}\to S_{i}$ such that $\pi_{i}$ corresponds to the map
\begin{equation}
 (x_{1}, \ldots, \widehat{x}_{i},\ldots x_{\ell}, y,z) \to 
        (x_{1}, \ldots, x_{i-1},y_{1},\ldots,  y_{i}),
\label{tdiff.3b}\end{equation}
where the $\;\widehat{ }\;$ notation above the variable $x_{i}$ means it is omitted. 
Thus, the coordinates $(x_{i+1},\ldots, x_{\ell}, y_{i+1},\ldots,y_{\ell},z)$ restrict to give coordinates on the fibres of the fibration $\pi_{i}$.  
    With such coordinates, the Lie algebra $\cV_{\fS}(X)$ is locally spanned over $\CI(X)$ by the vector fields
\begin{equation}
  \frac{\pa}{\pa z^{j}},  \; w_{i}x_{i} \frac{\pa}{\pa x_{i}},  \; w_{i}\frac{\pa}{\pa y_{i}^{n_{i}}},   
\label{tdiff.4a}\end{equation}   
for $ j\in\{1,\ldots,q\}, \; i\in \{1,\ldots,\ell\}, \;  n_{i}\in \{1,\ldots, k_{i} \},$ where $w_{i}= \prod_{m=i}^{\ell}x_{m}$.  Thus, in these
coordinates,  a $\fS$-vector field $\xi\in \cV_{\fS}(X)$ is of the form 
\begin{equation}
\xi = \sum_{i=1}^{\ell} a_{i} x_{i}w_{i}\frac{\pa}{\pa x_{i}}  +\sum_{i=1}^{\ell} \sum_{j=1}^{k_{i}} b_{ij}w_{i}\frac{\pa}{\pa y_{i}^{j}} +
\sum_{m=1}^{q} c_{m} \frac{\pa}{\pa z^m},
\label{tdiff.4b}\end{equation}
with $a_{i}, b_{ij}, c_{m}\in \CI(X)$. \\


Since $\cV_{\fS}(X)$ is a $\CI(M)$-module, there exists a smooth
vector bundle ${}^{\pi}TX\to X$ and a natural map $\iota_{\pi}:{}^{\pi}TX\to
TX$ which restricts to an isomorphism  on $X\setminus \pa X$ such that 
\begin{equation}
     \cV_{\fS}(X)=  \iota_{\pi} \CI(X;{}^{\pi}TX).
\label{pts.1}\end{equation}
At a point $p\in X$, the fibre of ${}^{\pi}TX$ above $p$ can be defined by
\begin{equation}
   {}^{\pi}T_{p}X= \cV_{\fS}/ \mathcal{I}_{p}\cdot \cV_{\fS} ,
\label{pts.2}\end{equation}
where $\mathcal{I}_{p}\subset \CI(X)$ is the ideal of smooth functions 
vanishing at $p$.  Although the map $\iota_{\pi}: {}^{\pi}TX\to TX$ fails to be  an isomorphism of vector bundles, notice that ${}^{\pi}TX$ is nevertheless isomorphic to $TX$, but not in a natural way.

Unless the hypersurface $H_i$ has no boundary, notice that the kernel of the natural map $\left.{}^{\pi}TX\right|_{H_i}\to \left. TX\right|_{H_i}$ does not form a vector bundle over $H_i$.  To obtain a vector bundle on $H_i$, we need to introduce an intermediate vector bundle in between ${}^{\pi}TX$ and $TX$.  Let 
\[
      X_i= X\cup_{H_i} X
\]
be the manifold with corners obtained by gluing two copies of $X$ along the boundary hypersurface $H_i$.  The manifold $X_i$ naturally has an iterated fibration structure induced from the one of $X$.  Hoping this will lead to no confusion, we will also denote this iterated fibration by $\pi$.  We then have a corresponding Lie algebra $\cV_{\fS}(X_i)$ of $\fS$-vector fields as well as an associated $\fS$-tangent vector bundle ${}^{\pi}TX_i$.  Consider then the restriction of this vector bundle to one of the two copies of $X$ inside $X_i$,
\[
       {}^{\pi \setminus \pi_i}TX=  \left.{}^{\pi}TX_i\right|_{X}.
\]
Away from $H_i$, the vector bundle ${}^{\pi\setminus\pi_i}TX$ is canonically isomorphic to ${}^{\pi}TX$.  However, seen as a subspace $\CI(X;TX)$, the space of sections $\CI(X;{}^{\pi\setminus\pi_i}TX)$ is slightly bigger than $\cV_{\fS}(X)$.  We have in fact the following natural sequence of maps
\[
\xymatrix{
     {}^{\pi}TX \ar[r]  & {}^{\pi\setminus \pi_i}TX \ar[r]^{a_{\pi\setminus\pi_i}} & TX.
}
\]
When restricted to the boundary hypersurface $H_i$, the first map ${}^{\pi}TX \to {}^{\pi\setminus \pi_i}TX$ 
is such that its kernel ${}^{\pi}NH_i$ is naturally a vector bundle over $H_i$.  This vector bundle is the pullback of a vector bundle on $S_i$.  To see this, set 
\[
      {}^{\pi}TH_i := \{  \xi\in \left.{}^{\pi\setminus \pi_i}TX \right|_{H_i} \; ; \;  a_{\pi\setminus\pi_i}(\xi)\in TH_i\}.
\]
The fibration $\pi_i: H_i\to S_i$ induces the short exact sequence
\[
\xymatrix{
0 \ar[r] & {}^{\pi}T(H_i\setminus S_i) \ar[r] & {}^{\pi}TH_i \ar[r]^{(\pi_i)_*} &
  \pi_i^{*}{}^{\pi}TS_i \ar[r] & 0,
}
\]
where ${}^{\pi}TS_i$ is the $\fS$-tangent bundle of $S_i$ and ${}^{\pi}T(H_i\setminus S_i)$ is such that its restriction to each fibre $F_i$ of 
the fibration $\pi_i$ is the $\fS$-tangent bundle ${}^{\pi}TF_i$ of that fibre.  In particular,   this induces a canonical identification $\pi_{i}^{*}{}^{\pi}TS_{i} = {}^{\pi}TH_{i} / {}^{\pi}T(H_{i}/S_{i})$.  
Now, using the vector field $x_{i}^{2}\frac{\pa}{\pa x_{i}}$ induced by a tube system for $H_{i}$, we have a natural decomposition
\begin{equation}
   {}^{\pi}N H_{i} \cong \left(  {}^{\pi}TH_{i}/ {}^{\pi}T(H_{i}/ S_{i})\right) \times \bbR.
\label{nor.1}\end{equation}
This means a tube system for $H_{i}$ induces an isomorphism of vector
bundles 
\begin{equation}
    {}^{\pi}NH_{i} \cong \pi_{i}^{*} {}^{\pi}NS_{i}
\label{nor.2}\end{equation}
where ${}^{\pi}NS_{i}= {}^{\pi}TS_{i}\times \bbR$, and thus, a corresponding fibration
\begin{equation}
\xymatrix{  F_{i} \ar@{-}[r] &   {}^{\pi}NH_{i} \ar[d]^{(\pi_{i})_{*}}\\
                & {}^{\pi}NS_{i}.
}\label{nor.3}\end{equation}

An \textbf{iterated fibred corner metric}   (or \textbf{$\fS$-metric}) is a choice of metric $g_{\pi}$ for the vector bundle ${}^{\pi}TX$.  Via the map $\iota_{\pi}: {}^{\pi}TX\to TX$, it
restricts to give a complete Riemannian metric on $X\setminus \pa X$.  In the local coordinates \eqref{coord.1},  a special example of such a metric
would be
\begin{equation}
  g_{\pi}=  \sum_{i=1}^{\ell}   \frac{dx_{i}^{2}}{(x_{i}w_{i})^{2}} +  \sum_{i=1}^{\ell} \sum_{j=1}^{k_{i}} \frac{ (dy_{i}^{j})^{2}}{ w_{i}^{2}} +
\sum_{m=1}^{q} dz_{m}^{2}. 
\label{metric.2}\end{equation} 
The Laplacian associated to a $\fS$-metric is an example of $\fS$-differential operator of order $2$.  Similarly, if ${}^{\pi}TX\to X$ has a spin
structure, then the corresponding Dirac operator associated to a $\fS$-metric is a $\fS$-differential operator of order $1$.  

The \textbf{$\fS$-density bundle} associated to a manifold with fibred corners is the smooth real line bundle ${}^{\pi}\Omega$ with fibre above $p\in
X$ given by
\begin{equation}
   {}^{\pi}\Omega_{p}= \{ u: \Lambda^{\dim X}({}^{\pi}T_{p}X)\to \bbR \; ; \; 
                  u(t\omega)= |t| u(\omega),  \;\; \forall \omega\in \Lambda^{\dim X}({}^{\pi}T_{p}X), \; t\ne 0\}.
\label{density.1}\end{equation}    
A \textbf{$\fS$-density} is  an element of $\CI(X;{}^{\pi}\Omega)$.  In particular, the volume form of a $\fS$-metric is naturally a $\fS$-density.   
Via the map $\iota_{\pi}: {}^{\pi}TX\to TX$, a $\fS$-density restricts to give a density on the interior of $X$.  Let $\nu\in \CI(X;\Omega)$ be a
non-vanishing density on $X$, where $\Omega=\Omega(TX)$ is the density bundle associated to $TX$.   On $X\setminus \pa X$, a $\fS$-density $\nu_{\pi}$
can be written in terms of $\nu$ as
\begin{equation}
   \nu_{\pi} =   \left(\prod_{i=1}^{k} x_{i}^{2+\dim S_{i}} \right)^{-1}  h\nu , \quad \mbox{for some} \; h\in \CI(X).     
\label{density.2}\end{equation}   

As indicated in the introduction, $\fS$-metrics are conformally related to another type of metrics geometrically encoding the singularities of the stratified pseudomanifold.

\begin{definition}
On a manifold with fibred corners $(X,\pi)$ with a boundary defining function $x_H$ specified for each boundary hypersurface $H\in M_1 X$, an \textbf{iterated fibred cusp metric} $g_{\ifc}$ is a metric of the form
\[
        g_{\ifc}= x^2 g_{\pi},  \quad x= \prod_{H\in M_1X} x_H,
\]
where $g_{\pi}$ is a $\fS$-metric.   
\label{fc.1}\end{definition}

\section{The definition of $\fS$-pseudodifferential operators}\label{def.S.op}

Let $X$ be a manifold with fibred corners.  Let $H_{1},\ldots, H_{k}$ be an exhaustive list of its boundary hypersurfaces with
$x_{1},\ldots,x_{k}\in \CI(X)$ a choice of  corresponding boundary defining functions and 
$\pi_{i}: H_{i}\to S_{i}$ the corresponding fibrations.  
Consider the Cartesian product $X^{2}=X\times X$ with the projections $\pr_{R}: X\times X\to X$ 
and $\pr_{L}: X\times X \to X$ on the right and left factors respectively.  Then 
$x_{i}' := \pr_{R}^{*}x_{i}$ and $x_{i}:= \pr_{L}^{*} x_{i}$ are boundary defining functions
for $X\times H_{i}$ and $H_{i}\times X$ respectively.  The $b$-double space 
$X^{2}_{b}$ is the space obtained from $X^{2}$ by blowing up the p-submanifolds 
$H_{1}\times H_{1}, \ldots, H_{k}\times H_{k}$,
\begin{equation}
  X^{2}_{b}:= [X^{2}; H_{1}\times H_{1};\ldots; H_{k}\times H_{k}],
\label{ds.1}\end{equation}
with blow-down map $\beta_{b}: X^{2}_{b}\to X^{2}$.  Near $H_{i}\times H_{i}$, this amounts
to the introduction of polar coordinates
\[
     r_{i}:= \sqrt{ x_{i}^{2}+ (x_{i}')^{2}}, \quad \omega_{i}= \frac{x_{i}}{r_{i}}, \;
                   \omega_{i}'= \frac{x_{i}'}{r_{i}},
\]
where $r_{i}$ is a boundary defining function for the `new' hypersurface 
\[
     B_{i}:= \beta_{b}^{-1}(H_{i}\times H_{i}) \subset X^{2}_{b}
\] 
introduced by the blow-up, while near $B_{i}$, the functions
$\omega_{i}$ and $\omega_{i}'$ are boundary defining functions for the lifts of the `old' boundary hypersurfaces.  Notice that since $H_{i}\times
H_{i}$ and $H_{j}\times H_{j}$ are transversal
as $p$-submanifolds for $i\ne j$, the diffeomorphism class of $X^{2}_{b}$ stays the same if we change the order in which we blow up (\cf Proposition
5.8.2 in \cite{MelroseMWC} or \cite[p.21]{Mazzeo-MelroseETA}).  

To define the $\pi$-double space, consider the fibre diagonal on the $p$-submanifold $H_{i}\times
H_{i}$,
\begin{equation}
  D_{\pi_{i}}= \{ (h,h') \in H_{i}\times H_{i}\; ; \;  \pi_{i}(h)= \pi_{i}(h')\}.
\label{ds.2}\end{equation}    
To lift this $p$-submanifold to the front face $B_{i}$, notice that 
\[
      B_{i}= SN^{+}(H_{i}\times H_{i})
\] 
is by definition a quarter of circle bundle on $H_{i}\times H_{i}$ giving a canonical decomposition
\begin{equation}
    B_{i}= (H_{i}\times H_{i})\times [-1,1]_{s_{i}}, \ \mbox{with} \, s_{i}:= \omega_{i}-\omega_{i}'.
\label{ds.3}\end{equation}
Thus, we can define a lift of $D_{\pi_{i}}$ to $B_{i}$ by 
\begin{equation}
\Delta_{i}:= \{ (h,h',0)\in H_{i}\times H_{i}\times [-1,1]_{s_{i}} \; ; \; 
   \pi_{i}(h)= \pi_{i}(h') \}.
\label{ds.4}\end{equation}

The space $\Delta_{i}$ is a $p$-submanifold of $B_{i}$ and $X^{2}_{b}$.  To obtain the
$\pi$-double space, it suffices to blow up $\Delta_{i}$ in $X^{2}_{b}$ for $i\in \{1,\ldots,k\}$.
As opposed to the definition of $X^{2}_{b}$, the order in which the blow-ups are performed is important,
different orders leading to different diffeomorphism classes of manifolds with corners.  Fortunately, our assumptions on the partial order of
hypersurfaces of $X$ give us a systematic way to proceed.

More precisely, assume that the hypersurfaces of $X$ are labeled in such a way that
\begin{equation}
   i<j  \; \Longrightarrow  \; H_{i} < H_{j} \; \mbox{or} \; H_{i}\cap H_{j}=
   \emptyset .
\label{label.1}\end{equation}
With this convention, we define the $\pi$-double space by 
\begin{equation}
     X^{2}_{\pi}:= [ X^{2}_{b}; \Delta_{1};\ldots ; \Delta_{k} ].
\label{ds.5}\end{equation}
See Figure~\ref{pidouble} for a picture of $X^{2}_{\pi}$ when $X$ is a manifold with boundary.  Notice that the order in which we blow up is not completely determined by \eqref{label.1}, but a different choice of ordering satisfying
\eqref{label.1}  would  amount in commuting the
blow-ups of $p$-submanifolds which do not intersect, an operation which does not affect the diffeomorphism class of $X^{2}_{\pi}$.

\begin{figure}
\setlength{\unitlength}{0.7cm}
\begin{picture}(6,6)
\thicklines
\qbezier(0,3)(0.78,3)(1.5,2.6)
\qbezier(3,0)(3,0.78)(2.6,1.5)
\qbezier(1.5,2.6)(2.05,3.15)(2.6,2.6)
\qbezier(2.6,1.5)(3.15,2.05)(2.6,2.6)
\put(2.1,2.2){$\ff_{\pi}$}

\put(0,3){\vector(0,1){3}}
\put(-0.4,5.8){$x$}

\put(3,0){\vector(1,0){3}}
\put(6,-0.4){$x'$}

\thinlines
\put(2.6,2.6){\line(1,1){0.53}}
\put(3.4,3.4){\line(1,1){0.53}}
\put(4.2,4.2){\line(1,1){0.53}}
\put(5,5){\line(1,1){0.53}}
\put(4.4,3.8){$\Delta_{\pi}$}

\end{picture}
\caption{The $\pi$-double space}\label{pidouble}
\end{figure}
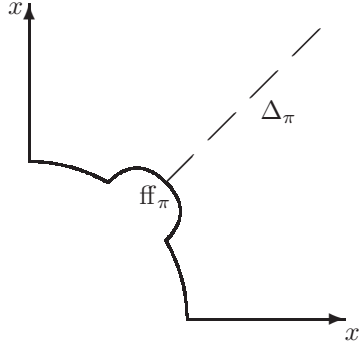

We have corresponding blow-down maps 
\begin{equation}
  \beta_{\pi-b}:X^{2}_{\pi}\to X^{2}_{b}, \quad \beta_{\pi}:= \beta_{b}\circ \beta_{\pi-b}: X^{2}_{\pi}\to X^{2}.
\label{ds.6}\end{equation} 
We denote the `new' hypersurface introduced by blowing up $\Delta_{i}$ by 
\begin{equation}
  \ff_{\pi_{i}}:= (\beta_{\pi-b})^{-1}(\Delta_{i}) \subset X^{2}_{\pi} .
\label{ds.7}\end{equation}
The $p$-submanifold
$\ff_{\pi_{i}}$ is called the \textbf{front face} associated to the boundary hypersurface $H_{i}$.   
  Let also 
  \[
  \Delta_{\pi}:= \overline{\beta_{\pi}^{-1}(
\overset{\circ}{\Delta}_{X})}
\]
denote  the lift of the diagonal $\Delta_{X}\subset X\times X$ to
$X^{2}_{\pi}$, where $\overset{\circ}{\Delta}_{X}$ is the interior of $\Delta_{X}$.   On $X^{2}_{\pi}$, we are particularly interested in the lift of
the Lie algebra $\cV_{\fS}(X)$ with respect to the natural maps
\begin{equation}
  \pi_{L}= \pr_{L}\circ \beta_{\pi}: X^{2}_{\pi}\to X, \quad \pi_{R}= \pr_{R}\circ \beta_{\pi}: X^{2}_{\pi}\to X.
\label{diff.1}\end{equation} 
\begin{lemma}
The lifted diagonal $\Delta_{\pi}$ is a $p$-submanifold of $X^{2}_{\pi}$.  Furthermore, the lifts to $X^{2}_{\pi}$ of the Lie algebra $\cV_{\fS}(X)$
via the maps $\pi_{L}$ and $\pi_{R}$ are transversal to the lifted diagonal $\Delta_{\pi}$.  
\label{diff.2}\end{lemma} 
\begin{proof}
The result is trivial in the interior of $X^{2}_{\pi}$.  Thus, let $p\in \Delta_{\pi}\cap \pa X^{2}_{\pi}$ be given.  We need to show that the lemma
holds in a neighborhood of $p$ in $X^{2}_{\pi}$.  Moreover, by symmetry, we only have to prove the result for the lift of $\cV_{\fS}(X)$ with respect
to the map $\pi_{L}$.    Let $H_{i_{1}}, \ldots, H_{i_{\ell}}$ be the boundary hypersurfaces of $X$ containing the point
$\pi_{L}(p)\in \pa X$.  After relabelling if necessary, we can assume $H_{1}, \ldots, H_{\ell}$ are the hypersurfaces containing $\pi_{L}(p)$ and
that 
\begin{equation}
   H_{1}< H_{2}<\cdots <H_{\ell}.  
\label{diff.3}\end{equation}  
Near $\pi_{L}(p)\in \pa X$, let $(x_{i}, y_{i}, z)$ be coordinates as in \eqref{coord.1}.  Recall that in such coordinates, 
the Lie algebra $\cV_{\fS}(X)$ is locally spanned over $\CI(X)$ by the vector fields
\begin{equation}
  \frac{\pa}{\pa z^{j}},  \; w_{i}x_{i} \frac{\pa}{\pa x_{i}},  \; w_{i}\frac{\pa}{\pa y_{i}^{n_{i}}},   
\label{diff.4}\end{equation}   
for $ j\in\{1,\ldots,q\}, \; i\in \{1,\ldots,\ell\}, \;  n_{i}\in \{1,\ldots, k_{i} \},$ where $w_{i}= \prod_{m=i}^{\ell}x_{m}$.  

On $X^{2}$, we can then consider the coordinates 
\begin{equation}
x_{i},y_{i},z, x_{i}',y_{i}',z',
\label{diff.4b}\end{equation}
 where $(x_{i}, y_{i},z)$ is seen as the pullback of our coordinates from the left factor of $X^{2}$, while $(x_{i}', y_{i}',z')$ is the pullback of
our coordinates from the right factor of $X^{2}$.  On the $b$-double space $X^{2}_{b}$, we can then consider  the local coordinates
\begin{equation}
    r_{i}= x_{i}',  \; s_{i}= \frac{x_{i}-x_{i}'}{x_{i}'}, \; y_{i}, \; y_{i}', \;z,\; z',
\label{diff.5}\end{equation}      
In these coordinates, we have that
\[
         x_{i}= r_{i}(s_{i}+1), \quad w_{i}= w_{i}' \sigma_{i},  \quad \mbox{where} \;  w_{i}'= \prod_{j=i}^{\ell} r_{j}, \quad \sigma_{i}=
\prod_{j=i}^{\ell} (s_{j}+1).
\]
Thus, under the map $pr_{L}\circ \beta_{b}$, the vector fields of \eqref{diff.4} lift to
\begin{equation}
   \frac{\pa}{\pa z^{j}},  \; w_{i}'\sigma_{i}(s_{i}+1) \frac{\pa}{\pa s_{i}},  \; w_{i}'\sigma_{i}\frac{\pa}{\pa y_{i}^{n_{i}}}.   
   \label{diff.6}\end{equation}
Finally, on $X^{2}_{\pi}$, we can consider the local coordinates near $p$ given by
\begin{equation}
r_{i}= x_{i}',  \; S_{i}= \frac{s_{i}}{w_{i}'}, \; Y_{i}= \frac{y_{i}-y_{i}'}{w_{i}'}, \; y_{i}', \;z,\; z'.
\label{diff.7}\end{equation}
In these coordinates, the lifted diagonal $\Delta_{\pi}$ is given by the subset 
$\{S_{i}=0, Y_{i}=0, z=z'\}$.  In particular, this shows it is a $p$-submanifold.  From \eqref{diff.6}, we also see that under the map $\pi_{L}$,  the
vector fields of \eqref{diff.4} lift  to 
\begin{equation}
    \frac{\pa}{\pa z^{j}}, \; \sigma_{i}(S_{i}w_{i}'+1) \frac{\pa}{\pa S_{i}},  \;\sigma_{i}\frac{\pa}{\pa Y_{i}^{n_{i}}}, \quad \mbox{where}
\; \sigma_{i}= \prod_{j=i}^{\ell} (S_{j}w_{j}'+1).   
\label{diff.8}\end{equation}
These vector fields are clearly transversal to the lifted diagonal 
$\Delta_{\pi}= \{ S_{i}=0, Y_{i}=0, z=z'\}$, which completes the proof.

\end{proof}

\begin{corollary}
The natural diffeomorphism $\Delta_{\pi} \cong X$ induced by the map $\pi_{L}$ (or alternatively by the map  $\pi_{R}$) is covered by natural
identifications
\[
       N\Delta_{\pi} \cong {}^{\pi}TX,  \quad N^{*}\Delta_{\pi}\cong {}^{\pi}T^{*}X,
\] 
where $N\Delta_{\pi}$ is the normal bundle of $\Delta_{\pi}$ in $X^{2}_{\pi}$.
\label{ds.9}\end{corollary}

\begin{remark}
It can also be proved that the Lie algebra $\cV_{\fS}(X)$ lifts via $\pi_{L}$ or $\pi_{R}$ to give a Lie subalgebra of $\cV_{b}(X^{2}_{\pi})$.  Near
$\Delta_{\pi}$, this follows from the local description \eqref{diff.8}.  Since we do not need this fact elsewhere on $X^{2}_{\pi}$, we omit the proof.
 
\label{diff.9}\end{remark}

On the $\pi$-double space, the Schwartz kernels of $\fS$-differential operators admit a simple description.  Let us first describe the Schwartz kernel
of the identity operator.  In the local coordinates  \eqref{diff.4b} near  $\Delta_{X}\cap(\pa X\times \pa X)$ , the Schwartz kernel of the identity
operator can be written as
\begin{equation}
\begin{aligned}
  K_{\Id} &=  \left(\prod_{i=1}^{\ell} \delta(x_{i}-x_{i}') \delta(y_{i}-y_{i}')dx_{i}dy_{i}' \right) \delta(z-z')dz' \\
              &=  \delta(x-x')\delta(y-y')\delta(z-z') dx' dy'dz',  
\end{aligned}  
\label{diff.10}\end{equation}
where in the second line we suppressed the subscripts to lighten the notation.  Using the coordinates \eqref{diff.5} on $X^{2}_{b}$, this Schwartz
kernel becomes
\begin{equation}
K_{\Id}= \delta(s) \delta(y-y') \delta(z-z') \frac{dx' dy' z'}{w_{1}'}, \quad w_{1}'= \prod_{j=1}^{\ell}x_{j}'.
\label{diff.11}\end{equation}
Finally, using the local coordinates \eqref{diff.7} on the $\pi$-double space $X^{2}_{\pi}$, this becomes
\begin{equation}
\begin{aligned}
   K_{\Id}&= \delta(S)\delta(Y)\delta(z-z')  \frac{dx' dy'dz'}{  \prod_{i=1}^{\ell} (x_{i}')(w_{i}')^{k_{i}+1}}, \quad w_{i}'= \prod_{j=i}^{\ell}
x_{j}',  \\
   &=\delta(S)\delta(Y)\delta(z-z')  \frac{dx' dy'dz'}{  \prod_{i=1}^{\ell} (x_{i}')^{2+\dim S_{i}}},  \\
       &= \delta(S)\delta(Y)\delta(z-z') \pi_{R}^{*}(\nu_{\pi}),
 \end{aligned}  
\label{diff.12}\end{equation}
where $\nu_{\pi}\in \CI(X;{}^{\pi}\Omega)$ is a non-vanishing $\fS$-density on $X$.  Let
\[
\cD^{0}(\Delta_{\pi})= \CI(X^{2}_{\pi}) \cdot \mu
\] 
be the space of smooth $\delta$-functions on the $p$-submanifold $\Delta_{\pi}\subset X^{2}_{\pi}$,  where $\mu$ is a non-vanishing delta function
with smooth coefficient as in \eqref{diff.12}.  From the local computation \eqref{diff.12},  we see that 
\begin{equation}
  K_{\Id}\in \cD^{0}(\Delta_{\pi})\cdot \pi_{R}^{*}(\nu_{\pi}).
\label{diff.13}\end{equation}
Thus, if $P\in \Diff^{k}_{\fS}(X)$ is a $\fS$-differential operator of order $k$, we see from Lemma~\ref{diff.2} that 
\begin{equation}
   K_{P}= \pi_{L}^{*}P\cdot K_{\Id} \in \cD^{k}(\Delta_{\pi}) \cdot \pi_{R}^{*}(\nu_{\pi}),
\label{diff.14}\end{equation}
where $D^{k}(\Delta_{\pi})$ is the space of delta functions of order at most $k$, namely
\begin{equation}
   \cD^{k}(\Delta_{\pi})= \Diff^{k}(X^{2}_{\fS})\cdot \cD^{0}(\Delta_{\pi}).  
\label{diff.15}\end{equation}
In fact, since $\pi_{L}^{*}(\cV_{\fS}(X))$ is transversal to $\Delta_{\pi}$ by Lemma~\ref{diff.2}, the space of Schwartz kernels of $\fS$-differential
operators of order $k$ is precisely given by
\begin{equation}
    \cD^{k}(\Delta_{\pi})\cdot \pi_{R}^{*}(\nu_{\pi}).  
\label{diff.16}\end{equation}
If $E$ and $F$ are smooth complex vector bundles on $X$ and $\Diff^{k}_{\fS}(X;E,F)$ is the space of $\fS$-differential operators of order $k$ acting
from $\CI(X;E)$ to $\CI(X;F)$, then working in local trivializations, we can in a similar way identify the corresponding space of Schwartz kernels
with
\begin{equation}
   \cD^{k}(\Delta_{\pi})\cdot \CI(X^{2}_{\pi}; \beta_{\pi}^{*}\Hom(E,F) \otimes \pi_{R}^{*}({}^{\pi}\Omega)),
\label{diff.17}\end{equation}
where $\Hom(E,F)= \pr_{L}^{*}(F)\otimes \pr_{R}^{*}(E^{*})$.

Since delta functions are a special type of conormal distributions, this suggests to define
$\fS$-pseudodifferential operators of order $k$ acting from sections of $E$ to sections of $F$ by
\begin{multline}
\Psi^{m}_{\fS}(X;E,F):= \{ K \in I^{m}(X^{2}_{\pi}; \Delta_{\pi};  
     \beta_{\pi}^{*}(\Hom(E;F))\otimes \pi_{R}^{*}{}^{\pi}\Omega); \\ \; K\equiv 0  \; \mbox{at} \; 
          \pa X^{2}_{\pi}\setminus \ff_{\pi} \}.  
\label{ds.8}\end{multline}
Here, $\ff_{\pi}:= \cup_{i=1}^{k} \ff_{\pi_{i}}$ and  $I^{m}(X^{2}_{\pi}; \Delta_{\pi};  
     \beta_{\pi}^{*}(\Hom(E,F))\otimes \pi_{R}^{*}{}^{\pi}\Omega)$ is the space of conormal distributions
     of order $m$ at $\Delta_{\pi}$.  The notation $K\equiv 0$ at $\pa X^{2}_{\pi}\setminus \ff(X^{2}_{\pi})$ 
     means that its Taylor series identically vanishes at $     \pa X^{2}_{\pi}\setminus \ff(X^{2}_{\pi})$.  We can similarly define the space of
polyhomogeneous (or classical) $\fS$-pseudodifferential operators of order $m$ by
\begin{multline}
\Psi^{m}_{\fS-\phg}(X;E,F):= \{ K \in I^{m}_{\phg}(X^{2}_{\pi}; \Delta_{\pi};  
     \beta_{\pi}^{*}(\Hom(E;F))\otimes \pi_{R}^{*}({}^{\pi}\Omega)); \\ \; K\equiv 0  \; \mbox{at} \; 
          \pa X^{2}_{\pi}\setminus \ff_{\pi} \},  
\label{ds.8b}\end{multline}     
where     $I^{m}_{\phg}(X^{2}_{\pi}; \Delta_{\pi};  
     \beta_{\pi}^{*}(\Hom(E,F))\otimes \pi_{R}^{*}{}^{\pi}\Omega)$ is the space of polyhomogeneous conormal distributions
     of order $m$ at $\Delta_{\pi}$.   
     
       With these definitions, notice that there are natural inclusions
\begin{equation}
  \Diff^{k}_{\fS}(X;E,F)\subset \Psi^{k}_{\fS-\phg}(X;E,F)\subset \Psi^{k}_{\fS}(X;E,F).\label{diff.18}\end{equation}

\section{Groupoids}\label{grpd}

We refer to \cite{Renault,Mackenzie} for the classical definitions and
constructions related to groupoids and their Lie algebroids. We recall here the basic
definitions needed for this paper.

A {\bf groupoid} is a small
category in which every morphism is an isomorphism. Let us make this notion more explicit. 
A groupoid $\cG$ is a pair $(\Gr{0},\Gr{1})$ of sets together  with
structural morphisms: the unit $u: \Gr{0} \to \Gr{1}$, the source and range $s,r: \Gr{1} \to \Gr{0}$, the inverse
$\iota: \Gr{1} \to \Gr{1}$, and the multiplication $\mu$
which is defined on the set $\Gr2$ of pairs $(\alpha, \beta) \in \Gr{1} \times \Gr{1}$ such that
$s(\alpha) = r(\beta)$. Here, the set $\Gr0$ denotes the set of objects (or
units) of the groupoid, whereas the set $\Gr1$ denotes the set of
morphisms of $\cG$. The identity morphism of any object of $\cG$ enables one to identify that object with a
morphism of $\cG$. This leads to the injective map $u : \Gr0 \to \cG$. Each morphism $g \in
\cG$ has a ``source'' and a ``range.'' The
inverse of a morphism $\alpha$ is denoted by $\alpha^{-1}=\iota(\alpha)$. The
structural maps satisfy the following properties,
\begin{enumerate}[(i)]
\item  $r(\alpha \beta)=r(\alpha)$ and $s(\alpha \beta)=s(\beta)$,
       for any pair $(\alpha, \beta)$ in $\Gr2$,
\item  $s(u(x))=r(u(x))=x$,  $u(r(\alpha))\alpha=\alpha$,
$\alpha u(s(\alpha))=\alpha$,
\item  $r(\alpha^{-1})=s(\alpha)$,\ $s(\alpha^{-1})=r(\alpha)$,\ $\alpha\alpha^{-1}=u(r(\alpha))$,\
and $\alpha^{-1}\alpha=u(s(\alpha))$,
\item  the partially defined multiplication $\mu$ is associative.
\end{enumerate}

We shall need groupoids with additional structures.

\begin{definition}
A {\bf Lie groupoid} (resp. {\bf locally compact groupoid}) is a groupoid
\begin{equation*}
        \cG=(\Gr0,\Gr1,s,r,\mu,u,\iota)
\end{equation*}
such that $\Gr0$ and $\Gr1$ are manifolds with corners (resp. locally compact spaces), the
structural maps $s,r,\mu,u,$ and $\iota$ are differentiable (resp. continuous), the
source map $s$ is a submersion (resp. surjective and open) and $\cG_x :=
s^{-1}(x)$, $x\in M$, are all Hausdorff manifolds without corners (resp. locally compact Hausdorff spaces).
\end{definition}
We will also encounter the notion of continuous family groupoid (\cite{Paterson}).
\begin{definition}\label{def.continuous.family}
A locally compact groupoid $\cG$ is a  {\bf continuous family groupoid} when it is covered by open sets $U$ with homeomorphisms
$\Phi_f=(f,\phi_f): U\to f(U)\times U_f$ where $f\in\{r,s\}$ and $U_f\subset\RR^n$ such that the following holds:
\begin{enumerate}
 \item for all $(U,\Phi_f)$ and $(V,\Psi_f)$ as above such that $W=U\cap V\not=\emptyset$, the map  $\Psi_f\circ\Phi_f^{-1}:\Phi_f(W)\longrightarrow \Psi_f(W)$ is of class $C^{0,\infty}$, which means that  $x\longmapsto\psi_f\circ\phi_f^{-1}(x,\cdot)$ is continuous from $f(W)$ to $C^\infty(\phi_f(W),\psi_f(W))$ (which has the topology of uniform convergence on compacta of all derivatives);
\item The inversion and product maps are locally $C^{0,\infty}$ in the above sense.  
\end{enumerate}
We say that $(U,\Phi_f)$ is a $C^{0,\infty}$ local chart for $(\cG,f)$. 
\end{definition}

A simple example of Lie groupoid is
the {\bf pair groupoid} associated to a
smooth manifold $M$. It is obtained by taking $\Gr0 = M$, $\Gr1 = M \times M$, $s(x, y)
= y$, $r(x, y) = x$, $(x, y)(y, z) = (x, z)$, $u(x)
= (x, x)$ and with inverse $\iota (x, y) = (y, x)$.

\smallskip 
Like vector bundles, groupoids can be pulled back.
More precisely, let $G\rightrightarrows M$ be a locally compact Hausdorff groupoid with source $s$ and range
$r$. If $f:N\rightarrow M$ is a surjective map, the {\bf pullback} groupoid $\pb{f}(G)\rightrightarrows N$ of $G$ by $f$ is by
definition the set
\begin{equation}
\pb{f}(G):=\{(x,\gamma,y)\in N\times G\times N \ \vert \
r(\gamma)=f(x),\ s(\gamma)=f(y)\}
\label{tire.arriere.groupoide}\end{equation}
with the structural morphisms given by
\begin{enumerate}
\item the unit map $x \mapsto (x,f(x),x)$,

\item the source map $(x,\gamma,y)\mapsto y$ and range map
$(x,\gamma,y) \mapsto x$,

\item the product $(x,\gamma,y)(y,\eta,z)=(x,\gamma \eta ,z)$ and
inverse $(x,\gamma,y)^{-1}=(y,\gamma^{-1},x)$.
\end{enumerate}

\smallskip \noindent The results of  \cite{MRW} apply to show that the groupoids $G$
and $\pb{f}(G)$ are Morita equivalent when $f$ is surjective and open.

\smallskip The infinitesimal object associated to a Lie groupoid is its Lie
algebroid, which we define next.

\begin{definition}
A  {\bf Lie algebroid} $\cA$ over a manifold $M$ is a vector bundle
$\cA \to M$, together with a Lie algebra structure on the space
$\Gamma(\cA)$ of smooth sections of $\cA$ and  a bundle map $\varrho:
\cA \rightarrow TM$ whose extension to sections of these bundles
satisfies

(i) $\varrho([X,Y])=[\varrho(X),\varrho(Y)]$, and

(ii) $[X, fY] = f[X,Y] + (\varrho(X) f)Y$,

\noindent for any smooth sections $X$ and $Y$ of $\cA$ and any
smooth function $f$ on $M$.
\end{definition}

\smallskip The map $\varrho$ is called the {\bf anchor map} of $\cA$. Note that
we allow the base $M$ in the definition above to be a manifold
with corners.

\smallskip 

Now, let $\cG = \Gr1 \overset{s}{\underset{r}\rightrightarrows}\Gr0$ be a Lie
groupoid. We denote by $T^s\cG$ the subbundle of $T\Gr1$ of $s$-vertical tangent
vectors. In other words, $T^s\cG$ is the kernel of the differential $Ts$ of
$s$.

\smallskip  \noindent
For any $\alpha$ in $\Gr1$, let $R_{\alpha}: \cG_{r(\alpha)} \rightarrow
\cG_{s(\alpha)}$ be the right multiplication by $\alpha$. A tangent vector
field $Z$ on $\Gr1$ is {\bf right invariant} if it  satisfies,

\begin{itemize}
\item[--] $Z$ is $s$-vertical, namely $Ts(Z)=0$.
\item[--] For all $(\alpha,\beta)$ in $\cG^{(2)}$, $Z(\alpha  \beta)=TR_{\beta}(Z(\alpha))$.
\end{itemize}

\medskip \noindent The Lie algebroid ${\cA}  \cG$ of a Lie groupoid $\cG$ is defined as follows \cite{Mackenzie},
\begin{itemize}
\item[--] The fibre bundle ${\cA}  \cG \rightarrow \cG^{(0)}$ is the restriction of
$T^s\cG$ to $\cG^{(0)}$. In other words, $\cA \cG=\cup_{x\in \cG^{(0)}} T_x
\cG_x$ is the union of the tangent spaces to the $s$-fibre at the corresponding unit.
\item[--] The anchor $\rho:{\cA}  \cG \rightarrow T\cG^{(0)}$
is the restriction of the differential $Tr$ of $r$
to ${\cA}  \cG$.
\item[--] If $Y:U \rightarrow {\cA}  \cG$ is a local section of ${\cA}  \cG$,
  where $U$ is an open subset of $\cG^{(0)}$, we define the local {\bf
    right invariant vector field} $Z_Y$ {\bf associated} with $Y$ by
  $$Z_Y(\alpha)=TR_{\alpha}(Y(r(\alpha))) \ \makebox{ for all } \
  \alpha \in \cG^U:= r^{-1}(U) \ .$$

\noindent The Lie bracket is then defined by
$$\begin{array}{cccc} [\ ,\ ]: & \Gamma({\cA}  \cG)\times \Gamma({\cA} 
  \cG) & \longrightarrow & \Gamma({\cA}  G) \\
 & (Y_1,Y_2) & \mapsto & [Z_{Y_1},Z_{Y_2}]_{\cG^{(0)}}
\end{array}$$
where $[Z_{Y_1},Z_{Y_2}]$ denotes the $s$-vertical vector field obtained
with the usual bracket and $[Z_{Y_1},Z_{Y_2}]_{\cG^{(0)}}$ is the
restriction of $[Z_{Y_1},Z_{Y_2}]$ to ${\cG^{(0)}}$.
\end{itemize}

\begin{remark}
When $\cG$ is a continuous family groupoid, the vector bundle  ${\cA}  \cG \rightarrow \cG^{(0)}$ as defined above still exists. Indeed, the fibres $\cG_x$, $x\in\Gr0$ are smooth manifolds and we still can set 
\begin{equation}
 \cA \cG = \bigsqcup_{x\in \Gr0} T_x\cG_x.
\end{equation}
This vector bundle is smooth in the sense of \cite{Paterson} and it is called the Lie algebroid of $\cG$ again. 
\end{remark}

In this paper, a central example of Lie algebroid  is given by ${}^{\pi}TX$ with anchor map given by the natural map
$\iota_{\pi}: {}^{\pi}TX\to TX$. In fact, the space $ \overset{\circ}{X^{2}_{\pi}} \cup
   \overset{\circ}{\ff}_{\pi}$ has a natural structure of Lie groupoid with Lie algebroid naturally identified with ${}^{\pi}TX$ under the
   identification $X\cong \Delta_{\pi}$.  More precisely, we set
 \begin{equation}
        \mathcal{G}_{\pi}^{(0)}= \Delta_{\pi}, \quad \mathcal{G}^{(1)}_{\pi}=  \overset{\circ}{X^{2}_{\pi}} \cup
   \overset{\circ}{\ff}_{\pi}.
 \label{lg.1}\end{equation}    
For $\alpha\in \mathcal{G}^{(1)}_{\pi}$ with $\beta_{\pi}(\alpha)= (x_{1},x_{2})$, we define
the source and range of $\alpha$ by 
\begin{equation}
    s(\alpha)= x_{2}, \quad r(\alpha)= x_{1}.  
\label{lg.2}\end{equation}     
The map 
\begin{equation}
 \begin{array}{lccc}
    \iota: & \overset{\circ}{X}\times \overset{\circ}{X} & \to &  \overset{\circ}{X}\times \overset{\circ}{X} \\
              &   (x,x') & \mapsto & (x',x)    
   \end{array}
\label{lg.3}\end{equation}     
extends in a unique way to a smooth map $\iota: \mathcal{G}_{\pi}^{(1)}\to \mathcal{G}_{\pi}^{(1)}$
defining on $\mathcal{G}^{(1)}_{\pi}$ the inverse map.  Similarly, the natural multiplication
map on the groupoid $\overset{\circ}{X}\times \overset{\circ}{X} $ extends to give a composition
map 
\begin{equation}
  \mu : \mathcal{G}^{(2)}_{\pi}\to \mathcal{G}^{(1)}_{\pi}
\label{lg.4}\end{equation}
where 
\begin{equation}
\mathcal{G}^{(2)}_{\pi}= \{ (\alpha,\beta)\in \mathcal{G}^{(1)}_{\pi}\times \mathcal{G}^{(1)}_{\pi}; \; r(\beta)=s(\alpha) \}.
\label{lg.5}\end{equation}
To see that the Lie algebroid of $\cG_{\pi}$ is precisely ${}^{\pi}TX$, it suffices to use Corollary~\ref{ds.9} and to notice that $\cA\cG_{\pi}$ is isomorphic to $N\Delta_{\pi}$, a fact that follows from the observation that the source map of $\cG_{\pi}$ is a surjective submersion equal to the identity map when restricted to units.  The Lie groupoid $\cG_{\pi}$ admits a decomposition into simpler groupoids.  Indeed, for each boundary hypersurface $H_i$ of $X$, notice that the subgroupoid
\[
           (\ff_{\pi_i}\cap\interior{\ff}_{\pi}) \setminus \left ( \bigcup_{H_i<H_j} (\ff_{\pi_j}\cap \ff_{\pi_i} \cap \interior{\ff}_{\pi}) \right)
\]
is naturally isomorphic to the pull-back groupoid  $\pb{\pi_i}({}^{\pi}NS_i)$.  Since these subgroupoids give a partition of $\interior{\ff}_{\pi}$,  this means that, forgetting about the Lie structure, the groupoid $\cG_{\pi}$ can be written as a disjoint union of groupoids,   
\[
        \cG_{\pi} \cong (\interior{X}\times \interior{X}) \bigsqcup\left(\bigsqcup_{i=1}^k \pb{\pi_i}({}^{\pi}NS_i) \right).
\]

Since $\mathcal{G}_{\pi}$ is a smooth groupoid, one can consider the space $\Psi(\mathcal{G}_{\pi})$ of  pseudodifferential
operators on $\mathcal{G}_{\pi}$ \cite{Monthubert-Pierrot,NWX,ALN}. Conditions on the supports of these operators are necessary to perform their products.  They corresponds to those pseudodifferential operators of \eqref{ds.8} whose
Schwartz
kernels are compactly supported in $\mathcal{G}^{(1)}_{\pi}$.  
Thus, the space of operators in \eqref{ds.8} is not much larger.  It replaces the condition of compact support in $\mathcal{G}^{(1)}_{\pi}$ by a
condition of Schwartz decay at `infinity' with the infinity and 
the decay condition specified by the natural compactification $X^{2}_{\pi}$ of $\mathcal{G}^{(1)}_{\pi}$. 
Still, it has an important consequence.  In the larger space \eqref{ds.8}, the inverse of an invertible pseudodifferential is automatically contained
in the
same space  (with minus the order of the original operator), a fact which is not true for the smaller space $\Psi(\mathcal{G}_{\pi})$.
This is intimately related with the well-known fact that the Fourier transform is an automorphism of the space  of Schwartz functions
on $\bbR^{n}$, but does not preserve the space smooth functions with compact support.

For $\cG$ a locally compact groupoid, we need a Haar system to define an associated $C^*$-algebra.  Recall first that a (right continuous) Haar system \cite{Renault} on $\cG$ is a family $(\lambda_x)_{x\in \Gr0}$ of (positive, regular, Borel) measures $\lambda_x$ on $\cG$ such that 
\begin{enumerate}
 \item the support of $\lambda_x$ is equal to (the whole of) $\cG_x$,
 \item for all $f\in \cC_c(\cG)$, the map $x\mapsto \int_{\cG_x} f(\gamma)d\lambda_x$ is continuous,
 \item for all $\gamma\in\cG$ and $f\in \cC_c(\cG)$, on has  $\int_{\cG_{r(\gamma)}} f(\gamma'\gamma)d\lambda_{r(\gamma)}=\int_{\cG_{s(\gamma)}} f(\gamma')d\lambda_{s(\gamma)}$.
\end{enumerate}
Haar systems always exist on continuous family groupoids.  Given a continuous Haar system, the space $C_c(\cG)$ is endowed with natural product and involution, and this involutive algebra is then  automatically represented in an appropriate Hilbert space as well. This leads to the notions of reduced $C^*$-algebra of $\cG$, usually denoted by $\cC^*_r(\cG)$, and of universal (maximal) $C^*$-algebra of $\cG$, usually denoted by $\cC^*(\cG)$. Different Haar systems provide up to isomorphism the same reduced and universal $C^*$-algebras\cite{Renault,Paterson}.  Moreover, if $E$ is a Hermitian vector bundle on $\cG^{(0)}$, then the space of sections $\cC_c(\cG,r^*E)$ has also natural
Hilbert $\cC^*_r(\cG)$ and $\cC^*(\cG)$-modules completions denoted by  $\cC^*_r(\cG,E)$ and $\cC^*(\cG,E)$.   
Often, the reduced and universal completions coincide, in which case we omit the subscript $r$ in $\cC^*_r$.  For the groupoids considered in this paper, this can be seen using the following general criterion.  
\begin{lemma}
If a measured groupoid $\cG$ is the finite disjoint union of
measurewise amenable (see \cite[Definition
3.3.1]{Anantharaman-Renault})   groupoids $\cG_i$, that is,  $\cG=\sqcup_{i\in I} \cG_i$
and $\cG^{(0)}=\sqcup_{i\in I} \cG^{(0)}_i$, where everything is
assumed to be borelian, then $\cG$   is measurewise amenable.  In particular,
$\cC^*(\cG)$ is nuclear and equal to $\cC^*_r(\cG)$.  
\label{ma.1}\end{lemma}    
\begin{proof}
The fact $\cG$ is measurewise amenable follows from 
\cite[Proposition 5.3.4]{Anantharaman-Renault} applied to the Borel map $q :
\cG^{(0)}\to I$ defined by $q(x)=i$ if $x\in \cG^{(0)}_i$.  By \cite[6.2.14]{Anantharaman-Renault}, we then have that $\cC^*(\cG)$ is nuclear and equal to $\cC^*_r(\cG)$.
\end{proof}
For instance, this criterion can be applied to the groupoid $\cG_{\pi}$.  
\begin{lemma}
The groupoid $\cG_{\pi}$ is measurewise amenable.  In particular, $\cC^*(\cG_{\pi})$ is nuclear and equal to $\cC^*_r(\cG_{\pi})$.  
\label{ma.2}\end{lemma}
\begin{proof}
By Lemma~\ref{ma.1}, it suffices to observe that $\cG_{\pi}$ can be written as a disjoint union of
topologically  amenable (and thus measurewise amenable, by
\cite{Anantharaman-Renault}) groupoids,
\begin{equation}
  \cG_{\pi} = (\overset{\circ}{X}\times \overset{\circ}{X}) \bigsqcup \left( \bigsqcup_{i=1}^k
(H_i\underset{\pi_i}{\times}{}^{\pi}TS_i\underset{\pi_i}{\times}H_i)|_{G_i}\times\bbR \right),
\end{equation}
where $G_i=H_i\setminus(\cup_{j>i} H_j)$.
The topological amenability  of the various subgroupoids on the right-hand side can be justified as follows,
\begin{itemize}
 \item[(i)]  A vector bundle is topologically amenable as a bundle of
abelian groups;
 \item[(ii)] Topological amenability is preserved under equivalence of
groupoids (\cite{Anantharaman-Renault}). For instance, given a vector
bundle $E\to S$ and a locally trivial fibre bundle $p:H\to S$, the
groupoid $(H\underset{p}{\times}E\underset{p}{\times}H)\rightrightarrows
H$ is equivalent as a groupoid to the vector bundle $E$, and thus is
topologically amenable;
 \item[(iii)]  The cartesian product of amenable groupoids is amenable.
\end{itemize}
\end{proof}

\section{Action of $\fS$-pseudodifferential operators} \label{asp.0}

Let us first consider the space $\Psi^{m}_{\fS}(X;E,F)$ in \eqref{ds.8} in the simpler situation where $E=F=\underline{\bbC}$.  Notice that
\eqref{ds.8} can alternatively be rewritten as
\begin{equation}
   \Psi^{m}_{\fS}(X)=I^{m}(X^{2}_{\pi}; \Delta_{\pi})\cdot {\cC}^{\infty}_{\ff_{\pi}}(X^{2}_{\pi}; \pi_{R}^{*}({}^{\pi}\Omega)),
\label{asp.1}\end{equation}
where ${\cC}^{\infty}_{\ff_{\pi}}(X^{2}_{\pi}; \pi_{R}^{*}({}^{\pi}\Omega))$ is the space of smooth sections vanishing with all their derivatives at
all boundary faces except those contained in $\ff_{\pi}$.  To describe the action of $\fS$-operators on functions, we will need the following result
about the pushforward of conormal distributions.

\begin{lemma}
 The map $\pi_{L}= \pr_{L}\circ \beta_{\pi}: X^{2}_{\pi}\to X$ induces a continuous linear map
 \[
           (\pi_{L})_{*}: I^{m}(X^{2}_{\pi};\Delta_{\pi})\cdot{\cC}^{\infty}_{\ff_{\pi}}(X^{2}_{\pi};\Omega)\to \CI(X;\Omega).
 \]  
\label{asp.3}\end{lemma}  
\begin{proof}
 If $K\in I^{m}(X^{2}_{\pi};\Delta_{\pi})\cdot {\cC}^{\infty}_{\ff_{\pi}}(X^{2}_{\pi};\Omega)$ is supported near the lifted diagonal, then the result
follows from general properties of conormal distributions together with the fact the map $\pi_{L}$ is transversal to $\Delta_{\pi}$.  Thus, using a
cut-off function, we can assume $K\in {\cC}^{\infty}_{\ff_{\pi}}(X^{2}_{\pi};\Omega)$.  To proceed further, notice that $\pi_{L}$ is a $b$-fibration
(we refer to \cite{MelroseMWC} for a definition).  Indeed, as a blow-down map, $\beta_{\pi}$ is a surjective $b$-submersion.  Since the projection
$\pr_{L}: X^{2}\to X$ is also clearly a surjective $b$-submersion, so is the composite
$\pi_{L}= \pr_{L}\circ \beta_{\pi}$.  Thus, according to Proposition~2.4.2 in \cite{MelroseMWC}, $\pi_{L}$ is a $b$-fibration provided no boundary
hypersurface of $X^{2}_{\pi}$ is mapped to a boundary face of $X$ of codimension greater than one.  This is clear for the `old' hypersurfaces in
$X^{2}_{\pi}$, while the `new' hypersurfaces are mapped  under $\beta_{\pi}$ to boundary faces of $X^{2}$ of codimension $2$ which are then mapped
under $\pr_{L}$ to boundary faces of codimension 1 under the projection $\pr_{L}$.  

The lemma can then be seen as a special case of the Push-forward Theorem of \cite{MelrosePFT}
for $b$-fibrations.  Precisely, the lemma is a consequence of this theorem combined with the fact 
\[
      \pi_{L}^{-1}(H_{i})\cap \ff_{\pi}= \ff_{\pi_{i}}
\]   
for all boundary hypersurfaces $H_{i}\subset X$.  
\end{proof}

Since the previous lemma is dealing with smooth densities, it cannot be applied directly to the space of conormal distributions $\Psi^{m}_{\fS}(X)$.

\begin{lemma}
The tensor product identification $\pr_{L}^{*}\Omega \otimes \pr_{R}^{*}{}^{\pi}\Omega\equiv \Omega$ on the interior of $X^{2}$ extends to give an
isomorphism of spaces of sections
\[
       {\cC}^{\infty}_{\ff_{\pi}}(X^{2}_{\pi}; \beta_{\pi}^{*}( \pr_{L}^{*}\Omega \otimes \pr_{R}^{*}{}^{\pi}\Omega))=
{\cC}^{\infty}_{\ff_{\pi}}(X^{2}_{\pi};\Omega)
\]
\label{asp.2}\end{lemma}
\begin{proof}
It suffices to notice that the singular factors of sections of $\beta_{\pi}^{*}( \pr_{L}^{*}\Omega \otimes \pr_{R}^{*}{}^{\pi}\Omega)$   all arise at
faces not contained in $\ff_{\pi}$, and so are absorbed by the infinite order vanishing at these faces.  This can be seen using the local
coordinates.  Indeed, in the coordinates \eqref{diff.4b},  an element of $\CI(X^{2}_{\pi}; \pr_{L}^{*}\Omega\otimes \pr_{R}^{*}({}^{\pi}\Omega)$ is of the
form
\[
   \frac{h dx dy dz dx' dy' dz'}{\prod_{i=1}^{\ell} (x_{i}')^{2+\dim S_{i}}} =\frac{h dx dy dz dx' dy' dz'}{\prod_{i=1}^{\ell} x_{i}'
(w_{i}')^{k_{i}+1} }, \quad \mbox{for some} \; h\in \CI(X^2_{\pi}).
    \]
 Thus, in the coordinates of \eqref{diff.7}, it takes the form 
 \[
      \widetilde{h} dS dY dz dx' dy' dz' \quad \mbox{for some} \; \widetilde{h}\in \CI(X^{2}_{\pi}),
 \]   
 and the only possible singular terms occur when $S_{i}\to \infty$ or $Y_{i}\to \infty$, that is, at faces not contained in $\ff_{\pi}$.  
\end{proof}

We can then define a push-forward map
\begin{equation}
   (\pi_{L})_{*}: I^{m}(X^{2}_{\pi};\Delta_{\pi})\cdot {\cC}^{\infty}_{\ff_{\pi}}(X^{2}_{\pi};\pi_{R}^{*}{}^{\pi}\Omega)\to \CI(X)
\label{asp.4}\end{equation}
by requiring  that for $K\in I^{m}(X^{2}_{\pi};\Delta_{\pi})\cdot {\cC}^{\infty}_{\ff_{\pi}}(X^{2}_{\pi};\pi_{R}^{*}{}^{\pi}\Omega)$ and  any
non-vanishing section $v\in \CI(X;\Omega)$,\begin{equation}
     v\cdot (\pi_{L})_{*}K= (\pi_{L})_{*} ( \pi_{L}^{*}v\cdot K),
\label{asp.5}\end{equation}
where the right hand side of \eqref{asp.5} is in $\CI(X;\Omega)$ thanks to Lemma~\ref{asp.3} and Lemma~\ref{asp.2}.    This push-forward map provides
a way to make $\fS$-pseudodifferential operators act on functions.  To state the main result of this section, we still need to introduce some
notation.  If $M_{1}X$ is the set of boundary hypersurfaces and $A\subset M_{1}X$ is a subset, then set 
\[
     x_{A}= \prod_{H\in A} x_{H} 
\]
where $x_{H}\in \CI(X)$ is a choice of boundary defining function for $H$.  For any $A\subset M_{1}X$, consider the space
 \[
         \dot{\cC}^{\infty}_{A}(X;E) = \bigcap_{k\in\bbN} x^{k}_{A} \CI(X;E)
 \]
 of smooth sections on $X$ vanishing with all their derivatives on each boundary hypersurface $H\in A$. 
 When $A=M_{1}X$, this gives the space 
 \[
      \dot{\cC}^{\infty}(X;E)= \dot{\cC}^{\infty}_{M_{1}X}(X;E)
 \]  
of smooth sections  vanishing with all their derivatives on $\pa X$.  It is also useful to use the notation $\CI_{A}(X;E)=
\dot{\cC}^{\infty}_{M_{1}X\setminus A}(X;E)$.  Thus, for $A=M_{1}X$, we have $\CI_{M_{1}X}(X;E)=\CI(X;E)$.   

Each space $\dot{\cC}^{\infty}_{A}(X;E)$ comes with a natural structure of Fr\'echet space induced from the one of $\CI(X;E)$.  The corresponding
space of distributions $\dot{\cC}^{-\infty}_{A}(X;E)$ is defined to be the dual of  ${\cC}^{\infty}_{A}(X;E^{*}\otimes \Omega)$.  Similarly, we use
the notation $\cC^{-\infty}_{A}(X;E)$ to denote the dual of $\dot{\cC}^{\infty}_{A}(X;E^{*}\otimes \Omega)$.

\begin{proposition}
Via the push-forward map \eqref{asp.4}, an element $P\in \Psi^{m}_{\fS}(X;E,F)$ defines a continuous linear map
\[
       P: \CI(X;E) \to \CI(X;F).
\] 
For each subset $A\subset M_{1}X$,  this map restricts to give a continuous linear map
\[
  P: \dot{\cC}^{\infty}_{A}(X;E) \to \dot{\cC}_{A}^{\infty}(X;F).
\]
These maps extend by continuity in the distributional topology to  linear maps
\[
P: \dot{\cC}^{-\infty}_{A}(X;E) \to \dot{\cC}_{A}^{-\infty}(X;F)
\]
for all subsets $A\subset M_{1}X$.

\label{asp.6}\end{proposition}
\begin{proof}
The first assertion is a consequence of Lemma~\ref{asp.3} and Lemma~\ref{asp.2}.  
Using a partition of unity subordinate to a covering by open sets over which $E$ and $F$ restrict to be trivial, we can reduce to the case
$E=F=\underline{\bbC}$ to prove the second assertion.  

Let $A\subset M_{1}X$ be given.  Since the function
$(\frac{x_{A}}{x_{A}'})\in \CI(X^{2}\setminus \pa X^{2})$ pulls back to $X^{2}_{\pi}$ to give a function which is smooth on $\ff_{\pi}$ and has only
finite order singularities at hypersurfaces not in $\ff_{\pi}$, we see that
\[
   P \in \Psi^{m}_{\fS}(X;E,F)\; \Longrightarrow \; \widetilde{P}_{k}= x_{A}^{k}\circ P\circ x_{A}^{-k}
   \in \Psi^{m}_{\fS}(X;E,F)
\] 
for all $k\in \bbN$.  On the other hand,  given $u\in \dot{\cC}^{\infty}_{ A}(X;E)$, we can write it as
$u= x^{k}_{A} \widetilde{u}_{k}$ for some $\widetilde{u}_{k}\in\CI(X;E)$, so that
\[
      x^{-k}_{A}Pu= \widetilde{P}_{k} \widetilde{u}_{k} \in \CI(X;F)  \; \Longrightarrow 
      \; Pu\in x^{k}_{A}\CI(X;F).
\] 
Since $k\in \bbN$ is arbitrary, this means $Pu\in \dot{\cC}^{\infty}_{ A}(X;F)$.  

For the proof of the last assertion, choose a non-vanishing density in $\CI(X;{}^{\pi}\Omega)$ as well as Hermitian metrics for $E$ and $F$.  These
then define a $L^{2}$-inner product for sections of $E$ and $F$.  
To see the action of $P\in \Psi^{m}(X;E,F)$ extends to distributions, it suffices to notice that from \eqref{ds.8}, the formal adjoint of $P\in
\Psi^{m}(X;E,F)$ with respect to this $L^{2}$-inner product is an element of $\Psi^{m}_{\fS}(X;F,E)$, so that the action of $P$ on distributions can
be defined by duality.  

\end{proof}

The following proposition can be interpreted as a dual statement to the Schwartz kernel theorem.  
\begin{proposition}
A continuous linear operator $A: \dot{\cC}^{\infty}(X)\to \cC^{-\infty}(X)$ induces a continuous linear map $A: \cC^{-\infty}(X)\to
\dot{\cC}^{\infty}(X)$  if and only if its Schwartz kernel $K_{A}$ is an element of $\dot{\cC}^{\infty}(X\times X; \pr_{R}^{*}\Omega X)$ where 
$\Omega X$ is the density bundle on $X$ and $\pr_{R}: X\times X\to X$ is the projection on the right factor.   
\label{ideal.1}\end{proposition}
\begin{proof}
One proceeds as in the proof of Proposizione~1.2 in \cite{Parenti}.  Namely, it suffices to notice that if $\mathcal{L}( \cC^{-\infty}(X),
\dot{\cC}^{\infty}(X))$ denotes the space of continuous linear maps (with $\cC^{-\infty}(X)$ equipped with the strong dual topology), then (see
\cite{Treves})
\begin{equation*}
\begin{aligned}
    \mathcal{L}(\cC^{-\infty}(X),\dot{\cC}^{\infty}(X)) &\cong \dot{\cC}^{\infty}(X;\Omega X)\,  \widehat{\otimes} \, \dot{\cC}^{\infty}(X)   \\
    & \cong  \dot{\cC}^{\infty}(X\times X; \pr_{R}^{*}\Omega X).
 \end{aligned}   
\end{equation*}
\end{proof}
Let us denote by $\dot{\Psi}^{-\infty}_{\fS}(X)$ the space of operators with Schwartz kernel in $\dot{\cC}^{\infty}(X\times X; \pr_{R}^{*}\Omega X)$. 
From the definition of $\fS$-operators, it is clear that we have the identification $x^{\infty}\Psi^{-\infty}_{\fS}(X)= \dot{\Psi}^{-\infty}_{\fS}(X)$
where $x= \prod_{i=1}^{k} x_{i}$.  From Proposition~\ref{ideal.1}, we immediately obtain  the following.
\begin{corollary}
 For $A\in \dot{\Psi}^{-\infty}_{\fS}(X)$ and $B\in \Psi^{m}_{\fS}(X)$, we have
 \[
       AB\in \dot{\Psi}^{-\infty}_{\fS}(X),   \quad BA\in \dot{\Psi}^{-\infty}_{\fS}(X).  
 \]
\label{ideal.2}\end{corollary}

\section{Suspended $\fS$-operators} \label{sus.0}

Before describing the symbol maps associated to $\fS$-operators, we first need to discuss how to suspend them in the sense of \cite{Melrose_eta}.  To
this end, let $(X,\pi)$ be a manifold with fibred corners and let $H_{1},\ldots, H_{k}$ be its boundary hypersurfaces with corresponding boundary
defining functions $x_{1},\ldots, x_{k}$.  Let $V$ be a Euclidean vector space, that is, a finite dimensional real vector space with inner product
$\langle\cdot,\cdot\rangle_{V}$.  Consider on $V$ the function
\[
    \rho_{V}(v)= (1+ \langle v,v\rangle_{V})^{-\frac12}, \quad v\in V.
\]   
Let $\overline{V}$ be the radial compactification of $V$ as defined in \cite{MelroseGST}, so that $\rho_{V}$ extends to be a boundary defining
function for $\pa\overline{V}\subset \overline{V}$.  We can regard $\overline{V}$ as a manifold with fibred corners, the fibration on the boundary
being given by the identity map $\Id:\pa \overline{V}\to \pa\overline{V}$.   We can get a new manifold with fibred corners $(\overline{V}\times X,
\varpi)$ by taking the Cartesian product of $\overline{V}$ and $X$.  The iterated fibration structure $\varpi$ of $\overline{V}\times X$ is naturally
induced from those of $\overline{V}$ and $X$ as follows.  The fibration $\varpi_{0}$ on  the boundary hypersurface $Z_{0}= \pa\overline{V}\times X$ is
given by the projection on $\pa \overline{V}$, while the fibration of the boundary hypersurface $Z_{i}= \overline{V}\times H_{i}$ is given by
$\varpi_{i}=\Id\times \pi_{i}$.  The partial order on the boundary hypersurfaces of $\overline{V}\times X$ is specified by requiring that for all
$i,j\in \{1,\ldots,k\}$,  
\begin{equation}
    Z_{0}< Z_{i},  \quad   Z_{i}< Z_{j}\;  \Longleftrightarrow   H_{i} < H_{j}.
\label{sus.1}\end{equation}
Finally, the boundary defining function of $Z_{i}= \overline{V}\times H_{i}$ is taken to be the pullback of $x_{i}$ to $\overline{V}\times X$, while
we choose the boundary defining function $x_{0}$ of $Z_{0}$ to be the pullback of $\rho_V$ to $\overline{V}\times X$.   

Let $E$ and $F$ be smooth complex vector bundles on $\overline{V}\times X$ obtained by pulling back complex vector bundles on $X$ to
$\overline{V}\times X$.  Consider then the space 
$\Psi^{m}_{\fS}(\overline{V}\times X;E,F)$ of $\fS$-operators of order $m$ acting from sections of $E$ to sections of $F$.  From the previous section,
we know that an operator $P\in \Psi^{m}_{\fS}(\overline{V}\times X;E,F)$ induces a continuous linear map
\begin{equation}
    P: \cS(V\times X;E) \to \cS(V\times X; F)
\label{sus.2}\end{equation}
where $\cS(V\times X;E) = \dot{\cC}^{\infty}_{Z_{0}}(\overline{V}\times X;E)$ is the space of smooth sections of $E$ vanishing with all their
derivatives at the boundary hypersurface $Z_{0}= \pa \overline{V}\times X$,  and  similarly $\cS(V\times X;F) =
\dot{\cC}^{\infty}_{Z_{0}}(\overline{V}\times X;F)$.  Given $v\in V$, consider the diffeomorphism 
\begin{equation}
 \begin{array}{lrcl}
   T_{v}: & V\times X & \to & V\times X  \\
              & (w,p) & \mapsto & (w+v,p)
 \end{array}
\label{sus.3}\end{equation}
obtained by translating by $v$.  Since $E$ is the pullback of a vector bundle defined on $X$, we have a corresponding action 
\begin{equation}
 \begin{array}{lrcl}
   T_{v}^{*}: & \cS(V\times X;E) & \to & \cS(V\times X;E)  \\
              & \psi & \mapsto & \psi\circ T_{v}
 \end{array}
\label{sus.4}\end{equation}
For the same reason, we have an action $T_{v}^{*}: \cS(V\times X;F)\to \cS(V\times X;F)$.  
\begin{definition}
The space $\Psi^{m}_{\fS-\sus(V)}(X;E,F)$ of \textbf{ $V$-suspended $\fS$-operators of order $m$ on $X$} acting from sections of $E$ to sections of
$F$ is the subspace  of operators $P$ in 
$\Psi^{m}_{\fS}(\overline{V}\times X;E,F)$  such that for all $v\in V$, 
\[
      T_{-v}^{*} \circ P\circ T_{v}^{*}= P.  
\]
When $V=\bbR^{p}$, we use the notation $\Psi^{m}_{\fS-\sus(p)}(X;E,F)= \Psi^{m}_{\fS-\sus(\bbR^{p})}(X;E,F)$ and say the corresponding operators are
$p$-suspended.  
\label{sus.5}\end{definition}

In terms of the Schwartz kernel $K_{P}$ seen as a distribution on $V^{2}\times X^{2}$, the translation invariance in this definition means that for
all $v\in V$, 
\begin{equation}
     T_{(v,v)}^{*}K_{P}= K_{P}
\label{sus.6}\end{equation}
where $T_{(v,v)}$ is the diffeomorphism
\begin{equation}
   \begin{array}{lrcl}
   T_{v}: & V^{2}\times X^{2} & \to & V^{2}\times X^{2}  \\
              & (w,w',p,p') & \mapsto & (w+v, w'+v,p, p').
 \end{array}
 \label{sus.7}\end{equation}
If 
\begin{equation}
      \begin{array}{lrcl}
   a: & V^{2} & \to & V  \\
              & (v,v') & \mapsto & (v-v')
 \end{array}\label{sus.8}\end{equation}
denotes the projection onto the anti-diagonal of $V^{2}$, this means that $K_{P}$ is the pullback via the map $a\times \Id: V^{2}\times X^{2}\to
V\times X^{2}$ of a distribution on $V\times X^{2}$.  

To accurately describe this distribution, notice first that parallel transport with respect to the Euclidean metric on $V$ gives a canonical
identification of vector bundles $TV=V\times T_{0}V=V\times V$ extending naturally to a trivialization
\begin{equation}
    {}^{\Id}T\overline{V}\cong \overline{V}\times V.
\label{sus.9}\end{equation}  
Using this identification and Corollary~\ref{ds.9}, one can see that the linear isomorphism 
\begin{equation}
    \begin{array}{lrcl}
   L: & V\times V & \to & V\times V  \\
              & (v',w) & \mapsto & (v'+w,v')
 \end{array}
 \label{sus.10}\end{equation}
naturally extends to give an identification  ${}^{\Id}T\overline V \cong \cG^{(1)}_{\Id}(\overline{V})$ of non-compact manifolds with
boundary, where
\[
\cG^{(1)}_{\Id}(\overline{V})= (\overline{V}^{2}_{\Id}\setminus \pa \overline{V}^{2}_{\Id} ) \cup (\ff_{\Id}\setminus \pa \ff_{\Id})
\]
is the Lie groupoid associated to $\overline{V}^{2}_{\Id}$.  Since $a\circ L(v',w)=w$,  this means the map $a$ can be extended  to a map 
\begin{equation}
\overline{a}: \cG^{(1)}_{\Id}(\overline{V})\to V
\label{sus.11}\end{equation}
  by composing the identification $\cG_{\Id}^{(1)}(\overline{V})\cong {}^{\Id}T\overline{V}\cong \overline{V}\times V$ with the projection $\pr_{2}:
\overline{V}\times V\to V$ on the second factor.

On the other hand, the $\varpi$-double space is naturally given by 
\begin{equation}
    (\overline{V}\times X)^{2}_{\varpi} = \overline{V}^{2}_{\Id}\times X^{2}_{\pi},
\label{sus.12}\end{equation} 
where $\overline{V}^{2}_{\Id}$ is the $\Id$-double space of the manifold with fibred boundary $\overline{V}$.  Consider then the map  
\begin{equation}
     \alpha=   \overline{a}\times \Id: \cG^{(1)}_{\Id}(\overline{V})\times X^{2}_{\pi}\to V\times X^{2}_{\pi}.
\label{sus.13}\end{equation}
In terms of this map, the translation invariance condition in Definition~\ref{sus.5} means that as a distribution on $\cG^{(1)}_{\Id}(\overline{V})\times
X^{2}_{\pi}$, the Schwartz kernel of a $V$-suspended $\fS$-operator is the pullback of a distribution on $V\times X^{2}_{\pi}$.  More precisely, we
have obtained the following.

\begin{lemma}
The space of Schwartz kernels of $V$-suspended $\fS$-operators of order $m$ acting from sections of $E$ to sections of $F$ is given by 
\begin{multline*}
\Psi^{m}_{\fS-\sus(V)}(X;E,F)= \{  \alpha^{*} K\; ; \; K\in I^{m}(\overline{V}\times X^{2}_{\pi}; \{0\}\times \Delta_{\pi}; \cV ), \\
K\equiv 0 \; \mbox{at} \; (\overline{V}\times \pa X^{2}_{\pi}\setminus \ff_{\pi})\cup (\pa\overline{V}\times X^{2}_{\pi}) \},   
\end{multline*}
where $\cV= \pr_{2}^{*}(\beta_{\pi}^{*}\Hom(E,F)\otimes \pi_{R}^{*}{}^{\pi}\Omega)\otimes \pr_{1}^{*}{}^{\Id}\Omega $ with $\pr_{1}:
\overline{V}\times X^{2}_{\pi}\to \overline{V}$ and $\pr_{2}: \overline{V}\times X^{2}_{\pi}\to X^{2}_{\pi}$ the natural projections.   From that
perspective, the action of an operator $P\in \Psi^{m}_{\fS-\sus(V)}(X;E,F)$ on a section
$u\in \cS(V\times X;E)$ is given by
\begin{equation*}
    Pu=  (\varpi_{L})_{*} ( \alpha^{*}K_{P} \cdot \varpi^{*}_{R}u),  
\end{equation*}
where $\varpi_{L}$ and $\varpi_{R}$ are the analog of the maps \eqref{diff.1} for the manifold with fibred corners $\overline{V}\times X$.  \label{sus.14}\end{lemma}

Seen as a distribution on $V\times X^{2}_{\pi}$, it is possible to take the Fourier transform in the $V$-factor of the Schwartz kernel $K_{P}$ of a
$V$-suspended $\fS$-operator $P$,
\begin{equation}
   K_{\widehat{P}}(\Upsilon) = \int_{V} e^{-i\Upsilon\cdot v} \; K_{P}(v), \quad \Upsilon\in V^{*}.  
\label{sus.15}\end{equation}
We will call $\Upsilon\in V^{*}$ the \textbf{suspension parameter}. 
This gives for each $\Upsilon\in V^{*}$ the Schwartz kernel $K_{\widehat{P}}(\Upsilon)$ of 
 a $\fS$-operator $\widehat{P}(\Upsilon)\in \Psi^{m}_{\fS}(X;E,F)$.  Similarly, if $\nu$ denotes the translation invariant density on $V$ associated
to our choice of inner product 
$\langle\cdot,\cdot\rangle_{V}$, then we can define the Fourier transform 
\begin{equation}
   \cF_{E}: \cS(V\times X;E) \to \cS(V^{*}\times X;E)
\label{sus.16}\end{equation}
  by
  \begin{equation}
   \widehat{u}(\Upsilon)= \cF_{E}(u)(\Upsilon)= \int_{V} e^{-i\Upsilon\cdot v} \; u(v) \nu, \quad 
   \Upsilon\in V^{*},
  \label{sus.17}\end{equation}
with inverse Fourier transform given by
\begin{equation}
  u(v)= \cF^{-1}_{E}(\widehat{u})(v)= \frac{1}{(2\pi)^{\dim V}} \int_{V^{*}} e^{i \Upsilon\cdot v}\; 
  \widehat{u}(\Upsilon) \nu^{*},
\label{sus.18}\end{equation}
 where $\nu^{*}$ is the density on $V^{*}$ dual to $\nu$.  With these definitions, we have as expected that the action of 
 $P$ on $\cS(V\times X;E)$ can be described by
 \begin{equation}
   \widehat{Pu}(\Upsilon) = \widehat{P}(\Upsilon)\widehat{u}(\Upsilon), \quad \forall \Upsilon \in V^{*}.
 \label{sus.19}\end{equation}   
  In other words, the Fourier transform of $P$ is given by
  \begin{equation}
    \widehat{P}= \cF_{F}\circ P \circ \cF^{-1}_{E}.
  \label{sus.20}\end{equation}  
  If $Q\in \Psi^{m}_{\fS-\sus(V)}(X;G,E)$ is another $V$-suspended operator, where $G$ is a complex vector bundle on $\overline{V}\times X$ given by
the pullback of a complex vector bundle on $X$, then we have in particular that 
  \begin{equation}
    \widehat{P\circ Q}(\Upsilon)= \widehat{P}(\Upsilon) \circ \widehat{Q}(\Upsilon).
  \label{sus.21}\end{equation}  
 That is, under the Fourier transform, the convolution product in the $V$-factor becomes pointwise composition.  Since an operator $P$ can be
recovered from $\widehat{P}$ by taking the inverse Fourier transform, we see that $\widehat{P}$ completely describes the operator $P$.  It is important
however to notice that the Fourier transform of an operator $P\in \Psi^{m}_{\fS-\sus(V)}(X;E,F)$ is \textbf{not} an arbitrary smooth family of $\fS$-operators. 
For instance, as can be readily seen by taking the Fourier transform of $K_{P}$ in directions conormal to $\Delta_{\pi}\subset X^{2}_{\pi}$, me
must have that 
 \begin{equation}
   (D_{\Upsilon}^{\alpha} \widehat{P})(\Upsilon)\in \Psi^{m-|\alpha|}_{\fS}(X;E,F), \quad 
   \forall \alpha\in \bbN^{\dim V}_{0}, \; \forall \Upsilon\in V^{*}.
 \label{sus.22}\end{equation}      
 For operators of order $-\infty$, we can completely characterize the image of the Fourier transform.  It is given by smooth families of
$\fS$-operators
 \[
        V^{*} \ni \Upsilon \mapsto  \widehat{P}(\Upsilon)\in \Psi^{-\infty}_{\fS}(X;E,F)
 \]    
such that for any Fr\'echet semi-norm $\|\cdot \|$ of the space $\Psi^{-\infty}_{\fS}(X;E,F)$, we have
\begin{equation}
  \sup_{\Upsilon}  \|  \Upsilon^{\alpha}D^{\beta}_{\Upsilon} \widehat{P} \|< \infty, \quad \forall \alpha,\beta\in \bbN_{0}^{\dim V}.
\label{sus.23}\end{equation}    
For operators of order $m\in\bbR$, one has more generally that for any Fr\'echet semi-norm $\|\cdot\|$ of $\Psi^{m}_{\fS}(X;E,F)$, the Fourier
transform $\widehat{P}$ of a suspended operator $P\in \Psi^{m}_{\fS-\sus(V)}(X;E,F)$ must satisfy 
\begin{equation}
\sup_{\Upsilon}  \|  (1+|\Upsilon|^{2})^{\frac{|\alpha|-m}{2}} D^{\alpha}_{\Upsilon}\widehat{P}\| <\infty  \quad \forall \alpha\in \bbN^{\dim V}_{0}.
\label{sus.24}\end{equation}   
 In this latter case however,  these conditions are not sufficient to fully characterize the image of the Fourier transform.

The discussion above has a straightforward generalizations to families.  Namely, consider a fibration 
\begin{equation}
\xymatrix{  F  \ar@{-}[r] &   H  \ar[d]^{\phi} \\
                & S
}
\label{sus.25}\end{equation}
where $S$ is a manifold with corners and where the fibres are manifolds with fibred corners.  We suppose that the fibration is locally trivial in the
sense that for each $s\in S$, there is a neighborhood $\cU$ of $s$, a manifold with fibred corners $F$ and a diffeomorphism $\psi: \phi^{-1}(\cU)\to
\cU\times F$ inducing a commutative diagram
\begin{equation}
\xymatrix{  \phi^{-1}(\cU) \ar[rr]^{\psi} \ar[rd]^{\phi} &  &  \cU\times F \ar[ld]^{\pr_{\cU}} \\
                & \cU &
}
\label{sus.26}\end{equation}    
such  that for all $u\in \cU$, the restriction
\[
    \psi: \phi^{-1}(u)\to \{u\} \times F
\]
is a diffeomorphism of manifold with fibred corners.   
For such a fibration, we can consider the space of fibrewise $\fS$-operators of order $m$ 
$\Psi^{m}_{\fS}(H/S;E,F)$ where $E$ and $F$ are smooth complex vector bundles on $H$.  If moreover $V\to S$ is a smooth Euclidean vector bundle, that
is, a smooth real vector bundle equipped with a fibrewise inner product, we can then consider the space of fibrewise $V$-suspended $\fS$-operators 
$\Psi^{m}_{\fS-\sus(V)}(H/S;E,F)$.  Thus, an operator $P\in \Psi^{m}_{\fS-\sus(V)}(H/S;E,F)$ is a smooth family 
\[
   S\ni s \mapsto P_{s}\in \Psi^{m}_{\fS-\sus(V_{s})}(\phi^{-1}(s);E,F)
\]
of fibrewise $V$-suspended $\fS$-operators, where $V_{s}$ is the fibre of $V$ above $s\in S$.    

\section{Symbol Maps}\label{sm.0}

As for other calculi of pseudodifferential operators on singular spaces, various symbol maps can be defined.  The ordinary symbol map can be defined
in terms of the principal symbol map for conormal
distributions introduced by H\"ormander (see Theorem 18.2.11 in \cite{Hormander3}),
\begin{equation}
\xymatrix{
    I^{m}(Y,Z;\Omega_{Y}^{\frac{1}{2}}) \ar[r]^-{\sigma_{m}} & S^{[M]}(N^{*}Z; \Omega^{\frac{1}{2}}(N^{*}Z))
}    
\label{sm.1}\end{equation}
with $M= m-\frac{1}{4}\dim Y+ \frac{1}{2}\dim Z$ where $\phi:N^{*}Z\to Z$ is the natural projection and
\begin{equation}
  S^{[M]}(N^{*}Z)= S^{M}(N^{*}Z)/ S^{M-1}(N^{*}Z),
\label{sm.2}\end{equation}
where $S^{M}(N^{*}Z)$ is the usual space of functions $\psi\in \CI(N^{*}Z)$ such that in a local trivialization $\left. N^{*}Z\right|_{\cU} \cong
\cU\times \bbR^{n}_{\xi}$ with local variable $u$ in $\cU$, 
\begin{equation}
   \sup_{u,\xi} \frac{ | D^{\alpha}_{u} D^{\beta}_{\xi} \psi |}{ (1+|\xi|^{2})^{ \frac{M-|\beta|}{2}}} < \infty
    \quad \forall \; \alpha, \beta\in \bbN_{0}^{n}.
\label{sm.3}\end{equation}
In our case, $Y=X^{2}_{\pi}$ and $Z=\Delta_{\pi}$.  By Corollary~\ref{ds.9}, $N^{*}\Delta_{\pi}\cong {}^{\pi}T^{*}X$.  Since ${}^{\pi}\Omega_{R}$ is
naturally isomorphic to ${}^{\pi}\Omega_{L}^{\frac12}\otimes {}^{\pi}\Omega^{\frac{1}{2}}_{R}$ when restricted to the diagonal and since the singular
symplectic form of 
${}^{\pi}T^{*}X$ provides a natural trivialization of $\Omega({}^{\pi}T^{*}X)$, we get a map 
\begin{equation}
\xymatrix{
  \Psi^{m}_{\fS}(X;E,F) \ar[r]^-{\sigma_{m}} & S^{[m]}({}^{\pi}T^{*}X; \phi^{*}\Hom(E,F))
 } 
\label{sm.3}\end{equation}
inducing a short exact sequence
\begin{equation}
\xymatrix @C=1.5pc{
  0 \ar[r] & \Psi^{m-1}_{\fS}(X;E,F) \ar[r] &  \Psi^{m}_{\fS}(X;E,F) \ar[r]^-{\sigma_{m}} & S^{[m]}({}^{\pi}T^{*}X; \phi^{*}\Hom(E,F)) \ar[r] & 0.
 } 
\label{sm.4}\end{equation}
Here, $\phi: {}^{\pi}T^{*}X\to X$ is the bundle projection.  When we consider instead polyhomogeneous pseudodifferential operators of degree $m$, the
principal symbol is a homogeneous section of degree $m$ on
${}^{\pi}T^{*}X\setminus \{0\}$, so it defines a map
\begin{equation}
\xymatrix{   
\Psi^{m}_{\fS-\phg}(X;E,F) \ar[r]^-{\sigma_{m}}  & \CI( {}^{\pi}S^{*}X; \Lambda^{m}\otimes 
   \phi^{*}\Hom(E,F)) 
}
\label{sm.5}\end{equation}
where $\Lambda$ is the dual of the tautological real line bundle of ${}^{\pi}S^{*}X$.

\begin{definition}
An operator $P\in \Psi^{m}_{\fS}(X;E,F)$ is \textbf{elliptic} if its principal symbol $\sigma_{m}(P)$ is invertible.  
\label{ell.1}\end{definition}

To study the asymptotic behavior of $\fS$-operators at each boundary hypersurface, it is also useful to introduce other symbols, that is, normal
operators in the terminology of \cite{Mazzeo-Melrose}.   Those additional symbol maps are defined by restricting the Schwartz kernel of the operator to the various front faces,
\begin{equation}
   \sigma_{\pa_{i}}: \Psi^{m}_{\fS}(X;E,F)\to \Psi^{m}_{\ff_{\pi_{i}}}(H_{i};E,F)
\label{sm.6}\end{equation}
with 
\begin{multline}
 \Psi^{m}_{\ff_{\pi_{i}}}(H_{i};E,F)= \left\{ K\in I^{m}(\ff_{\pi_{i}},\Delta_{\ff_{\pi_{i}}}; \beta_{\pi}^{*}(\Hom(E,F))\otimes
\left.\pi_{R}^{*}{}^{\pi}\Omega\right|_{\ff_{\pi_{i}}}) ;   \right .\\
\left. K\equiv 0 \; \mbox{at} \; (\pa \ff_{\pi_{i}} \cap \pa\ff_{\pi}) \right\},
\label{sm.7}\end{multline}
where $\Delta_{\ff_{\pi_{i}}}= \ff_{\pi_{i}}\cap \Delta_{\pi}$.  
The symbol map $\sigma_{\pa_{i}}$ clearly induces a short exact sequence
\begin{equation}
\xymatrix{
0 \ar[r] &  x_{i}\Psi^{m}_{\fS}(X;E,F)\ar[r] & \Psi^{m}_{\fS}(X;E,F)
\ar[r]^-{\sigma_{\pa_{i}}}   & \Psi^{m}_{\ff_{\pi_{i}}}(H_{i};E,F)\ar[r] & 0,   }
\label{sm.7b}\end{equation}
where $x_{i}$ is the boundary defining function of $H_{i}$.
\begin{remark}
Since $\beta_{\pi}^*( \frac{x_i}{x_i'})$ is equal to $1$ on $\ff_{\pi_i}$, notice that for $z\in \bbC$, 
\[
    P\in \Psi^m_{\fS}(X;E,F)\; \Longrightarrow \; P_{i,z}:=x_i^z Px_i^{-z}\in \Psi^m_{\fS}(X;E,F)\quad \mbox{with} 
     \quad \sigma_{\pa_i}(P_{i,z})= \sigma_{\pa_i}(P).
\]
\label{notw.1}\end{remark}

Clearly, the space $\mathcal{G}^{(1)}_{\ff_{\pi_{i}}}= \ff_{\pi_{i}}\setminus( \pa\ff_{\pi_{i}}\cap \pa\ff_{\pi})$ has
a natural Lie groupoid structure induced from the one of $\mathcal{G}^{(1)}_{\pi}$ with units given
by $\mathcal{G}^{(0)}_{\ff_{\pi_{i}}}= \Delta_{\ff_{\pi_{i}}}$.  The conormal distributions in
$\Psi^{m}_{\ff_{\pi}}(X)$ which have compact support on $\mathcal{G}^{(1)}_{\ff_{\pi_{i}}}$ can
be understood as elements of the algebra $\Psi^{*}(\mathcal{G}^{(1)}_{\ff_{\pi_{i}}})$ of 
pseudodifferential operators associated to the Lie groupoid $\mathcal{G}^{(1)}_{\ff_{\pi_{i}}}$.

The space $\Psi^{m}_{\ff_{\pi_{i}}}(H_{i};E,F)$ can also be interpreted as a space of suspended
$\fS$-operators.  To see this, notice that since the fibres of the fibration $\pi_{i}: H_{i}\to S_{i}$ are naturally manifolds with fibre corners with
typical fibre $F_{i}$ having iterated fibration structure $\pi_{F_{i}}$,  we can form the fibrewise $\pi_{F_{i}}$-double space
\begin{equation}
\xymatrix{  (F_{i})^{2}_{\pi_{F_{i}}} \ar@{-}[r] &   (H_{i}\times_{\pi_{i}}H_{i})_{\pi_{F_{i}}}  \ar[d]\\
                & S_{i}.
}
\label{sm.7bb}\end{equation}  
If $\overline{{}^{\pi}NS_{i}}$ denotes the radial compactification of the vector bundle ${}^{\pi}NS_{i}\to S_{i}$ defined in \eqref{nor.2}, then
notice that the front face $\ff_{\pi_{i}}$ is naturally identified with the total space of the fibration obtained from the fibration \eqref{sm.7bb} by
pulling it back to 
$\overline{{}^{\pi}NS_{i}}$.  This means we have a natural fibration 
\begin{equation}
\xymatrix{  (F_{i})^{2}_{\pi_{F_{i}}} \ar@{-}[r] &   \ff_{\pi_{i}} \ar[d]\\
                & \overline{{}^{\pi}NS_{i}}.
}
\label{sm.7c}\end{equation}
With this identification, the Schwartz kernels in \eqref{sm.7} corresponds to the Schwartz kernels of ${}^{\pi}NS_{i}$-suspended $\fS$-operators
associated to the fibration $\pi_{i}: H_{i}\to S_{i}$, that is,
\begin{equation}
  \Psi^{m}_{\ff_{\pi_{i}}}(H_{i};E,F)= \Psi^{m}_{\fS-\sus({}^{\pi}NS_{i})}(H_{i}/S_{i};E,F).
\label{sm.7d}\end{equation}
Recalling the identification \eqref{nor.2}, we see that,  as a suspended operator, the symbol $\sigma_{\pa_{i}}(P)$ has a natural action on Schwartz
sections,
\begin{equation}
  \sigma_{\pa_{i}}(P): \cS({}^{\pi}NH_{i};E)\to \cS({}^{\pi}NH_{i};F).
  \label{sm.7e}\end{equation}

\begin{definition}
An operator $P\in \Psi^{m}_{\fS}(X;E,F)$ is said to be \textbf{fully elliptic} if it is elliptic and if for all $i\in \{1,\ldots,k\}$, 
$\sigma_{\pa_{i}}(P)$ is invertible as a map 
\[
 \sigma_{\pa_{i}}(P): \cS({}^{\pi}NH_{i};E)\to \cS({}^{\pi}NH_{i};F). 
\]  
A $V$-suspended $\fS$-operator $P\in \Psi^{m}_{\fS-\sus(V)}(X;E,F)$ is said to be \textbf{fully elliptic} if, as a $\fS$-operator in
$\Psi^{m}_{\fS}(\overline{V}\times X;E,F)$, it is elliptic
and if for all boundary hypersurfaces of the form $Z_{i}= \overline{V}\times H_{i}$, the corresponding symbol $\sigma_{\pa_{i}}(P)$ is invertible as a
map
 \[
 \sigma_{\pa_{i}}(P): \cS({}^{\varpi}NF_{i};E)\to \cS({}^{\varpi}NF_{i};F). 
\]  
\label{fuel.1}\end{definition}

If $H_{i}$ and $H_{j}$ are two hypersurfaces such that $H_{i}<H_{j}$, then the associated symbols
$\sigma_{\pa_{i}}$ and $\sigma_{\pa_{j}}$ satisfy a certain compatibility condition, namely, their respective
restrictions to $\ff_{\pi_{i}}\cap \ff_{\pi_{j}}$ agree.  From the point of view of suspended operators,
this means that the restriction of $\sigma_{\pa_{j}}(P)\in \Psi^{m}_{\ff_{\pi_{j}}}(X)$ to
$\left. {}^{\pi}N^{*}S_{j} \right|_{S_{ji}}$ is the symbol of the suspended family $\sigma_{\pa_{i}}(P)$ associated
to the face $H_{j}\cap H_{i}$.

\section{Composition} \label{com.0}

To show that $\fS$-operators compose nicely, various strategy could be used.  One approach consists in defining pseudodifferential operators using Lie
groupoids as in \cite{NWX}, in which case the fact the  composition of operators in the calculus remains in the calculus follows directly from the
definition.  As indicated earlier, the disadvantage with such an approach is that the inverse of an invertible operator is not typically within the
algebra.  Another approach, developed by Melrose and collaborators (see for instance \cite{Mazzeo-Melrose}), is to consider a triple space suitably
blown up where composition can be represented by a pushforward map coming from a $b$-fibration.  The result then follows from the description in
\cite{MelroseMWC} of  general mapping properties that such pushforward maps  satisfy.  Such an approach is likely to work in our context, but might
involves a rather complicated  triple space.  Instead, we will proceed by less geometric means and follow the approach of \cite{Krainer} by working
locally and using a proof by induction on the dimension of the manifold with fibred corners.  

\begin{theorem}
Let $E,F,G$ be smooth complex vector bundles on a manifold with fibred corners $X$.  Then for $A\in \Psi^{m}_{\fS}(X;F,G)$ and $B\in
\Psi^{n}_{\fS}(X;E,F)$, we have that 
\[
     A\circ B\in \Psi^{m+n}_{\fS}(X;E,G), \quad \mbox{with} \;  \sigma_{\pa_{i}}(A\circ B)= \sigma_{\pa_{i}}(A)\circ \sigma_{\pa_{i}}(B)
     \]
     for each hypersurface $H_{i}\subset X$ of $X$.  Moreover, the induced map
     \[
      \Psi^{m}_{\fS}(X;F,G)\times  \Psi^{n}_{\fS}(X;E,F)\; \to \;\Psi^{m+n}_{\fS}(X;E,G)    \]
      is continuous with respect to the natural Fr\'echet topology on each space.  A similar result holds for polyhomogeneous $\fS$-operators.
\label{composition}\end{theorem}

To describe the inductive step in the proof of this theorem, 
consider, for $p\in \bbN$,  the new manifold with corners $\overline{\bbR^{p}}\times X$ where $\overline{\bbR^{p}}$ is the radial compactification of
$\bbR^{p}$ as described in \cite{MelroseGST}.  A natural boundary defining function for the boundary $\pa\overline{\bbR^p}\cong \bbS^{p-1}$ is given
by $(1+r^{2})^{-\frac12}$ where $r$ is the Euclidean distance from the origin.  

 Notice that $\overline{\bbR^p}\times X$ has a natural structure of manifold with fibred corners induced from the one of $X$.  Indeed, the fibration
on the boundary hypersurface $Z_{0}=\pa\overline{\bbR^{p}}\times X$ is given by the projection on $\pa\overline{\bbR^{p}}$, while on the boundary hypersurface
$Z_{i}=\overline{\bbR^{p}}\times H_{i}$, where $H_{i}\subset X$ is a hypersurface of $X$ with fibration $\pi_{i}: H_{i}\to S_{i}$, the fibration is
given by
\[
          \Id\times \pi_{i}: \overline{\bbR^{p}}\times H_{i}\to \overline{\bbR^{p}}\times S_{i}.
\]
\begin{lemma}
Suppose that the conclusion of Theorem~\ref{composition} holds for the manifold with fibred corners $X$.  Then for $A\in
\Psi^{m}_{\fS}(\overline{\bbR^p}\times X)$ and $B\in \Psi^{n}_{\fS}(\overline{\bbR^p}\times X)$, we have 
\[
   A\circ B\in \Psi^{m+n}_{\fS}(\overline{\bbR^p}\times X), \quad \mbox{with} \;   
   \sigma_{\pa_{j}}(A\circ B)=\sigma_{\pa_{j}}(A)\circ \sigma_{\pa_{j}}(B)
 \]
 for all boundary hypersurfaces $Z_{j}\subset \overline{\bbR^p}\times X$. 
\label{com.14}\end{lemma}
\begin{proof}
Using the Fourier transform on $\bbR^{p}$, we can describe the action of operators $A\in \Psi^{m}_{\fS}(\overline{\bbR^p}\times X)$ and $B\in
\Psi^{n}_{\fS}(\overline{\bbR^{p}}\times X)$  on $u\in \dot{\cC}^{\infty}(\overline{\bbR^{p}}\times X)$ by
\begin{equation}
\begin{gathered}
 Au(t)= \frac{1}{(2\pi)^{p}} \int e^{i(t-t')\cdot \tau} a(t;\tau) u(t')dt'd\tau, \\
  Bu(t)= \frac{1}{(2\pi)^{p}} \int e^{i(t-t')\cdot \tau} b(t;\tau) u(t')dt'd\tau. 
 \end{gathered} 
\label{com.10}\end{equation}
Here, $a$ and $b$ are operator-valued symbols,
\begin{equation}
   a\in \CI( \overline{\bbR^{p}}; \Psi^{m}_{\fS-\sus(p)}(X)), \quad 
   b\in \CI( \overline{\bbR^{p}}; \Psi^{n}_{\fS-\sus(p)}(X)),
\label{com.11}\end{equation}
where $\Psi^{\ell}_{\fS-\sus(p)}(X)$ is the space of $\bbR^{p}$-suspended $\fS$-operators of order $\ell$ on $X$ and the variable $\tau\in\bbR^p$ in
\eqref{com.10} is seen as the suspension parameter.

If we forget that $a$ and $b$ are operator-valued, then there symbol class is the one introduced in \cite{Parenti} and \cite{Shubin} (see also 
\cite{MelroseGST}).  In this setting, there are standard methods to study the composition of operators, see for instance the proof of Proposizione~1.4
in \cite{Parenti}, or in the context of the Weyl calculus, the proof of Theorem~29.1 in \cite{Shubin_book}.  Since the operator-valued symbols are such
that
\begin{equation}
      a\in \CI(\overline{\bbR_{t}^{p}}; \Psi^{m}_{\fS-\sus(p)}(X)) \; \Longrightarrow \;  D_{t}^{\alpha}D^{\beta}_{\tau}a \in
(1+t^{2})^{-\frac{|\alpha|}{2}} \CI(\overline{\bbR^{p}_{t}}; \Psi^{m-|\beta|}_{\fS-\sus(p)}(X)), \quad 
  \label{com.11b}\end{equation}
where $\tau$ is the suspension parameter, these methods have a straightforward generalization.

 Indeed,  let $c(t,\tau)$ be the operator-valued symbol such that
 \begin{equation}
 ABu(t)= \frac{1}{(2\pi)^{p}} \int e^{i(t-t')\cdot \tau} c(t,\tau) u(t')dt'd\tau.
  \label{par.1}\end{equation}
 As in \cite{Parenti}, for each $N\in \bbN$, we have
 \begin{equation}
   c(t,\tau)= \sum_{|\alpha|< N} \frac{1}{\alpha !} \pa_{\tau}^{\alpha} a(t,\tau) D_{t}^{\alpha}b(t,\tau) +
    R_{N}(t,\tau)
 \label{par.2}\end{equation}
 with remainder term $ R_{N}(t,\tau)$ given by
  \begin{equation}
    \sum_{|\alpha|=N} \frac{N}{\alpha !} \int_{0}^{1}(1-\lambda)^{N-1} 
   \left(  \frac{1}{(2\pi)^{p}} \int e^{-iz\cdot \zeta}  \pa_{\tau}^{\alpha} a(t,\tau+\lambda \zeta) D^{\alpha}_{t}b(t+z,\tau)dzd\zeta \right)
d\lambda
 \label{par.3}\end{equation}
 Using our inductive hypothesis on the composition of $\fS$-operators on $X$ as well as \eqref{com.11b}, we can essentially proceed as in
\cite{Parenti} to estimate the remainder term $R_{N}(t,\tau)$.  The only significant difference is that  taking a $\tau$ derivative of the symbol $a$ not
only improve its decay as $\tau$ tends to infinity, but it also reduces its order as an operator on $X$.  If we only focus on the decay behavior in
$\tau$,  we obtain
 \[
     R_{N}\in \cS^{m+n-N}_{-N}( \bbR^{p}_{t}\times \bbR^{p}_{\tau}; \Psi^{m+n}_{\fS}(X))
 \]
 where $ \cS^{k}_{\ell}( \bbR^{p}_{t}\times \bbR^{p}_{\tau}; \Psi^{m+n}_{\fS}(X))$ is the space of operator-valued symbols $q$ such that for any
Fr\'echet semi-norm $\| \cdot \|$ of $\Psi^{m+n}_{\fS}(X)$,   
 \[
     \sup_{t,\tau} \left(  (1+|\tau|^{2})^{\frac{|\beta|-k}{2}} (1+|t|^{2})^{\frac{|\alpha|-\ell}{2}} 
     \|  D_{t}^{\alpha}D_{\tau}^{\beta} q \|  \right)  < \infty \quad \forall\; \alpha,\beta\in \bbN_{0}^{p}.
 \]

It is more useful however to control $R_{N}$ as a symbol valued in a space of lower order $\fS$-operators.  Thus, if  instead we use half of the
$\tau$ derivative to reduce the order of the operator, we obtain for $N$ even
\begin{equation}
  R_{N}\in \cS^{m+n-\frac{N}{2}}_{-N} (\bbR^{p}_{t}\times \bbR^{p}_{\tau}; \Psi^{m+n-\frac{N}{2}}_{\fS}(X)).  
\label{par.4}\end{equation}
By Borel's lemma, there is an operator-valued symbol $e\in\CI(\overline{\bbR^{p}_{t}};
\Psi^{m+n}_{\fS-\sus(p)}(X))$ such that 
\[
            e(t,\tau) \sim \sum_{\alpha} \frac{1}{\alpha !} \pa_{\tau}^{\alpha} a(t,\tau) D^{\alpha}_{t}b(t,\tau),
\]
where this notation means that for all $N\in \bbN$, 
\[
       e(t,\tau) - \sum_{\alpha<N} \frac{1}{\alpha !} \pa_{\tau}^{\alpha} a(t,\tau) D^{\alpha}_{t}b(t,\tau)
        \in (1+ |t|^{2})^{-\frac{N}{2}}\CI(\overline{\bbR^{p}_{t}}; \Psi^{m+n-N}_{\fS-\sus(p)}(X)),
\]
we see from \eqref{par.4} that 
\[
  c(t,\tau)-e(t,\tau) \in \dot{\cC}^{\infty}(\overline{\bbR^{p}_{t}}\times \overline{\bbR^{p}_{\tau}}; \Psi^{-\infty}_{\fS}(X))= 
\dot{\cC}^{\infty}(\overline{\bbR^{p}_{t}}; \Psi^{-\infty}_{\fS-\sus(p)}(X)).  
\]  
  Therefore, we have that  
  \begin{equation}
   c \in \CI(\overline{\bbR^{p}_{t}}; \Psi^{m+n}_{\fS-\sus(p)}(X)) \quad \mbox{with} \; c
   \sim \sum_{\alpha} \frac{1}{\alpha !} \pa_{\tau}^{\alpha} a(t,\tau) D^{\alpha}_{t}b(t,\tau),
  \label{par.5}\end{equation}
  which implies in particular that $AB \in \Psi^{m+n}_{\fS}(\overline{\bbR^{p}}\times X)$.    Moreover, since this argument automatically provides
uniform control on each of the Fr\'echet semi-norms of $c$ in terms of the Fr\'echet semi-norms of $a$ and $b$, we see that the induced map 
\[
   \Psi^{m}_{\fS}(\overline{\bbR^{p}}\times X)\times \Psi^{n}_{\fS}(\overline{\bbR^p}\times X)  \; \to \;
   \Psi^{m}_{\fS}(\overline{\bbR^{p}}\times X)
\]  
is continuous with respect to the natural Fr\'echet topology on each spaces.

 By our assumptions on $X$, It follows directly that $\sigma_{\pa_{j}}(A\circ B)= \sigma_{\pa_{j}}(A)\circ\sigma_{\pa_{j}}(B)$ for hypersurfaces of
the form $Z_{j}=\overline{\bbR^p}\times H_{j}$, while for the hypersurface $Z_{0}= \pa\overline{\bbR^p}\times X$, it is a consequence of the
asymptotic expansion of the operator-valued symbol $c$ in \eqref{par.5}, namely 
\[
   \sigma_{\pa_{0}}(AB)= \left. c \right|_{\pa \overline{\bbR^{p}_{t}}} =  \left. a\right|_{\pa \overline{\bbR^{p}_{t}}}    \left. b \right|_{\pa
\overline{\bbR^{p}_{t}}}   = \sigma_{\pa_{0}}(A) \sigma_{\pa_{0}}(B).
\]

\end{proof}

To proceed further, we need some notation.  For each hypersurface $H_{i}$ of $X$, let 
\begin{equation}
\nu_{i}:H_{i}\times [0,\epsilon_{i})\to \cU_{i}\subset X
\label{com.1}\end{equation}
 be a tubular neighborhood of $H_{i}$ in $X$ compatible with the boundary defining function $x_{i}$, that is, such that $x_{i}(\nu_{i}(h,t))=t$
for $h\in H_{i}$ and $t\in [0,\epsilon_{i})$.  
To show that $\fS$-operators form an algebra, we will use the previous lemma to deal with $\fS$-operators having their Schwartz kernels supported near
 the front face $\ff_{\pi_{i}}$, but supported away from the front faces $\ff_{\pi_{j}}$ for $H_{j}<H_{i}$.    

\begin{lemma}
Suppose that the conclusions of Theorem~\ref{composition} hold for all manifolds with fibred corners $Y$ of dimension less than the one of $X$.
Suppose that $A\in \Psi^{m}_{\fS}(X)$ and $B\in \Psi^{n}_{\fS}(X)$ are such that their Schwartz kernels are  supported inside the set 
\[
       \beta_{\pi}^{-1}\left ( \nu_{i}(\pi_{i}^{-1}(V_{i})\times [0,\epsilon_{i}))^{2}  \right)\subset X^{2}_{\pi}
\]
where $V_{i}\subset S_{i}\setminus\pa S_{i}$ is some open set in the interior of the base $S_{i}$ of the fibration $\pi_{i}:H_{i}\to S_{i}$.  Then 
\[
A\circ B\in \Psi^{m+n}_{\cS}(X) \quad \mbox{with}  \; \sigma_{\pa_{j}}(A\circ B)= \sigma_{\pa_{j}}(A)\circ \sigma_{\pa_{j}}(B)
\]
for all hypersurfaces $H_{j}\subset X$.  
\label{com.2}\end{lemma}
\begin{proof}
Let $F_{i}$ be the typical fibre of the fibration
\begin{equation}
   \xymatrix{   F_{i}\ar@{-}[r] &  H_{i} \ar[d]^{\pi_{i}}  \\
                                          &  S_{i}.
   }
\label{com.3}\end{equation}
Then as described in \S~\ref{mwfc.0}, the fibre $F_{i}$ is naturally a manifold with fibred corners.  Since
$\dim F_{i}<\dim X$, it is part of our assumptions that  $\Psi^{m}_{\fS}(F_{i})\circ \Psi^{n}_{\fS}(F_{i})\subset \Psi^{m+n}_{\fS}(F_{i})$.  The
strategy of the proof is to reduce composition of the operators $A$ and $B$ to Lemma~\ref{com.14}.

Let $\{W_{q}\}_{q\in \cQ}$ be a finite covering of the closure of $V_{i}$ in $S_{i}\setminus \pa S_{i}$ by open sets in $S_{i}\setminus \pa S_{i}$  
diffeomorphic to open balls and such that the fibration \eqref{com.3} restricts to a trivial fibration over each $S_{i}$.  Let $\varphi_{q}\in
\CI_{c}(W_{q})$ be functions which restricts to give a partition of unity on $V_{i}$ and let $\tilde{\varphi}_{q}\in \CI_{c}(W_{q})$ be  functions
such that $\varphi_{q}\tilde{\varphi}_{q}= \varphi_{q}$.    
Let 
\[
   \psi_{q}= (\nu_{i})_{*} \pr^{*}\pi_{i}^{*} \varphi_{q}, \quad  \tilde{\psi}_{q}= (\nu_{i})_{*} \pr^{*}\pi_{i}^{*} \tilde{\varphi}_{q},  
\]
be the corresponding pulled back functions on $\cU_{i}= \nu_{i}( H_{i}\times [0,\epsilon_{i}))$
where $\pr: H_{i}\times [0,\epsilon_{i})\to H_{i}$ is the projection on the left factor.  Then we can write the operator $A$ as
\begin{equation}
A = \sum_{q} A\psi_{q}  = \sum_{q} \left(  \tilde{\psi}_{q}A\psi_{q} + (1-\tilde{\psi}_{q})A\psi_{q}\right).
\label{com.4}\end{equation}    
Since $\tilde{\psi}_{q}\psi_{q}=\psi_{q}$, the Schwartz kernel of the second term is supported away from the diagonal in $X^{2}$, which means it is an
element of $\dot{\Psi}^{-\infty}_{\fS}(X)$.  Thus, we have that 
\begin{equation}
    A\equiv \sum_{q} \tilde{\psi}_{q}A\psi_{q}  \mod \dot{\Psi}^{-\infty}_{\fS}(X).
\label{com.5}\end{equation} 
Similarly, we have that 
\begin{equation}
   \psi_{q}B= \psi_{q}B\tilde{\psi}_{q} +  \psi_{q}B(1-\tilde{\psi}_{q}) \equiv \psi_{q}B\tilde{\psi}_{q}  \mod  \dot{\Psi}^{-\infty}_{\fS}(X).
    \label{com.6}\end{equation}
Thus, using Corollary~\ref{ideal.2}, we see that 
\begin{equation}
\begin{aligned}
AB & \equiv  \sum_{q} \tilde{\psi}_{q}A\psi_{q}B \mod \dot{\Psi}^{-\infty}_{\fS}(X) \\
& = \sum_{q} \tilde{\psi}_{q}A\psi_{q}\tilde{\psi}_{q}B =  \sum_{q} (\tilde{\psi}_{q}A\tilde{\psi}_{q})\psi_{q}B  \\
 & \equiv \sum_{q} (\tilde{\psi}_{q}A\tilde{\psi}_{q})(\psi_{q}B \tilde{\psi}_{q}) 
   \mod  \dot{\Psi}^{-\infty}_{\fS}(X).
\end{aligned}
\label{com.7}\end{equation}
This means we can assume both $K_{A}$ and $K_{B}$ are supported in the subset  
\[
\beta_{\pi}^{-1}\left ( \nu_{i}(\pi_{i}^{-1}(W_{q})\times [0,\epsilon))^{2}  \right)\subset X^{2}_{\pi}.
\]
Since we are assuming $W_{q}$ is diffeomorphic to an open ball, this means there exists an embedding 
\begin{equation}
     \iota_{q}: W_{q} \hookrightarrow \bbS^{p_{i}-1},
\label{com.8}\end{equation}
where $p_{i}-1= \dim S_{i}= \dim W_{q}$.  Since the fibration $\pi_{i}$ is trivial when restricted to $W_{q}$, we can assume 
\[
        \pi_{i}^{-1}(W_{q})= F_{i}\times W_{q}
\]
with $\pi_{i}$ given by projecting on the right factor.   The embedding \eqref{com.8} can be extended to an embedding 
\begin{equation}
\xymatrix{
   W_{q}\times [0,\epsilon_{i}) \ar[r]^(0.45){\iota_{q}\times \Id} & \bbS^{p_{i}-1}\times [0,\epsilon_{i}) \ar[r] & \overline{\bbR^{p_{i}}} 
   }
\label{com.9}\end{equation}
where the second map is the standard collar neighborhood of $\bbS^{p_{i}-1}= \pa \overline{\bbR^{p_{i}-1}}$ in the radial compactification
$\overline{\bbR^{p_{i}-1}}$ of $\bbR^{p_{i}}$ using the boundary defining function $\frac{1}{\sqrt{r^{2}+1}}$ where $r$ is the distance from the
origin.   

Via these identifications, this means we can regard $A$ and $B$ as operators acting on functions of $\bbR^{p_{i}}\times F_{i}$, more precisely:  $A\in
\Psi^{m}_{\fS}(\overline{\bbR^{p_{i} } }\times F_{i})$, $B\in \Psi^{n}_{\fS}(\overline{\bbR^{p_{i} } }\times F_{i})$.  The result then follows by
applying Lemma~\ref{com.14}. 

\end{proof}

We have now all the ingredients to prove the composition theorem.

\begin{proof}[Proof of Theorem~\ref{composition}]
By using a partition of unity, we can work locally in open sets where the vector bundles $E$, $F$ and $G$ are trivial.  Thus, without loss of
generality, we can assume that $E=F=G=\underline{\bbC}$ and $A\in \Psi^{m}_{\fS}(X)$, $B\in \Psi^{n}_{\fS}(X)$.  

Since the case where $\dim X=0$ is trivial, we can assume by induction on the dimension that the theorem is true for manifolds with fibred corners of
dimension less than the one of $X$.  For each boundary hypersurface $H_{i}$ of $X$, consider the tubular neighborhood $\nu_{i}: H_{i}\times
[0,\epsilon_{i})_{x_{i}}\to \cU_{i}\subset X$ of \eqref{com.1}.  Let also $\chi_{i},\widetilde{\chi}_{i}, \widehat{\chi}_{i}\in
\CI_{c}(\cU_{i})\subset \CI(X)$ be non-negative cut-off functions such that $\widehat{\chi}_{i}\equiv 1$ near $H_{i}$,  $\chi_{i}\widehat{\chi}_{i}=
\widehat{\chi}_{i}$ and $\widetilde{\chi}_{i}\chi_i=\chi_{i}$.  Using the cut-off functions $\chi_{i}$, $\widetilde{\chi}_{i}$ and
$\widehat{\chi}_{i}$, we can rewrite the composition of $A$ and $B$ as
\begin{equation}
\begin{aligned}
   AB &= A \chi_{i} B  +  A(1-\chi_{i})B  \\
         &=  \widetilde{\chi}_{i} A \chi_{i} B+  (1-\widetilde{\chi}_{i}) A \chi_{i} B  + 
         \widehat{\chi}_{i}  A(1-\chi_{i})B  + (1-\widehat{\chi}_{i})A(1-\chi_{i})B.
   \end{aligned}   
\label{com.15}\end{equation} 
Since $\chi_{i}\widehat{\chi}_{i}= \widehat{\chi}_{i}$ and $\widetilde{\chi}_{i}\chi_{i}=\chi_{i}$, the Schwartz kernels of
$(1-\widetilde{\chi}_{i})A\chi_{i}$ and $\widehat{\chi}_{i}A(1-\chi_{i})$ are both supported away from the diagonal in $X\times X$, which means the
operators $(1-\widetilde{\chi}_{i})A\chi_{i}$ and $\widehat{\chi}_{i}A(1-\chi_{i})$ are both in 
$\dot{\Psi}^{-\infty}_{\fS}(X)$.  Thus, using Corollary~\ref{ideal.2}, we see that modulo operators
in $\dot{\Psi}^{-\infty}_{\fS}(X)$, we have
\begin{equation}
  AB\equiv \widetilde{\chi}_{i}A \chi_{i} B + (1-\widehat{\chi}_{i})A(1-\chi_{i})B \mod \dot{\Psi}^{-\infty}_{\fS}(X).
\label{com.16}\end{equation}
Similarly, if $\chi_i'\in \CI_c(\cU_i)$ is such that $\widehat{\chi}_i\chi_i'=\chi_i'$ and $\chi_i'\equiv 1$ near $H_i$, then we can write the operator $B$ as
\begin{equation}
\begin{aligned}
  B &= \widehat{\chi}_{i}B \chi_{i}+ \widehat{\chi}_{i}B(1-\chi_{i}) + (1-\widehat{\chi}_{i})B \chi_{i}'+ (1-\widehat{\chi}_{i})B (1-\chi_{i}')  \\
  &\equiv \widehat{\chi}_{i} B\chi_{i} + (1-\widehat{\chi}_{i}) B (1-\chi_{i}') \mod \dot{\Psi}^{-\infty}_{\fS}(X).
\end{aligned}  
\label{com.17}\end{equation}
If $\check{\chi}_i\in \CI_c(\cU_i)$ is another cut-off function such that $\check{\chi}_i \widetilde{\chi}_i=\widetilde{\chi}_i$, then we can also write $B$ as
\begin{equation}
\begin{aligned}
  B &= \widetilde{\chi}_{i}B \check{\chi}_{i}+ \widetilde{\chi}_{i}B(1-\check{\chi}_{i}) + (1-\widetilde{\chi}_{i})B \chi_{i}+ (1-\widetilde{\chi}_{i})B (1-\chi_{i})  \\
  &\equiv \widetilde{\chi}_{i} B\check{\chi}_{i} + (1-\widetilde{\chi}_{i}) B (1-\chi_{i}) \mod \dot{\Psi}^{-\infty}_{\fS}(X).
\end{aligned}  
\label{com.17b}\end{equation}
Substituting \eqref{com.17} and \eqref{com.17b} in \eqref{com.16}, 
we see by Corollary~\ref{ideal.2} that 
\begin{equation}
\begin{aligned}
 AB &\equiv (\widetilde{\chi}_{i} A \chi_{i})(\widetilde{\chi}_{i}B\check{\chi}_{i})  + \widetilde{\chi}_{i} A \chi_{i}
(1-\widetilde{\chi}_{i})B(1-\chi_{i})  + (1-\widehat{\chi}_{i})A(1-\chi_{i})\widehat{\chi}_{i}B \chi_{i}  \\
  & \; + (1-\widehat{\chi}_{i})A(1-\chi_{i})(1-\widehat{\chi}_{i})B(1-\chi_{i}') \mod \dot{\Psi}^{-\infty}_{\fS}(X)  \\
 & \equiv  (\widetilde{\chi}_{i} A \chi_{i})(\widetilde{\chi}_{i}B\check{\chi}_{i})+ (1-\widehat{\chi}_{i})A(1-\chi_{i})(1-\widehat{\chi}_{i})B(1-\chi_{i}')
\mod \dot{\Psi}^{-\infty}_{\fS}(X).
 \end{aligned} 
\label{com.18}\end{equation}
Thus, from \eqref{com.18}, we can reduce the problem of composition to two situations,
\begin{itemize}
\item $K_{A}$ and $K_{B}$ are supported near $\ff_{\pi_{i}}$;
\item $K_{A}$ and $K_{B}$ are supported away from $\ff_{\pi_{i}}$.
\end{itemize}
In particular, if $H_{i}$ is a minimal hypersurface with respect to the partial order of hypersurfaces of $X$, then the first term on the right hand
side of \eqref{com.18} can be taken care of by Lemma~\ref{com.2}.
  In fact, starting with the minimal hypersurfaces $H_{i}$ and proceeding recursively on the partial order of boundary  hypersurfaces of $X$ using \eqref{com.18}
and Lemma~\ref{com.2} at each step, we can reduce to the case where $K_{A}$ and $K_{B}$ are supported away from $\ff_{\pi_{i}}$ for all $i$.  Adding
operators in $\dot{\Psi}^{-\infty}_{\fS}(X)$ if necessary, we can even reduce to the case the Schwartz kernels of $A$ and $B$  have compact support in
$(X\setminus\pa X)^{2}$.  By doubling $X$ to get a smooth closed manifold, this reduces to the standard result about composition of pseudodifferential
operators on closed manifolds.  It is straightforward to check that polyhomogeneity is preserved under composition.      
\end{proof}

\section{Mapping properties}\label{mp.0}

Let $(X,\pi)$ be a manifold with fibred corners.  Let $H_{1},\ldots, H_{k}$ be its boundary hypersurfaces with choice of boundary defining functions
$x_{1},\ldots, x_{k}$.  As for the $\Phi$-calculus of \cite{Mazzeo-Melrose}, an important ingredient in the study of mapping properties of
$\fS$-operators is the construction of a parametrix for fully elliptic operators.  We will also need such a result for $\fS$-suspended operators, in
which case the notation
\begin{multline*}
\dot{\Psi}^{-\infty}_{\fS-\sus(V)}(X;E_{1},E_{2})= \{ A\in \Psi^{-\infty}_{\fS-\sus(V)}(X;E_{1},E_{2});  \\ \widehat{A}(\Upsilon)\in
\dot{\Psi}^{-\infty}_{\fS}(X;E_{1},E_{2}) \; \forall \Upsilon\in V^{*} \},
\end{multline*}
for $E_{1}$ and $E_{2}$ complex vector bundles over $X$, is useful to describe the error term. 

\begin{proposition}[Parametrix]
If $P\in \Psi^{m}_{\fS}(X;E,F)$ is fully elliptic, then there exists $Q\in 
\Psi^{-m}_{\fS}(X;F,E)$ such that 
\begin{equation*}
  \Id-QP\in \dot{\Psi}^{-\infty}_{\fS}(X;E), \quad \Id - PQ\in \dot{\Psi}^{-\infty}_{\fS}(X;F).
\end{equation*}
Moreover, $\ker P\subset \dot{\cC}^{\infty}(X;E)$ and $\ker P^{*}\subset \dot{\cC}^{\infty}(X;F)$.
Similarly, if $V$ is an Euclidean vector space and $P\in \Psi^{m}_{\fS-\sus(V)}(X;E,F)$ is fully elliptic, then there exists $Q\in 
\Psi^{-m}_{\fS-\sus(V)}(X;F,E)$ such that 
\begin{equation*}
  \Id-QP\in \dot{\Psi}^{-\infty}_{\fS-\sus(V)}(X;E), \quad \Id - PQ\in \dot{\Psi}^{-\infty}_{\fS-\sus(V)}(X;F).
\end{equation*}
\label{fp.1}\end{proposition}

\begin{proof}
Using this proposition and Corollary~\ref{fp.17} below and proceeding by induction on the dimension of $X$, we can assume  that 
$\sigma_{\pa_{i}}(P)^{-1}\in \Psi^{-m}_{\ff_{\pi_{i}}}(H_{i};F,E)$.  This means we can
choose $Q_{0}\in \Psi^{-m}_{\fS}(X;F,E)$ such that $\sigma_{-m}(Q_{0})=
\sigma_{m}(P)^{-1}$ and $\sigma_{\pa_{j}}(Q_{0})=\sigma_{\pa_{j}}(P)^{-1}$.  Then we have 
\begin{equation}
   \Id-Q_{0}P\in x\Psi^{-1}_{\fS}(X;E), \quad \Id-PQ_{0}\in x\Psi^{-1}_{\fS}(X;F).  
\label{fp.2}\end{equation}
Suppose for a proof by induction that we have defined 
$Q_{\ell}\in x^{\ell}\Psi^{-m-\ell}_{\fS}(X;F,E)$ for $\ell\le n$ such that 
$\tilde{Q}_{n}=Q_{1}+\cdots +Q_{n}$ satisfies
\begin{equation}
   \Id-\tilde{Q}_{n}P\in x^{n+1}\Psi^{-n-1}_{\fS}(X;E), \quad \Id-P\tilde{Q}_{n}\in x^{n+1}\Psi^{-n-1}_{\fS}(X;F).  
\label{fp.2b}\end{equation}
Then, setting $\tilde{Q}_{n+1}= \tilde{Q}_{n}+Q_{n+1}$, we would like to find $Q_{n+1} \in x^{n+1}\Psi^{-m-n-1}_{\fS}(X;F,E)$  such that 
\begin{equation}
   \Id- \tilde{Q}_{n+1}P= \Id- \tilde{Q}_{n}P- Q_{n+1}P \in 
   x^{n+2}\Psi^{-n-2}_{\fS}(X;E),
\label{fp.3}\end{equation}
that is, such that 
\begin{equation}
   Q_{n+1}P= \Id -\tilde{Q}_{n}P \quad \mbox{modulo} \quad x^{n+2}\Psi^{-n-2}_{\fS}(X;E).
\label{fp.3b}\end{equation}
Thus, taking $Q_{n+1}= (\Id-\tilde{Q}_{n}P)\tilde{Q}_{n}$ will give
\begin{equation}
  \Id-\tilde{Q}_{n+1}P\in x^{n+2}\Psi^{-n-2}_{\fS}(X;E)
\label{fp.4}\end{equation}
with $\tilde{Q}_{n+1}= \tilde{Q}_{n}+Q_{n+1}$.  As one can check, we 
will also have that 
\begin{equation}
  \Id -PQ_{n+1}\in x^{n+2}\Psi^{-n-2}_{\fS}(X;F).  
\label{fp.5}\end{equation}
We can then define $Q$ to be the asymptotic sum of the $Q_{\ell}$ giving
the desired parametrix.  If $f\in \ker P$, then 
\begin{equation}
\begin{aligned}
  Pf=0 & \Rightarrow \quad QPf=0 \\
          &  \Rightarrow \quad f = (\Id-QP)f\in \dot{\mathcal{C}}^{\infty}(X;E)
 \end{aligned}
\label{fp.6}\end{equation}
since $\Id-QP\in \dot{\Psi}^{-\infty}_{\fS}(X;E)$.  There is a similar argument for the kernel of $P^{*}$.  For fully elliptic $V$-suspended
$\fS$-operators, the proof is similar and is left to the reader.
\end{proof}

\begin{corollary}
If $V$ is an Euclidean vector space and
$P\in \Psi^{m}_{\fS-\sus(V)}(X;E,F)$ is a fully elliptic $V$-suspended operators which is invertible as a map $P: \cS(V\times X;E)\to \cS(V\times
X;F)$, then it has an inverse in $\Psi^{-m}_{\fS-\sus(V)}(X;F,E)$.  
\label{fp.17}\end{corollary}
\begin{proof}
Let $P\in \Psi^{m}_{\fS-\sus(V)}(X;E,F)$ be as in the statement of the corollary and let $Q\in \Psi^{-m}_{\fS-\sus(V)}(X;F,E)$ be the parametrix of
Proposition~\ref{fp.1}.  in particular, we have that
\[
      \widehat{P}(\Upsilon)\widehat{Q}(\Upsilon)= \Id +\widehat{R}(\Upsilon), \quad \forall \; \Upsilon \in V^*, \; \mbox{where} \; R\in \dot{\Psi}^{-\infty}_{\fS-\sus(V)}(X;F).
\]
By \eqref{sus.23} we see that $\widehat{R}(\Upsilon)$ is small for $|\Upsilon|$ large, so that there exists $K>0$ with the property that 
$\Id+ \widehat{R}(\Upsilon)$ is invertible for $|\Upsilon|>K$ with inverse of the form $\Id +\widehat{S}(\Upsilon)$, where 
\[
        \widehat{S}(\Upsilon)= \sum_{k=1}^{\infty} (-1)^k \widehat{R}(\Upsilon)^k \in \dot{\Psi}^{-\infty}_{\fS}(X;F)
\]
satisfies \eqref{sus.23}.  Thus, for $|\Upsilon|>K$, we have that
\begin{equation}
   \widehat{P}(\Upsilon)^{-1}= \widehat{Q}(\Upsilon)(\Id+\widehat{S}(\Upsilon)).
\label{inverse.1}\end{equation}
Now, the invertibility of $P$ clearly implies the invertibility of $\widehat{P}(\Upsilon)$ for all $\Upsilon\in V^{*}$.  Using the parametrix $Q$, we have
\begin{equation}
\widehat{P}(\Upsilon)^{-1}= \widehat{P}(\Upsilon)^{-1}(\widehat{P}(\Upsilon)\widehat{Q}(\Upsilon)-\widehat{R}(\Upsilon))= \widehat{Q}(\Upsilon)-\widehat{P}(\Upsilon)^{-1}\widehat{R}(\Upsilon). 
\label{inverse.2}\end{equation}
By Proposition~\ref{ideal.1}, we must have $\widehat{P}(\Upsilon)^{-1}\widehat{R}(\Upsilon)\in\dot{\Psi}^{-\infty}_{\fS}(X; F, E)$ for all $\Upsilon\in V^{*}$.  Thus, from \eqref{inverse.1} and \eqref{inverse.2}, we see that 
\[
    \widehat{P}(\Upsilon)^{-1}= \widehat{Q}(\Upsilon) + \widehat{W}(\Upsilon),
\]
where $W\in \dot{\Psi}^{-\infty}_{\fS-\sus(V)}(X;F,E)$ is such that 
$\widehat{W}(\Upsilon)= \widehat{Q}(\Upsilon)\widehat{S}(\Upsilon)$ for 
$|\Upsilon|>K$.  Taking the inverse Fourier transform, we finally obtain that
\[
       P^{-1}= Q+ W \in \Psi^{-m}_{\fS-\sus(V)}(X;F,E).
\]
\end{proof}

As we will see, this last corollary will be useful to study the action of $\fS$-operators on square integrable functions.  Precisely, let $g_{\pi}$ be
a choice of $\fS$-metric and let $dg_{\pi}\in \CI(X; {}^{\pi}\Omega)$ be its volume form.  Let $L^{2}_{g_{\pi}}(X)$ be the corresponding space of
functions on $X\setminus \pa X$ that are square integrable with respect to the density $dg_{\pi}$.   
To establish the $L^{2}$-boundedness of $\fS$-pseudodifferential operators of order zero, we will, as in 
\cite{Mazzeo-Melrose},
follow the standard trick of H\"ormander relying on the construction of an approximate square root.    
\begin{proposition}
If $B\in \Psi^{0}_{\fS}(X)$ is formally self-adjoint with respect to a positive $\fS$-density $\nu$ on
$X$, then there exists $C>0$ sufficiently large so that 
\[
       C+B= A^{*}A+R
\]
for some $A\in \Psi^{0}_{\fS}(X)$ and $R\in \dot{\Psi}^{-\infty}_{\fS}(X)$.
\label{mp.1}\end{proposition}
\begin{proof}
The proof is by induction on the depth of $X$.  The 
case where $X$ is a closed manifold is well-known and the case where $X$ is a manifold with
boundary is proven by Mazzeo and Melrose in \cite{Mazzeo-Melrose}.  

 For $i\in\{1,\ldots,k\}$,  let $\nu_{S_{i}}$  be a positive section of $\Omega({}^{\pi}NS_{i})$ and write
$\left.\nu\right|_{H_{i}}= \nu_{F_{i}}\otimes \nu_{S_{i}}$ where $\nu_{F_{i}}$ is a positive density
in the fibres of the fibration $\pi_{i}:H_{i}\to S_{i}$.   Then the suspended family of $\fS$-pseudodifferential operators
$\hat{\sigma}_{\pa_{i}}(B)(\eta)$ with $\eta\in {}^{\pi}N^{*}S_{i}$
is formally self-adjoint with respect to the density $\nu_{F_{i}}$.  By our induction hypothesis (see Corollary~\ref{raccar.1}), 
for $C>0$ big enough, $(\hat{\sigma}_{\pa_{i}}(B)(\eta) +C)\in \Psi^0_{\fS}(F_{i})$ has a unique positive 
square root so that $(C+ \sigma_{\pa_{i}}(B))$ also has a unique positive square root in
$\Psi^{0}_{\ff_{\pi_{i}}}(X)$.  
Similarly, $(C+\sigma_{0}(B))$ has unique positive square root provided $C>0$ is large enough. 
Thus, we can find $A_{0}\in \Psi_{\fS}^{0}(X)$ such that 
\begin{equation}
\sigma_{0}(A_{0})= (C+\sigma_{0}(B))^{\frac12}, \quad 
\sigma_{\pa_{i}}(A_{0})=(C+\sigma_{\pa_{i}}(B))^{\frac12}, \quad i\in\{1,\ldots,k\}.
\label{mp.2}\end{equation}
Replacing $A_{0}$ by $\frac{1}{2}(A_{0}+A_{0}^{*})$ if necessary, we can assume that $A_{0}$ is formally self-adjoint with
\begin{equation}
  C+B-A^{2}_{0} \in x\Psi^{-1}_{\fS}(X).  
\label{mp.3}\end{equation}
To get an error term in $\dot{\Psi}^{-\infty}_{\fS}(X)$, we can proceed by induction.  Thus, assume that
we have found a formally self-adjoint operator $A_{\ell}\in\Psi_{\fS}^{0}(X)$ such that
\begin{equation}
     C+B- A^{2}_{\ell}= R_{\ell+1}\in x^{\ell+1}\Psi^{-\ell-1}_{\fS}(X).
\label{mp.4}\end{equation} 
Writing $A_{\ell+1}= A_{\ell}+Q_{\ell}$ where the formally self-adjoint operator $Q_{\ell}\in x^{\ell+1}\Psi_{\fS}^{-\ell-1}(X)$ is to be found, we
have 
\begin{equation}
\begin{aligned}
C+B- A_{\ell+1}^{2} &= R_{\ell+1}-Q_{\ell}A_{\ell}-A_{\ell}Q_{\ell}- Q_{\ell}^{2}  \\
               &= R_{\ell+1}-Q_{\ell}A_{\ell}-A_{\ell}Q_{\ell}
\end{aligned}
\label{mp.5}\end{equation}
modulo $x^{\ell+2}\Psi^{-\ell-2}_{\fS}(X)$.  First, this means we need to solve
\begin{equation}
  \sigma_{-\ell-1}(R_{\ell+1})= 2\sigma_{0}(A_{\ell})\sigma_{-\ell-1}(Q_{\ell}),
\label{mp.8}\end{equation}
which clearly has a formally self-adjoint solution $Q_{\ell,0}\in x^{\ell+1}\Psi^{-\ell-1}_{\fS}(X)$.
Thus, replacing $A_{\ell}$ by $A_{\ell,0}= A_{\ell}+Q_{\ell,0}$ in \eqref{mp.5}, this means we have to solve \eqref{mp.5}
with $R_{\ell+1}$ replaced by $R_{\ell+1,0}\in x^{\ell+1}\Psi^{-\ell-2}_{\fS}(X)$. 

Proceeding by induction on $i\in\{1,\ldots,k\}$ and with the convention that $x_0=1$ and $w_i=\prod_{j=1}^{i}x_j$, assume more generally we have found formally self-adjoint operators $Q_{\ell,j} \in w_{j-1}x^{\ell+1}\Psi^{-\ell-2}_{\fS}(X)$ for $j\le i-1$ such that 
$A_{\ell,i-1}= A_{\ell}+ \sum_{j=0}^{i-1} Q_{\ell,j}$ satisfies
\begin{equation}
  C+B-A^2_{\ell,i-1}= R_{\ell+1,i-1}\in w_{i-1} x^{\ell+1}\Psi^{-\ell-2}_{\fS}(X).
\label{mp.4b}\end{equation}
To find $Q_{\ell,i}$, write $Q_{\ell,i}=x^{\ell+1}_{i}\tilde{Q}_{\ell,i}$
 where 
$\tilde{Q}_{\ell,i}\in x^{-\ell-1}_{i}w_{i-1}x^{\ell+1}\Psi^{-\ell-2}_{\fS}(X)$.  Using Remark~\ref{notw.1}, this means we need to solve 
\begin{equation}
   \hat{\sigma}_{\pa_{i}}(x_{i}^{-\ell-1}R_{\ell+1,i-1}) =\hat{\sigma}_{\pa_{i}}(\tilde{Q}_{\ell,i})\hat{\sigma}_{\pa_{i}}(A_{0})+
   \hat{\sigma}_{\pa_{i}}(A_{0})\hat{\sigma}_{\pa_{i}}(\tilde{Q}_{\ell,i}).
\label{mp.6}\end{equation}
As pointed out in \cite{Mazzeo-Melrose}, this is 
solvable with $\widetilde{Q}_{\ell,i}$ formally self-adjoint as $\hat{\sigma}_{\pa_{i}}(A_{0})^{2}$ is positive and \eqref{mp.6} is the linearization of
the square root equation
\begin{equation}
( \hat{\sigma}_{\pa_{i}}(A_{0})+ \hat{\sigma}_{\pa_{i}}(\tilde{Q}_{\ell,i}))^{2}= 
\hat{\sigma}_{\pa_{i}}(x_{i}^{-\ell-1}R_{\ell+1,i-1})+ \hat{\sigma}_{\pa_{i}}(A_{0})^{2}.
\label{mp.7}\end{equation}
 Thus, we can find $Q_{\ell,i}$ such  that \eqref{mp.6} satisfied.  Replacing $Q_{\ell,i}$ by $\frac{Q_{\ell,i}^*+ Q_{\ell,i}}{2}$ if necessary, we can assume furthermore that $Q_{\ell,i}$ is formally self-adjoint.  Thus, taking $A_{\ell+1}=A_{\ell}+\sum_{i=0}^{k}Q_{\ell,i}$ insures that  
$A_{\ell+1}= A_{\ell+1}^{*}$ and 
\begin{equation}
   C+B-A^{2}_{\ell+1}\in x^{\ell+2}\Psi^{-\ell-2}_{\fS}(X).
\label{mp.9}\end{equation}
We can then define $A$ as an asymptotic sum specified by the $A_{\ell}$.

\end{proof}

\begin{theorem}
Any element $P\in \Psi^{0}_{\fS}(X;E_{1},E_{2})$ defines a bounded linear operator from $\mathcal{H}_{1}=L^{2}(X;E_{1})$ to
$\mathcal{H}_{2}=L^{2}(X;E_{2})$ with $L^{2}$-norms defined
by a positive $\fS$-density on $X$ and  Hermitian metrics on $E_{1}$ and $E_{2}$.  Furthermore, the map
$\Psi^{0}_{\fS}(X;E_{1},E_{2})\to \mathcal{L}(\mathcal{H}_{1},\mathcal{H}_{2})$ is continuous.  
\label{mp.10}\end{theorem}
\begin{proof}

Considering a local trivialization if necessary, we can assume that $E_{1}=E_{2}=\underline{\bbC}$ and $\cH_{1}=\cH_{2}=L^{2}(X)$.  Then
$B=-P^{*}P\in \Psi^{0}_{\fS}(X)$ is formally self-adjoint.  By the previous proposition, there exists
$C>0$ and $A\in \Psi^{0}_{\fS}(X)$ formally self-adjoint such that 
\begin{equation}
   C-P^{*}P= A^{*}A+ R
\label{mp.11}\end{equation} 
for some $R\in x^{\infty}\Psi^{-\infty}_{\fS}(X)$.  Thus, given $u\in \dot{\mathcal{C}}^{\infty}(X)$, we
have
\begin{equation}
 \begin{aligned}
  \| Pu\|^{2} & = C\| u\|^{2} - \|Au\|^{2} - \langle u,Ru\rangle \\
    &\le C\|u\|^{2} + |\;\langle u,Ru\rangle| \le C' \| u\|^{2},
 \end{aligned} 
\label{mp.12}\end{equation}
where the fact elements of $\dot{\Psi}^{-\infty}_{\fS}(X)$ are in $\mathcal{L}(\cH)$ has been
used.  Thus, there is a well-defined linear map 
\begin{equation}
    \Psi^{0}_{\fS}(X)\to \cL(\cH).
\label{mp.12b}\end{equation}
Since the map
\[
         \Psi^{0}_{\fS}(X)\ni A \mapsto \langle u, Av\rangle_{\cH} = K_{A}( \pi_{L}^{*}(u \nu_{\pi})\otimes \pi_{R}^{*}(v))
\]
is continuous for all $u,v\in \dot{\cC}^{\infty}(X)$,  where $\nu_{\pi}$ is the $\fS$-density used to define the $L^{2}$-norm,  we see that the 
graph of the linear map \eqref{mp.12b} is closed with the respect to the topology induced by the norms $A\mapsto |\langle u,Av\rangle|$.  Since this topology is weaker than the norm topology, this means the graph of this map is also closed when we use the norm topology on $\cL(\cH)$.    The map \eqref{mp.12b} is
therefore continuous by the closed graph theorem.

\end{proof}

There is a similar result for suspended $\fS$-operators.  Let $V$ be a Euclidean vector space and let $g_{V}$ be the corresponding Euclidean metric. 
On the manifold with fibred corners $\overline{V}\times X$, consider the $\varpi$-metric
\[
           g_{\varpi}= \pr_{1}^{*} g_{V}+ \pr_{2}^{*}g_{\pi}
\]
where $\pr_{1}: \overline{V}\times X\to \overline{V}$ and $\pr_{2}: \overline{V}\times X\to X$ are the projections on the first and second factors
respectively. 

\begin{corollary}
Any element $P\in \Psi^{0}_{\fS-\sus(V)}(X;E_{1},E_{2})$ defines a bounded linear operator from 
$\cH_{1}= L^{2}_{g_{\varpi}}(\overline{V}\times X;E_{1})$ to $\cH_{2}= L^{2}_{g_{\varpi}}(\overline{V}\times X;E_{2})$ with $L^{2}$-norm defined by a
volume form $dg_{\varpi}$ and Hermitian metrics on $E_{1}$ and $E_{2}$.  Furthermore, the map 
\[
         \Psi^{0}_{\fS-\sus(V)}(X;E_{1},E_{2})\to \cL(\cH_{1},\cH_{2})
\]
is continuous.  
\label{mps.1}\end{corollary}
\begin{proof}
Since our proof of Theorem~\ref{mp.10} is by induction on the depth of $X$ and since the inductive step is not yet completed, we cannot at this
stage simply use the statement of Theorem~\ref{mp.10} for the manifold with fibred corners $\overline{V}\times X$ to obtain the result.  Instead,
consider the Fourier transform of $P$,
\[
         \Upsilon \mapsto \widehat{P}(\Upsilon)\in \Psi^{0}_{\fS}(X;E_{1},E_{2}), \quad \Upsilon\in V^{*}.
\]
By Theorem~\ref{mp.10}, we know that for each $\Upsilon\in V^{*}$, the operator $\widehat{P}(\Upsilon)$ induces a continuous linear map
\[
   \widehat{P}(\Upsilon): L^{2}_{g_{\pi}}(X;E_{1}) \to L^{2}_{g_{\pi}}(X;E_{2}).
\]
Let $g_{V^{*}}$ be the metric on $V^{*}$ which is dual to $g_V$ and let
\[
    g_{\varpi^{*}}= \pr_{1}^{*} g_{V^{*}}+ \pr_{2}^{*}g_{\pi}
\] 
be the corresponding metric on $V^{*}\times X$.  Since the Fourier transform induces an isomorphism of Hilbert spaces
\[
     \cF_{i}:   L^{2}_{g_{\varpi}}(V\times X;E_{i})\to L^{2}_{g_{\varpi^{*}}}(V^{*}\times X;E_{i}),
\]
we conclude from \eqref{sus.24} and Theorem~\ref{mp.10} that $P\in\cL(\cH_{1},\cH_{2})$.  The continuity of the map
$\Psi^{0}_{\fS-\sus(V)}(X;E_{1},E_{2})\to \cL(\cH_{1},\cH_{2})$ can be proved in the same way as before.

\end{proof}

As a family of suspended operators, the symbol $\sigma_{\pa_{i}}(P)$ of an operator $P\in \Psi^{0}_{\fS}(X;E,F)$ will act on the Banach space
$L^{2}_{g_{\pi}}({}^{\pi}NH_{i}/S_{i};E)$ obtained by taking the closure of the space of Schwartz sections $\cS({}^{\pi}NH_{i};E)$ with respect to the
norm
\begin{equation}
  \| f\|_{L^{2}_{g_{\pi}}({}^{\pi}NH_{i}/S_{i};E)}= \sup_{s\in S_{i}} 
    \| \left. f\right.|_{\phi_{i}^{-1}(s)}\|_{L^{2}_{g_{\pi}}(\phi_{i}^{-1}(s);E)}, \quad f\in \cS({}^{\pi}NH_{i};E),
\label{lds.1}\end{equation} 
where $\phi_{i}= \pi_{i}\circ\nu_{i}: {}^{\pi}NH_{i}\to S_{i}$ and $\nu_{i}: {}^{\pi}NH_{i}\to H_{i}$ is the vector bundle projection.  On each fibre
of $\phi_{i}$, the $L^{2}$-norm of a section of $E$ is specified by a choice of Hermitian metric on $E$ and the natural density induced by $g_{\pi}$. 
Thus, from Corollary~\ref{mps.1}, we see that the symbol $\sigma_{\pa_{i}}(P)$ of an operator $P\in \Psi^{0}_{\fS}(X;E,F)$ naturally induce a
continuous linear map
\begin{equation}
   \sigma_{\pa_{i}}(P): L^{2}_{g_{\pi}}({}^{\pi}NH_{i}/S_i;E) \to L^{2}_{g_{\pi}}({}^{\pi}NH_{i}/S_i;F).
   \label{lds.2}\end{equation}    
Notice that the Banach space $L^{2}_{g_{\pi}}({}^{\pi}NH_{i}/S_{i};E)$ also has a natural structure of $\cC^{0}(S_{i})$-Hilbert module.

To complete the inductive step necessary to the proof of Proposition~\ref{mp.1}, we can now use this fact with Corollary~\ref{fp.17}  to construct the
unique positive square root of   
the operator $C+B$ in Proposition~\ref{mp.1} and its suspended versions.

\begin{corollary}
Given a formally self-adjoint operator $B\in \Psi^{0}_{\fS-\sus(V)}(X)$, there exists a positive constant $C$
 such that $C+B$ is invertible and has a well-defined formally self-adjoint positive definite square root in $\Psi^{0}_{\fS-\sus(V)}(X)$.  
 \label{raccar.1}\end{corollary}
\begin{proof}
From Corollary~\ref{mps.1}, we know that $B$ gives a bounded operator
\begin{equation}
   B: L^{2}_{g_{\varpi}}(V\times X)\to L^{2}_{g_{\varpi}}(V\times X).
\label{fp.18}\end{equation}
Thus, taking $C$ big enough, we can define the square root of $C+B$ 
as a bounded operator by
\begin{equation}
C^{\frac{1}{2}}(1+ \frac{B}{C})^{\frac{1}{2}}=
 C^{\frac12} \sum_{j=0}^{\infty} \frac{f^{(j)}(0)}{j!} \left(\frac{B}{C}\right)^{j}
\label{fp.19}\end{equation}
using the power series of $f(x)=(1+x)^{\frac{1}{2}}$ at $x=0$.  
\begin{figure}[h]
\setlength{\unitlength}{1cm}
\begin{picture}(10,4)
\thicklines
\put(5,0){\vector(0,1){4}}
\put(0,2){\vector(1,0){10}}
\put(8,2){\circle{2}}
\put(8.7,2){\vector(0,1){0.1}}
\put(8.9,2.5){$\Gamma$}
\put(8,2){\circle*{0.1}}
\put(8.1,2.1){C}
\put(10,1.7){x}
\put(4.7,3.8){y}
\end{picture}
\caption{}\label{figcont}
\end{figure}
To see it
is an element of $\Psi^{0}_{\fS-\sus(V)}(X)$, we can use the alternative representation in terms of a contour integral
\begin{equation}
  (C+B)^{\frac{1}{2}}= \frac{1}{2\pi i} \int_{\Gamma}
    \lambda^{\frac12}(\lambda- (C+B))^{-1} d\lambda
\label{fp.20}\end{equation} 
where $\Gamma$ is an anti-clockwise circle centered at $C$ and radius
$r$ such that $\|B\|_{\cL(\cH)}< r< C$ (see Figure~\ref{figcont}).

This way, the family $(\lambda- (C+B))$ is invertible for all $\lambda\in \Gamma$ and the square root of
$\lambda$ is well-defined along $\Gamma$.  By Corollary~\ref{fp.17}, the family
$(\lambda- (C+B))^{-1}$ is a smooth family in $\Psi^{0}_{\fS-\sus(V)}(X)$ for $\lambda\in \Gamma$.  Thus, we see from \eqref{fp.20} that
$(C+B)^{\frac{1}{2}}$ is in
$\Psi^{0}_{\fS-\sus(V)}(X)$ as well.  
\end{proof}

\begin{theorem}
For $\delta>0$, an operator $A\in \Psi^{-\delta}_{\fS}(X;E)$ is compact when acting on 
$\cH= L^{2}_{g_{\pi}}(X;E)$ if and only if
$\sigma_{\pa_{j}}(A)=0$ for all $j\in\{1,\ldots,k\}$.  In particular, a polyhomogeneous $\fS$-operator $A\in\Psi^{0}_{\fS-\phg}(X;E)$ of order zero is
compact when acting on $L^{2}_{g_{\pi}}(X;E)$ if and only if
$A\in x\Psi^{-1}_{\fS-\phg}(X;E)$.
\label{compactness}\end{theorem}
\begin{proof}
Without loss of generality, we can assume $E=\underline{\bbC}$ is the trivial vector bundle.  
By definition, the space of compact operators $\cK(\cH)$ is the closure in
$\cL(\cH)$ of operators of finite ranks.  Clearly, since $\dot{\mathcal{C}}^{\infty}(X)$ is dense in $L^{2}_{g_{\pi}}(X)$, we can as well 
define $\cK(\cH)$ as the closure of finite rank operators represented by an
element of $\dot{\Psi}^{-\infty}_{\fS}(X)$.  These operators of finite rank
are certainly dense in $\dot{\Psi}_{\fS}^{-\infty}(X)$.  Thus, $\cK(\cH)$ is
given by the closure of $\dot{\Psi}_{\fS}^{-\infty}(X)$ in $\cL(\cH)$.  Since the 
map $\Psi^{0}_{\fS}(X)\to \cL(\cH)$ is continuous, we conclude that the closure
of $\dot{\Psi}^{-\infty}_{\fS}(X)$ in $\Psi^{-\delta}_{\fS}(X)$, namely,
$x\Psi^{-\delta}_{\fS}(X)$, is included in $\cK(\cH)$.

Conversely, let $A\in \Psi^{-\delta}_{\fS}(X)$ be a compact operator.  Suppose for a contradiction that $\sigma_{\pa_{i}}(A)\ne 0$ for some
$i\in\{1,\ldots,k\}$.   This means that we can find $y_{i}\in S_{i}$ and a function 
$f\in \CI_{c}({}^{\pi}N_{y_{i}}H_{i})$ such that
\begin{equation}
  \left.  \sigma_{\pa_{i}}(A)\right|_{y_{i}} f \ne 0.   
\label{compact.1}\end{equation} 

Without loss of generality, we can assume in fact that $y_i\in S_i\setminus \pa S_i$.  
Let $\cV$ be a small neighborhood of $y_{i}\in S_{i}$ such that the fibration $\pi_{i}:H_{i}\to S_{i}$ is trivial over $\cV$, namely, there is a
diffeomorphism $\psi: \pi^{-1}_{i}(\cV)\to F_{i}\times \cV$ inducing a commutative diagram
\begin{equation}
\xymatrix{
      \pi^{-1}_{i}(\cV) \ar[rr]^{\psi} \ar[rd]^{\pi_{i}}  &  & \cV \times F_{i}\ar[dl]^{\pr_{L}} \\
           &   \cV  &
}
\label{compact.2}\end{equation} 
where $\pr_{L}: \cV\times F_{i}\to \cV$ is the projection on the left factor.  Let $\iota_{i}: H_{i}\times [0,\epsilon)_{x_{i}}\to X$ be a tubular
neighborhood of $H_{i}$ in $X$ compatible with the boundary defining function $x_{i}$.  Using the diffeomorphism $\psi$,  we can identify 
the open set $\iota_{i}(\pi_{i}^{-1}(\cV)\times [0,\epsilon)_{x_{i}})\subset X$ with the open set
\begin{equation}
   \cV\times F_{i}\times [0,\epsilon)_{x_{i}}.  
\label{compact.3}\end{equation}
Choosing $\cV$ to be smaller if needed, we can assume it is diffeomorphic to an open ball in the Euclidean space.  Let $y$ be a choice of coordinates
on $\cV$ such that the point $y_{i}\in \cV$ corresponds to $y=0$.  On the open set $\cV \times (0, \epsilon)_{x_{i}}$, consider the coordinates
\begin{equation}
     u= \frac{1}{x_{i}}, \quad v= \frac{y}{x_{i}}.  
\label{compact.4}\end{equation}
Considering alternatively $v$ and $u$ as linear coordinates on the vector space 
${}^{\pi_{i}}N_{y_{i}}S_{i}= T_{y_{i}}S_{i}\times \bbR_{u}$, we regard $\cV\times (0,\epsilon)_{x_{i}}$ as an open subset in
${}^{\pi_{i}}N_{y_{i}}S_{i},$ and consequently we can regard $\cU= \cV\times F_{i}\times (0,\epsilon)_{x_{i}}$ as a subset of
${}^{\pi_{i}}N_{y_{i}}S_{i}\times F_{i}= {}^{\pi_{i}}N_{y_{i}}H_{i}$.  For $k\in \bbN_{0}$, consider the new function
\begin{equation}
     f_{k}(u,v,z)= f(u-k, v, z),  \quad z\in F_{i}, 
\label{compact.5}\end{equation}    
obtained by translating $f$ in the $u$ variable.  Since we assume that the support of $f$ is compact, by taking $k$ sufficiently large, we can insure that
the support of $f$ is contained in the open set $\cU$.  In fact, since the operator $\sigma_{\pa_{i}}(A)$ is translation invariant, we will still have
that \eqref{compact.1}  holds after translating $f$ in the $u$ variable, so  without loss of generality, we can assume that the support of $f_{k}$ is
contained in $\cU$ for all $k\in \bbN_{0}$.  Again, by translation invariance of $\sigma_{\pa_{i}}(A)$, we will have that 
\[
     \left.  \sigma_{\pa_{i}}(A)\right|_{y_{i}} f_{k} \ne 0 
\]   
for all $k\in \bbN_{0}$.  Since the function $f_{k}$ is supported in $\cU$, we can also regard it as a function on $X$.  Let $\chi\in \CI_{c}(
\cV\times F_{i} \times [0,\epsilon)_{x_{i}})$ be a cut-off function such that $\chi\equiv 1$ in a neighborhood of $ \{y_{i}\}\times F_{i}\times
\{0\}$.  Thus, if we consider the operator $P = \chi A\in \Psi^{-\delta}_{\fS}(X)$, $P$ will also obviously be compact, and we will have that 
\[
    \left.  \sigma_{\pa_{i}}(P)\right|_{y_{i}} =  \left.  \sigma_{\pa_{i}}(A)\right|_{y_{i}}.
\]
Now, thanks to the cut-off function $\chi$, the action of $P$ on $f\in \CI_{c}(U)\subset \CI(X)$ is given by:
\[
P f_{k}(u,v,z)= \int_{U} K_{P}( u, v,u',v',z,z') f_{k}(u', v', z') du'dv' dz',
\]
where the integral is in the sense of distributions.  
Similarly, the action of $\left. \sigma_{\pa_{i}}(P)\right|_{y_{i}}$ can be described by
\[
\left.\sigma_{\pa_{i}}(P)\right|_{y_{i}} f_{k}(u,v,z)= \int_{U} K_{\left.\sigma_{\pa_{i}}(P)\right|_{y_{i}}}( u, v,u',v',z,z') f_{k}(u', v', z')
du'dv' dz'
\]
Since as a function on $\cU\subset X$, the support of the function $f_{k}$ is uniformly approaching the fibre $\pi_{i}^{-1}(y_{i})\subset H_{i}$ as
$k\to +\infty$, we see from the 
definition of the normal operator that we must have that as $k$ tends to infinity, 
\begin{equation}
         Pf_{k}- \sigma_{\pa_{i}}(P)f_{k}\to 0
\label{compact.6}\end{equation}
in the $L^{2}$-norm defined by the $\fS$-metric
\[
    g+ du^{2}+ dv^{2} + g_{F_{i}},
\]    
where $g_{F_{i}}$ is a choice of $\fS$-metric on $F_{i}$.
By translation invariance of this metric and of $\left.\sigma_{\pa_{i}}(P)\right|_{y_{i}}$, we have that,
on ${}^{\pi_{i}}N_{y_{i}}H_{i}$,  
\[
   \| \left.\sigma_{\pa_{i}}(P)\right|_{y_{i}} f_{k}\|_{L^{2}}=  \| \left.\sigma_{\pa_{i}}(P)\right|_{y_{i}} f\|_{L^{2}}\ne 0 
\]
If we restrict $\left.\sigma_{\pa_{i}}(P)\right|_{y_{i}} f_{k}$ to   $\cU$, we still clearly have that  
\[
   \lim_{k\to \infty}  \| \left.\sigma_{\pa_{i}}(P)\right|_{y_{i}} f_{k}\|_{L^{2}(\cU)} = \| \left.\sigma_{\pa_{i}}(P)\right|_{y_{i}}
f_{}\|_{L^{2}({}^{\pi_{i}}N_{y_{i}}S_{i})} \ne 0.
\]
On the other hand, $\left.\sigma_{\pa_{i}}(P)\right|_{y_{i}} f_{k}$ being moved to infinity as $k\to \infty$, we see that it converges pointwise to
zero everywhere on $\cU$, so that the sequence $\sigma_{\pa_{i}}(P)f_{k}$ cannot converge in $L^{2}$.  We conclude from \eqref{compact.6}
that the sequence $Pf_{k}$ also fails to converge in $L^{2}$.  Since by translation invariance of the metric, the sequence $f_{k}$ is bounded in
$L^{2}$,  this contradicts the fact $P$ is a compact operator.  To avoid a contradiction, we must conclude that $\sigma_{\pa_{i}}(A)=0$ for all
$i\in\{1,\ldots,k\}$, which completes the proof.     
\end{proof}

More generally, there are natural Sobolev spaces associated to $\fS$-operators.  
As before, let $g_{\pi}$ be a $\fS$-metric on $X$ and let $E\to X$ be a complex 
vector bundle with a Hermitian metric, so that we have a corresponding space $L^{2}_{g_{\pi}}(X;E)$ of square integrable sections.  For $m>0$, we
define the associated $\fS$-Sobolev space by
\begin{equation}
H^{m}_{\fS}(X;E)= \{ f\in \dot{\cC}^{-\infty}(X;E) \; ; \; Pf \in L^{2}_{g_{\pi}}(X;E) \quad \forall P\in\Psi^{m}_{\fS}(X;E)\},
\label{Sob.1}\end{equation}
while for $m<0$, we define it by
\begin{multline}
H^{m}_{\fS}(X;E)= \{ f\in \dot{\cC}^{-\infty}(X;E) ;  \\
       f= \sum_{i=1}^{N} P_{i}g_{i}, \quad g_{i}\in L^{2}_{g_{\pi}}(X;E),
       P_{i}\in \Psi^{-m}(X;E)  \}.
\label{Sob.2}\end{multline}
If $V$ is a Euclidean vector space, we define the corresponding $V$-suspended $\fS$-Sobolev space by 
\begin{equation}
  H^{m}_{\fS-\sus(V)}(X;E)= H^{m}_{\fS}(\overline{V}\times X;E).
\label{Sob.3}\end{equation}
These spaces can be given the structure of a Hilbert space using fully elliptic operators.  More
precisely, for $m>0$, let $A_{m}\in \Psi^{\frac{m}{2}}_{\fS}(X;E)$ be a choice of elliptic $\fS$-operator and consider the formally self-adjoint
operator $D_{m}\in \Psi^{m}_{\fS}(X;E)$ defined by
\begin{equation}
   D_{m}= A^{*}_m A_m + \Id_{E}.
\label{Sob.4}\end{equation} 
\begin{lemma}
  For $m>0$, the operator $D_{m}$ is fully elliptic and invertible.  In particular, its inverse
    $D_{-m}:= (D_m)^{-1}$ is an element of $\Psi^{-m}_{\fS}(X;E)$.
\label{Sob.5}\end{lemma}
\begin{proof}
Let $H_{1},\ldots, H_{k}$ be the boundary hypersurfaces of $X$ and suppose that they are labelled  in such a way that
\[
      H_{i}< H_{j} \; \Longrightarrow \;  i<j.
\] 
We will first prove by induction on $i\in \{1,\ldots, k\}$ starting with $i=k$ that 
$\sigma_{\pa_{i}}(D_{m})$ is fully elliptic and  invertible.  For $i=k$, the fibres of the fibration $\pi_{k}: H_{k}\to S_{k}$ are closed manifolds,
so that in this case, $\sigma_{\pa_{k}}(D_{m})$ is automatically fully elliptic since it is  elliptic.  Thus, for $i\in \{1,\ldots, k\}$, the
inductive step we need to show is that if $\sigma_{\pa_{i}}(D_{m})$ is fully elliptic, then it is invertible.  To see this, 
fix $s\in S_{i}$ and consider the ${}^{\pi}N_{s}S_{i}$-suspended operator $\sigma_{\pa_{i}}(A_{m})_{s}$ above $s$.  For a fixed $\Upsilon\in
{}^{\pi}N_{s}^{*}S_{i}$, consider the operator
\[
         Q=  \widehat{\sigma_{\pa_{i}}(D_{m})}(\Upsilon)\in \Psi^{m}_{\fS}(\pi_{i}^{-1}(s);E). 
\]
Thus, if $B= \widehat{\sigma_{\pa_{i}}(A_{m})}(\Upsilon)$, we have that $Q=B^{*}B+\Id_{E}$.
By Proposition~\ref{fp.1}, if $Qu=0$, then $u\in \dot{\cC}^{\infty}(\pi_{i}^{-1}(s);E)$.  Thus, we have      
in particular
\begin{equation}
\begin{aligned}
 Qu=0 & \Longrightarrow \; \langle u, B^{*}Bu +u\rangle_{L^{2}}, \\
            & \Longrightarrow \; \| Bu\|^{2}_{L^{2}} + \| u\|^{2}_{L^{2}}=0, \\
            & \Longrightarrow \; u\equiv 0.
\end{aligned}
\label{Sob.6}\end{equation}
Thus, since $Q$ is formally self-adjoint, we have that $\ker Q= \ker Q^{*}= \{0\}$, so that $Q$ is invertible.  Since $\Upsilon\in N^{*}_{s}S_{i}$ was
arbitrary, this means that $\sigma_{\pa_{i}}(D_{m})_{s}$ is invertible.  Thus, since $s\in S_{i}$ was arbitrary, this means  that
$\sigma_{\pa_{i}}(D_{m})$ is invertible, which completes the inductive step.  

With this argument, we see $D_{m}$ is fully elliptic.  In particular, by Proposition~\ref{fp.1},  if $D_{m}u=0$, then $u\in \dot{\cC}^{\infty}(X;E)$. 
We can then show $D_{m}$ is invertible using a similar argument as in \eqref{Sob.6}, which completes the proof.

\end{proof}

Using the operator $D_{m}$ with $D_{m}= (D_{-m})^{-1}$ for $m<0$ and $D_{0}=\Id_{E}$, we can then define an inner product on $H^{m}_{\fS}(X;E)$ by
\begin{equation}
  \langle u,v\rangle_{H^{m}_{\fS}(X;E)} = \langle D_{m}u, D_{m}v\rangle_{L^{2}_{g_{\pi}}(X;E)},
\label{Sob.7}\end{equation}
with corresponding norm
\begin{equation}
  \| u \|_{H^{m}_{\fS}(X;E)} = \|D_{m}u \|_{L^{2}_{g_{\pi}}(X;E)}.
\label{Sob.8}\end{equation}
Using Theorem~\ref{mp.10}, it is straightforward to check $H^{m}_{\fS}(X;E)$ is precisely the closure of $\dot{\cC}^{\infty}(X;E)$ with respect to
this norm.

\begin{proposition}
Any $\fS$-pseudodifferential operator $P\in \Psi^{m}_{\fS}(X;E,F)$ induces
a bounded linear map
\[
       P: x^{\ell}H^{p}_{\fS}(X;E)\to x^{\ell}H_{\fS}^{p-m}(X;F)
\]
for $p,\ell\in \bbR$.
\label{fp.13}\end{proposition}
\begin{proof}
Thinking of $E$ and $F$ as subbundles of a bigger bundle $H$, we reduce
to the case where $E=F$.  The result then follows from Theorem~\ref{mp.10}
by noticing
\[
    P= x^{\ell}D_{m-p}\widetilde{P} D_{p}x^{-\ell} \quad \mbox{with} \; \widetilde{P}= D_{p-m}x^{-\ell}Px^{\ell}D_{-p}  \in \Psi^{0}_{\fS}(X;E).
\]
\end{proof}
In particular, we conclude from Proposition~\ref{fp.13} that for all $\ell\in\bbR$ and $p\in\bbR$, the operator $D_{m}$ induces an isomorphism
\begin{equation}
D_{m}: x^{\ell}H^{p}_{\fS}(X;E)\to x^{\ell}H_{\fS}^{p-m}(X;F)\label{fp.12}\end{equation} 

\begin{proposition}
We have a continuous inclusion $x^{\ell}H_{\fS}^{m}(X;E)\subset x^{\ell'}H^{m'}_{\fS}(X;E)$ if and only if $\ell\ge \ell'$ and $m \ge m'$.  
The inclusion is compact if and only if $\ell>\ell'$ and $m>m'$.
\label{fp.14}\end{proposition}
\begin{proof}
The fact that these are continuous inclusions follows from the isomorphism \eqref{fp.12} and Proposition~\ref{fp.13}.  The statement about compactness follows by using the
isomorphism \eqref{fp.12} and the fact that for $\epsilon>0$, the operator
$x^{\epsilon}D_{-\epsilon}\in x^{\epsilon}\Psi^{-\epsilon}_{\fS}(X;E)$ is a compact operator from $L^{2}_{g_{\pi}}(X;E)$
to itself.
\end{proof}

By the parametrix construction of Proposition~\ref{fp.1} as well as 
Proposition~\ref{fp.13} and Proposition~\ref{fp.14}, an operator
$P\in\Psi_{\fS}^{m}(X;E,F)$ is Fredholm 
as an operator 
\begin{equation}
  P: x^{\ell}H_{\fS}^{p+m}(X;E)\to x^{\ell}H_{\fS}^{p}(X;F)
\label{fp.16}\end{equation} 
whenever it is fully elliptic.  When $P$ is polyhomogeneous, it is also possible to establish the converse. 

\begin{theorem}
An operator $P\in \Psi^{m}_{\fS-\phg}(X;E,F)$ induces a Fredholm operator
\[
     P: x^{\ell}H_{\fS}^{p+m}(X;E)\to x^{\ell}H_{\fS}^{p}(X;F)
\]
if and only if it is 
fully elliptic.  
\label{fp.15}\end{theorem}
\begin{proof}
We will follow the approach of \cite[Theorem~4]{LMN2000}.
First, by considering instead the operator $\tilde{P}= x^{-\ell}D_{p}PD_{-p-m}x^{\ell}$,
we can assume that $P$ is of order $0$ and is seen as a bounded operator
\[
       P: L^{2}_{g_{\pi}}(X;E)\to L^{2}_{g_{\pi}}(X;F).
\]
Furthermore, by considering instead the operator
\[
  \left( \begin{array}{cc}
    0 & P^{*} \\
    P & 0
   \end{array} \right) : L^{2}_{g_{\pi}}(X;E\oplus F)\to L^{2}_{g_{\pi}}(X;E\oplus F),
\] 
we can reduce to the case $E=F$ with $P$ self-adjoint.  By Theorem~\ref{mp.10}, we have a continuous linear map
\[
    \iota: \Psi^{0}_{\fS-\phg}(X;E) \to \cL(\cH,\cH),
\]
where $\cH= L^{2}_{g_{\pi}}(X;E)$.  Let $\cP^{0}_{\fS-\phg}(X;E)$ be the image of this map and $\overline{\cP}^{0}_{\fS-\phg}(X;E)$ its closure in
$\cL(\cH,\cH)$.
Now, the principal symbol induces a continuous linear map
\[
  \sigma_{0}: \Psi^{0}_{\fS-\phg}(X;E)\to \CI({}^{\pi}S^{*}X;\hom(E)).
\]
Using instead the $\cC^{0}$-topology on $\CI({}^{\pi}S^{*}X;\hom(E))$, this extends to a homomorphism of $C^*$-algebras
\[
  \overline{\sigma}_{0}: \overline{\cP}^{0}_{\fS-\phg}(X;E)\to \cC^{0}({}^{\pi}S^{*}X;\hom(E)).
\]
Similarly, the symbol map $\sigma_{\pa_{i}}$ induces a continuous linear map
\[
   \overline{\sigma}_{\pa_{i}}: \overline{\cP}^{0}_{\fS-\phg}(X;E) \to \overline{\cP}^0_{\ff_{\pi_i}-\phg}(H_i;E),
\]
where $\overline{\cP}^0_{\ff_{\pi_i}-\phg}(H_i;E)$ is the closure of $\Psi^0_{\ff_{\pi_i}-\phg}(H_i;E)$ in $\cL(\cH_{i},\cH_{i})$ with $\cH_{i}$ the Banach space $L^{2}_{g_{\pi}}({}^{\pi}NH_{i}/S_{i};E)$ introduced in \eqref{lds.2}.
By Theorem~\ref{compactness}, this induces an injective map
\begin{equation}
(\overline{\sigma}_{0}, \bigoplus_{i=1}^{k} \overline{\sigma}_{\pa_{i}}):
\overline{\cP}^{0}_{\fS-\phg}(X;E)/\cK\hookrightarrow \cC^{0}({}^{\pi}S^{*}X;\hom(E))\oplus ( \bigoplus_{i=1}^{k} \overline{\cP}^0_{\ff_{\pi_i}-\phg}(X;E))
\label{pfc.1}\end{equation}
where $\cK\subset \cL(\cH,\cH)$ is the subspace of compact operators.  Since this is an injective map of $C^{*}$-algebras mapping the identity to the
identity, it is a standard fact (see for instance Proposition~1.3.10 in \cite{Dixmier}) that an element of $\overline{\cP}^{0}_{\fS-\phg}(X;E)/\cK$ is
invertible if and only if its image under the map \eqref{pfc.1} is invertible.  Since a bounded operator in $\cL(\cH,\cH)$ is Fredholm if and only if
it is invertible in $\cL(\cH,\cH)/\cK$, the result follows.  
\end{proof}

\section{The semiclassical $\fS$-calculus}\label{ac.0}

Consider the manifold with corner $X^{2}_{\pi}\times [0,1]_{\epsilon}$ where $\epsilon$ should be considered as a semiclassical parameter.  The
\textbf{semiclassical $\pi$-double space} is obtained by blowing up the $p$-submanifold $\Delta_{\pi}\times \{0\}$,
\begin{equation}
  X^{2}_{\pi-\ad}= [X^{2}_{\pi}\times [0,1]_{\epsilon}; \Delta_{\pi}\times \{0\}]
\label{ac.1}\end{equation}
with blow-down map
\begin{equation}
  \beta_{\ad}: X^{2}_{\pi-\ad}\to X^{2}_{\pi}\times [0,1]_{\epsilon}.
\label{ac.2}\end{equation}
See Figure~\ref{addouble} for a picture of the semiclassical $\pi$-double space when $X$ is a manifold with boundary.
We denote the `new' boundary face obtained via this blow-up by
\begin{equation}
     \ff_{0}= \beta^{-1}_{\ad}(\Delta_{\pi}\times \{0\})\subset X^{2}_{\pi-\ad}.
\label{ac.3}\end{equation}
We also denote by 
\begin{equation}
  T_{\epsilon=0}= \overline{ \beta_{\ad}^{-1}( X^2_{\pi}\times \{0\}\setminus (\Delta_\pi\times \{0\}) }
\end{equation}
the lift of the `old' face $X^{2}_{\pi}\times \{0\}$ to $X^{2}_{\pi-\ad}$.

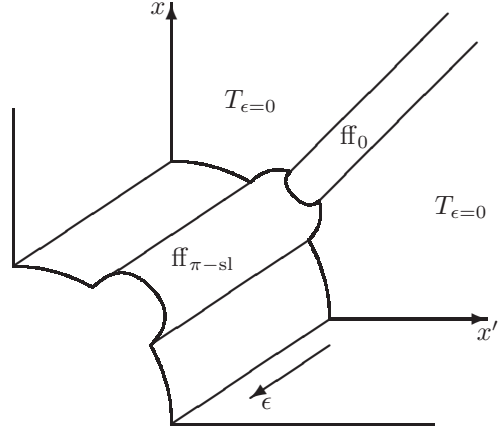
\begin{figure}[h]
\setlength{\unitlength}{0.7cm}
\begin{picture}(10,10)(0,-1)
\thicklines
\qbezier(0,3)(0.78,3)(1.5,2.6)
\qbezier(3,0)(3,0.78)(2.6,1.5)
\qbezier(1.5,2.6)(2.05,3.15)(2.6,2.6)
\qbezier(2.6,1.5)(3.15,2.05)(2.6,2.6)

\qbezier(3,5)(3.78,5)(4.5,4.6)
\qbezier(6,2)(6,2.78)(5.6,3.5)
\qbezier(4.5,4.6)(4.82,4.92)(5.25,4.8)
\qbezier(5.6,3.5)(5.92,3.82)(5.8,4.25)
\qbezier(5.25,4.8)(5.05,4.6)(5.325,4.325)
\qbezier(5.8,4.25)(5.6,4.05)(5.325,4.325)

\put(0,3){\line(3,2){3}}
\put(3,0){\line(3,2){3}}
\put(1.5,2.6){\line(3,2){3}}
\put(2.6,1.5){\line(3,2){3}}
\put(3,3){$\ff_{\pi-\ad}$}

\put(0,3){\line(0,1){3}}
\put(3,5){\vector(0,1){3}}
\put(2.6,7.8){$x$}

\put(3,0){\line(1,0){5}}
\put(6,2){\vector(1,0){3}}
\put(8.8,1.6){$x'$}


\put(6,1.5){\vector(-3,-2){1.5}}
\put(4.7,0.3){$\epsilon$}

\put(5.25,4.8){\line(1,1){3}}
\put(5.8,4.25){\line(1,1){3}}
\put(6.2,5.325){$\ff_{0}$}
\put(4,6){$T_{\epsilon=0}$}
\put(8,4){$T_{\epsilon=0}$}

\end{picture}
\caption{The semiclassical $\pi$-double space}\label{addouble}
\end{figure}

Notice that $\ff_{0}\setminus (\ff_0\cap T_{\epsilon=0})$ is naturally diffeomorphic to $N\Delta_{\pi}\cong {}^{\pi}TX$ and that $\ff_{0}$ is diffeomorphic to the radial
compactification of ${}^{\pi}TX\to X$.  We will also denote the lift
of $\ff_{\pi_{i}}\times [0,1]$ to $X^{2}_{\pi-\ad}$ by
\begin{equation}
  \ff_{\pi_{i}-\ad}= \beta^{-1}_{\ad}(\ff_{\pi_{i}}\times [0,1])\cong 
   [\ff_{\pi_{i}}\times [0,1]; (\Delta_{\pi}\cap \ff_{\pi_{i}})\times \{0\} ]. 
\label{ac.4}\end{equation}
It will be useful to consider the spaces 
\begin{equation}
   \ff_{\pi-\ad} = \bigcup_{i=1}^{k} \ff_{\pi_{i}-\ad}, \quad 
   \ff_{\ad}= \ff_{\pi-\ad}\cup \ff_{0}
\label{ac.5}\end{equation}
as well as the lift of $\Delta_{\pi}\times [0,1]$ to $X^{2}_{\pi-\ad}$,
\begin{equation}
  \Delta_{\ad}= \overline{\beta^{-1}_{\ad}(\Delta_{\pi}\times (0,1])}.
\label{ac.6}\end{equation}
Let also $\ff_{\Delta_{\ad}}= \ff_{\ad}\cup (X^{2}_{\pi}\times \{1\})$ be the
union of all the hypersurfaces of $\pa X^{2}_{\pi-\ad}$ having a non-empty
intersection with $\Delta_{\ad}$.

We can now define the space of \textbf{semiclassical $\fS$-pseudodifferential
operators}  of order $m$ by
\begin{multline}
  \Psi^{m}_{\fS-\ad}(X;E,F)= \left\{ K\in I^{m}(X^{2}_{\pi-\ad},
  \Delta_{\ad}; \beta_{\ad}^{*}p^*_{1}\beta_{\pi}^{*}(\Hom(E,F)\otimes
  \pi_{R}^*({}^{\pi}\Omega)))  \right.   \\
  \left. K  \equiv 0 \; \mbox{at} \quad \pa X^{2}_{\pi-\ad}\setminus \ff_{\Delta_{\ad}}   \right\},
\label{ac.7}\end{multline} 
where $p_{1}: X^{2}_{\pi}\times [0,1]\to X^{2}_{\pi}$ is the projection on the first factor.  Polyhomogeneous semiclassical $\fS$-operators can be defined in a similar way.   

As for $\fS$-pseudodifferential operators, there is a corresponding semiclassical
Lie groupoid
\begin{equation}
      \cG^{(0)}_{\pi-\ad}= \Delta_{\ad}, \quad \cG^{(1)}_{\pi-\ad}=
        \overset{\circ}{X}{}^{2}_{\pi-\ad}\cup \overset{\circ}{\ff}_{\Delta_{\ad}},
\label{ac.8}\end{equation}
where $\overset{\circ}{\ff}_{\Delta_{\ad}}= \ff_{\Delta_{\ad}}\setminus \pa\ff_{\Delta_{\ad}}$ is the interior of $\ff_{\Delta_{\ad}}$ as a subset of $\pa X^{2}_{\pi-\ad}$.  
Clearly, $\Delta_{\ad}$ is naturally identified with $X\times [0,1]$.  Under
this identification, we define the domain and range of $\alpha\in \cG^{(1)}_{\pi-\ad}$ with $p_{2}\circ \beta_{\ad}(\alpha)=\epsilon$ and
$\beta_{\pi}\circ p_{1}\circ \beta_{\ad}(\alpha)= (x_{1},x_{2})\in X^{2}$ by
\begin{equation}
   d(\alpha)= (x_{2},\epsilon), \quad r(\alpha)=(x_1,\epsilon),
\label{ac.9}\end{equation}
where $p_2: X^2_{\pi}\times [0,1]\to [0,1]$ is the projection on the second factor.
Since $\cG_{\pi-\ad}$ is a Lie groupoid, any choice of a metric on ${}^\pi TX\times[0,1]$ provides a (smooth) Haar system on $\cG_{\pi-\ad}$ \cite{Paterson}, giving to it the structure of a measured groupoid. As in the proof of Lemma~\ref{ma.2}, observe that $\cG_{\pi-\ad}$ can be written as a disjoint union of
 measurewise amenable groupoids,
\begin{equation}
  \cG_{\pi-\ad} = ({}^\pi TX) \bigsqcup (\overset{\circ}{X}\times \overset{\circ}{X})\times (0,1]_{\epsilon} \bigsqcup_{i=1}^k
(H_i\underset{\pi_i}{\times}{}^{\pi}TS_i\underset{\pi_i}{\times}H_i)|_{G_i}\times(0,1]_{\epsilon}\times\RR,
\end{equation}
where $G_i=H_i\setminus(\cup_{j>i} H_j)$.  Thus, by Lemma~\ref{ma.1}, we conclude that $\cG_{\pi-\ad}$ is measurewise amenable with $\cC^*(\cG_{\pi-\ad})$ nuclear and equal to $\cC^*_r(\cG_{\pi-\ad})$.

In the terminology of \cite{Connes}, $\cG_{\pi-\ad}$ is the tangent groupoid
of $\cG_{\pi}$.
From \cite{ALN} and \cite{NWX}, there is a calculus of pseudodifferential operators associated to this groupoid.  It corresponds
to operators in $\Psi^{*}_{\pi-\ad}(X;E,F)$ with Schwartz kernel 
compactly supported in $\cG^{(1)}_{\pi-\ad}$.  
As for $\cG^{(1)}_{\pi}$, the inverse map $\iota$ and the composition maps
comes from the natural smooth extensions of the corresponding maps on 
the Lie groupoid $\overset{\circ}{X}\times \overset{\circ}{X}\times[0,1]$
with domain and range given by $d(x_{1},x_{2},\epsilon)=(x_{2},\epsilon)$ and  $r(x_{1},x_{2},\epsilon)= (x_{1},\epsilon)$.

There are many symbol maps associated to 
$\Psi^{m}_{\fS-\ad}(X;E,F)$.  There is the obvious one associated to 
conormal distributions.  With the natural identification of $N^{*}\Delta_{\ad}$
with ${}^{\pi}TX\times [0,1]$, which has a canonical volume form, we can
write it as 
\begin{equation}
\xymatrix{
   \Psi^{m}_{\fS-\ad}(X;E,F) \ar[r]^-{\sigma_{m}} & 
   S^{[m]}(N^{*}\Delta_{\ad}; \phi^{*}\Hom(E,F))
   }
\label{ac.11}\end{equation}
where $\phi$ is the composition of the natural maps 
$N^{*}\Delta_{\ad}\to \Delta_{\ad}$ and 
$\Delta_{\ad}=X\times [0,1]\to X$.  This gives a short exact sequence
\begin{equation}
\xymatrix @C=1.3pc{
  0\ar[r] & \Psi^{m-1}_{\fS-\ad}(X;E,F) \ar[r] & 
  \Psi^{m}_{\fS-\ad}(X;E,F)\ar[r]^-{\sigma_{m}} & 
  S^{[m]}(N^{*}\Delta_{\ad}; \phi^{*}\Hom(E,F)) \ar[r] & 0.
}
\label{ac.12}\end{equation}
We say an operator $P\in \Psi^m_{\fS-\ad}(X;E,F)$ is \textbf{elliptic} if its principal symbol $\sigma_{m}(P)$ is invertible.  

Restriction to boundary hypersurfaces of $X^{2}_{\pi-\ad}$ leads to other types of symbols.  Restricting to the hypersurface $\ff_{0}$, we get the short exact sequence
\begin{equation}
\xymatrix @C=2.5pc{
0\ar[r] & \epsilon \Psi^{m}_{\fS-\ad}(X;E,F) \ar[r] & \Psi^{m}_{\fS-\ad}(X;E,F)
  \ar[r]^-{\sigma_{\epsilon=0}} & \Psi^{m}_{\ff_0 }(X;E,F)
  \ar[r] & 0.
  }
\label{ac.14}\end{equation}
On the other hand, restricting to the face
$\ff_{\pi_{i}-\ad}$ for $i\in\{1,\ldots,k\}$, we get the short exact sequence
\begin{equation}
\xymatrix@C=2.3pc{
0\ar[r] & x_{i} \Psi^{m}_{\fS-\ad}(X;E,F)\ar[r] & 
\Psi^{m}_{\fS-\ad}(X;E,F) \ar[r]^-{\sigma_{\ff_{\pi_{i}-\ad}}} &
  \Psi^{m}_{\ff_{\pi_{i}-\ad}}(X;E,F) \ar[r] & 0.  
}
\label{ac.15}\end{equation}
Combining the symbol maps $\sigma_{\ff_{\pi_{i}}}$ and $\sigma_{\epsilon=0}$, that is, restricting to the hypersurface $\ff_{\ad}$, we 
also get the short exact sequence
\begin{equation}
\xymatrix{
 0 \ar[r] & \epsilon x \Psi^{m}_{\fS-\ad}(X;E,F) \ar[r] &
  \Psi^{m}_{\fS-\ad}(X;E,F) \ar[r]^-{\sigma_{\ff_{\ad}}} &
   \Psi^{m}_{\ff_{\ad}}(X;E,F) \ar[r] & 0.
  }
\label{ac.16}\end{equation}
Finally, a symbol of particular importance is obtained by restricting at the face
$X^{2}_{\pi}\times \{1\}$, giving the short exact sequence
\begin{equation}
\xymatrix{
0 \ar[r] & (1-\epsilon)\Psi^{m}_{\fS-\ad}(X;E,F) \ar[r] &
  \Psi^{m}_{\fS-\ad}(X;E,F) \ar[r]^-{\sigma_{\epsilon=1}} & 
   \Psi^{m}_{\fS}(X;E,F) \ar[r] & 0.
   }
  \label{ac.13}\end{equation}
In fact, more generally, for $\epsilon\in (0,1]$, we can restrict $A\in\Psi^{m}_{\fS-\ad}(X;E,F)$  to the hypersurface $X^{2}_{\pi}\times \{\epsilon\}$ to get an operator $A_\epsilon\in \Psi^{m}_{\fS}(X;E,F)$.  This gives us a way of composing semiclassical $\fS$-operators,
\[
          (A\circ B)_\epsilon := A_{\epsilon}\circ B_\epsilon.  
\]
\begin{proposition}
If $E, F$ and $G$ are smooth vector bundles on $X$, then
\[
    \Psi^{m}_{\fS-\ad}(X;F,G)\circ \Psi^{n}_{\fS-\ad}(X;E,F)\subset \Psi^{m+n}_{\fS-\ad}(X;E,G)
\]
and the induced map is continuous with respect to the natural Fr\'echet topology.  Furthermore, the various symbol maps induce composition laws in such a way that they become algebra homomorphisms.  A similar result holds for polyhomogeneous semiclassical $\fS$-operators.
\label{compad.1}\end{proposition}
\begin{proof}
We can employ the same strategy as in the proof of Theorem~\ref{composition} and proceed by induction on the dimension of the manifold with fibred corners.  Notice that the second part of the proof of Theorem~\ref{composition} (starting with Lemma~\ref{com.2}) mostly involve partitions of unity and has a direct generalization to semiclassical $\fS$-operators.  This means the proposition follows from Lemma~\ref{compad.2} below, which is an analog of Lemma~\ref{com.14} for semiclassical $\fS$-operators.
\end{proof}
\begin{lemma}
Suppose that the conclusion of Proposition~\ref{compad.1} holds for the manifold with fibred corners $X$.  Then it also holds for the manifold with fibred corner $\overline{\bbR^p}\times X$ defined just before Lemma~\ref{com.14}.
\label{compad.2}\end{lemma}
\begin{proof}
The proof is similar to the one of Lemma~\ref{com.14}.  To avoid repetition, we will focus on the parts that require changes.  First, without loss of generality, we can assume $E=F=G=\underline{\bbC}$.   Using the Fourier transform on $\bbR^p$, we can describe the action of operators 
$A\in \Psi^m_{\fS-\ad}(\overline{\bbR^p}\times X)$ and $B\in \Psi^n_{\fS-\ad}(\overline{\bbR^p}\times X)$
by
\begin{equation}
\begin{gathered}
 A_{\epsilon} u(t)= \frac{1}{(2\pi \epsilon)^{p}} \int e^{i(t-t')\cdot \frac{\tau}{\epsilon}} a(t;\tau) u(t')dt'd\tau, \\
  B_{\epsilon} u(t)= \frac{1}{(2\pi\epsilon)^{p}} \int e^{i(t-t')\cdot \frac{\tau}{\epsilon}} b(t;\tau) u(t')dt'd\tau. 
\end{gathered} 
\label{compad.3}\end{equation}
Here, $a$ and $b$ are operator-valued symbols,
\begin{equation}
   a\in \CI( \overline{\bbR^{p}}; \Psi^{m}_{\fS-\ad-\sus(p)}(X)), \quad 
   b\in \CI( \overline{\bbR^{p}}; \Psi^{n}_{\fS-\ad-\sus(p)}(X)),
\label{compad.4}\end{equation}
where $\Psi^{\ell}_{\fS-\ad-\sus(p)}(X)$ is the space of $\bbR^{p}$-suspended semiclassical $\fS$-operators of order $\ell$ on $X$, and the variable $\tau\in\bbR^p$ in \eqref{com.10} is seen as the suspension parameter.  Precisely, as for suspended $\fS$-operators, the space $\Psi^{\ell}_{\fS-\ad-\sus(p)}(X)$ can be defined as the subspace of 
$\Psi^{\ell}_{\fS-\ad}(\overline{\bbR^p}\times X)$ consisting of operators that are unchanged by translations in $\bbR^p$.  
These operator-valued symbols are such that 
\begin{equation}
      a\in \CI(\overline{\bbR_{t}^{p}}; \Psi^{m}_{\fS-\ad-\sus(p)}(X)) \; \Longrightarrow \;  D_{t}^{\alpha}D^{\beta}_{\tau}a \in
(1+t^{2})^{-\frac{|\alpha|}{2}} \CI(\overline{\bbR^{p}_{t}}; \Psi^{m-|\beta|}_{\fS-\ad-\sus(p)}(X)), \quad 
  \label{compad.5}\end{equation}
so that the techniques of \cite{Parenti} can be applied.  More precisely, using the change of variable $\xi=\frac{\tau}{\epsilon}$, we deduce from \eqref{par.1}, \eqref{par.2} and \eqref{par.3} that 
\begin{equation}
 A_\epsilon \circ B_{\epsilon} u(t)= \frac{1}{(2\pi \epsilon)^{p}} \int e^{i(t-t')\cdot \frac{\tau}{\epsilon}} c(t,\tau) u(t')dt'd\tau,
  \label{compad.6}\end{equation}
where $c(t,\tau)$ is an operator-valued symbol which for $N\in \bbN$ can be written in the form
 \begin{equation}
   c(t,\tau)= \sum_{|\alpha|< N} \frac{\epsilon^{|\alpha|}}{\alpha !} \pa_{\tau}^{\alpha} a(t,\tau) D_{t}^{\alpha}b(t,\tau) +
    R_{N}(t,\tau)
 \label{compad.7}\end{equation}
 with remainder term $ R_{N}(t,\tau)$ given by
  \begin{equation}
    \sum_{|\alpha|=N} \frac{N\epsilon^{|\alpha|}}{\alpha !} \int_{0}^{1}(1-\lambda)^{N-1} 
   \left(  \frac{1}{(2\pi)^{p}} \int e^{-iz\cdot \zeta}  \pa_{\tau}^{\alpha} a(t,\tau+\epsilon\lambda \zeta) D^{\alpha}_{t}b(t+z,\tau)dzd\zeta \right)
d\lambda.
 \label{compad.8}\end{equation}
 Proceeding as in the proof of Lemma~\ref{com.14}, we can then show that for $N$ even,
 \[
     R_{N}\in \epsilon^N \cS^{m+n-\frac{N}{2}}_{-N}( \bbR^{p}_{t}\times \bbR^{p}_{\tau}; \Psi^{m+n-\frac{N}{2}}_{\fS-\ad}(X))
 \]
 where $ \cS^{k}_{\ell}( \bbR^{p}_{t}\times \bbR^{p}_{\tau}; \Psi^{m+n}_{\fS-\ad}(X))$ is the space of operator-valued symbols $q$ such that for any
Fr\'echet semi-norm $\| \cdot \|$ of $\Psi^{m+n}_{\fS-\ad}(X)$,   
 \[
     \sup_{t,\tau} \left((1+|\tau|^{2})^{\frac{|\beta|-k}{2}} (1+|t|^{2})^{\frac{|\alpha|-\ell}{2}} 
     \|  D_{t}^{\alpha}D_{\tau}^{\beta} q \|  \right)  < \infty \quad \forall\; \alpha,\beta\in \bbN_{0}^{p}.
 \]
Therefore, taking an asymptotic sum as before we can show that 
\begin{equation}
   c \in \CI(\overline{\bbR^{p}_{t}}; \Psi^{m+n}_{\fS-\ad-\sus(p)}(X)) \quad \mbox{with} \; c
   \sim \sum_{\alpha} \frac{\epsilon^{|\alpha|}}{\alpha !} \pa_{\tau}^{\alpha} a(t,\tau) D^{\alpha}_{t}b(t,\tau),
  \label{compad.9}\end{equation} 
 where the notation $\sim$ means that for all $N\in \bbN$, 
\[
       c(t,\tau) - \sum_{\alpha<N} \frac{\epsilon^{|\alpha|}}{\alpha !} \pa_{\tau}^{\alpha} a(t,\tau) D^{\alpha}_{t}b(t,\tau)
        \in \epsilon^N (1+ |t|^{2})^{-\frac{N}{2}}\CI(\overline{\bbR^{p}_{t}}; \Psi^{m+n-N}_{\fS-\ad-\sus(p)}(X)).
\]
In particular, this shows $A\circ B\in\Psi^{m+n}_{\fS-\ad}(\overline{\bbR^p}\times X)$.  To show that the various symbols are compatible with composition, we can proceed as in the proof of Lemma~\ref{com.14}.  
\end{proof}

\begin{definition}
An operator $P\in \Psi^m_{\fS-\ad}(X;E,F)$ is said to be \textbf{elliptic} if $\sigma_m(P)$ is invertible.  It is said to be \textbf{fully elliptic} if it is elliptic and $\sigma_{\epsilon=1}(P)$ is a fully elliptic $\fS$-operator.  
\label{ac.16b}\end{definition}

A natural sub-groupoid of $\cG^{(1)}_{\pi-\ad}$ is obtained by considering the interior of $\ff_{\ad}$ (as a subset of $\pa X^2_{\pi-\ad}$), 
\begin{equation}\label{Tnc_facon_eclatement}
   T\FCX= \overset{\circ}{\ff}_{\ad}=  \ff_{\ad}\setminus \pa\ff_{\ad}, \quad \mbox{where} \; \pa{\ff_{\ad}}:= \ff_{\ad} \cap \overline{  (\pa X^{2}_{\pi-\ad}) \setminus \ff_{\ad} }.
\end{equation}
The groupoid $T\FCX$ also contains $\ff_0 \setminus (\ff_0\cap T_{\epsilon=0})$ as a subgroupoid. 
It inherits from the Lie structure of $\cG_{\pi-\ad}$ the structure of a continuous family groupoid (\cite{Paterson}).  There is also an induced continuous Haar system once a Haar system is fixed on $\cG_{\pi-\ad}$.

As for $\cG_{\pi-\ad}$, the groupoid $T\FCX$ can be written as a disjoint union of
topologically  amenable groupoids,
\begin{equation}
  T\FCX = {}^\pi TX \sqcup_{i=1}^k
(H_i\underset{\pi_i}{\times}{}^{\pi}TS_i\underset{\pi_i}{\times}H_i)|_{G_i}\times(0,1)_{\epsilon}\times\RR,
\end{equation}
where $G_i=H_i\setminus(\cup_{j>i} H_j)$.  Thus, we conclude from Lemma~\ref{ma.1} that $T\FCX$ is measurewise amenable with $\cC^*(T\FCX)$ nuclear and equal to $\cC^*_r(T\FCX)$.

As we will now describe,  the $K$-theory of $T\FCX$ corresponds to the stable homotopy classes of fully elliptic polyhomogeneous $\fS$-operators.  For this purpose, we will restrict our attention to fully elliptic polyhomogeneous $\fS$-operators of order zero.  This is not a serious restriction.  if $P\in \Psi^{m}_{\fS-\phg}(X;E,F)$ is fully elliptic, we can replace it by the fully elliptic operator $P(\Delta_E+1)^{-\frac{m}{2}}\in \Psi^{0}_{\fS-\phg}(X;E,F)$, where $\Delta_E\in \Psi^2_{\fS-\phg}(X;E)$ is some (positive) Laplacian associated to a choice of $\fS$-metric on $X$ and a choice of Hermitian metric on $E$.  

\begin{definition}
Two fully elliptic operators $P_0\in \Psi^0_{\fS-\phg}(X; E_0, E_1)$ and $P_1\in \Psi^0_{\fS-\phg}(X;E_1,F_1)$ are \textbf{homotopic} if they can be connected by a continuous family of fully elliptic polyhomogeneous $\fS$-operators
\[
            P_t\in \Psi^0_{\fS-\phg}(X;E_t,F_t), \quad t\in [0,1].            
\]
We say instead that $P_0$ and $P_1$ are \textbf{stably homotopic} if they become homotopic after the addition to each of them of the identity operator $\Id_H$ acting on the sections of some complex vector bundle $H\to X$.
\label{hcep.1}\end{definition}
Stable homotopies induce an equivalence relation and we denote by $\fe(X)$ the set of fully elliptic operators modulo stable homotopies.  This set is in fact an abelian group with addition given by direct sum and inverse given by the parametrix construction of Proposition~\ref{fp.1}.  It can be identified with the K-theory of a mapping cone.  To see this,  
let us use the notation of the proof of Theorem~\ref{fp.15} and denote by $\cA=\overline{\cP}^{0}_{\fS-\phg}(X)$ the closure of $\Psi^0_{\fS-\phg}(X)$ in $\cL(\cH,\cH)$, where $\cH= L^2_{g_{\pi}}(X)$.  The algebra $\cA$ contains the subalgebra $\cK\subset \cL(\cH,\cH)$ of compact operators so that we can consider the quotient map
\begin{equation}
  q: \cA\to \cA/\cK.
\label{hcep.2}\end{equation}  
The algebra $\cA_0=\cC(X)$ of continuous functions on $X$ is another subalgebra of $\cA$.  Denote also by $q: \cA_0\to \cA/\cK$ the restriction of the quotient map to $\cA_0$.  Let 
\begin{equation}
\Con_q= \{ (a_0, a)\in \cA_0\oplus \cC([0,1);\cA/\cK); \; q(a_0)= a(0)\}
\label{hcep.3}\end{equation}
be the mapping cone of the map $q: \cA_0\to \cA/\cK$. Consider also the mapping cylinder
 \begin{equation}
\Con^+_q= \{ (a_0, a)\in \cA_0\oplus \cC([0,1];\cA/\cK); \; q(a_0)= a(0)\}.
\label{hcep.3b}\end{equation}
By Theorem~\ref{fp.15}, a fully elliptic operator $P\in \Psi^0_{\fS-\phg}(X;E,F)$,  defines a  $K$-class in 
\[
    K_0( \Con^+_q,\Con_q)\cong K_0( \Con_q).
\]
This $K$-class only depends on the stable homotopy class of $P$ so that there is a well-defined group homomorphism
\begin{equation}
\sigma_{\Con_q}: \fe(X)\to K_0( \Con_q).
\label{hcep.4}\end{equation}
\begin{proposition}
The map $\sigma_{\Con_q}$ is a group isomorphism.  
\label{hcep.4b}\end{proposition}
\begin{proof}
This can be seen as a particular case of a result of Savin \cite[Theorem~4]{Savin2005}.  Alternatively,  since Theorem~\ref{fp.15} identifies $\fe(X)$ with the relative $K$-group $K(q)$ associated to the homomorphism $q:\cA_0\to \cA/\cK$ (see for instance \cite{Androulidakis-Skandalis} or \cite{Karoubi} for a definition of $K(q)$), we can follow instead the approach in \cite[Theorem~3.29]{Karoubi}.  This consists in noticing that the map $\sigma_{\Con_q}$ naturally fits into a commutative diagram of exact sequences,
\begin{equation}
\xymatrix@C=0.5cm{
K_1(\cA_0)\ar[r] \ar[d]^{s} & K_1(\cA/\cK) \ar[r]\ar[d]^s & \fe(X) \ar[d]^{\sigma_{\Con_q}}\ar[r] & K_0(\cA_0) \ar[r]\ar[d]^{\Id} & K_0(\cA/\cK) \ar[d]^{\Id}  \\
K_0(\cS\cA_0) \ar[r] & K_0(\cS(\cA/\cA)) \ar[r] & K_0(\Con_q) \ar[r] & K_0(\cA_0) \ar[r] & K_0(\cA\setminus \cK),
}
\label{hcep.4c}\end{equation}  
where the bottom row is the Puppe sequence associated to $q:\cA_0\to \cA/\cK$ and $s$ denotes the suspension isomorphism.  The result then follows by applying the five-lemma to this diagram.

\end{proof}

The group $\fe(X)$ can also be related with the $K$-theory of the groupoid $T\FCX$.  Indeed, 
given a fully elliptic $\fS$-operator $P\in \Psi^0_{\fS-\phg}(X;E,F)$, let $\cP\in \Psi^0_{\fS-\ad-\phg}(X;E,F)$ be a corresponding fully elliptic semiclassical $\fS$-operator such that $\sigma_{\epsilon=1}(\cP)=P$.  The full ellipticity insures that 
$\sigma_{\ff_{\ad}}(\cP)\in \Psi^0_{\ff_{\ad}-\phg}(X;E,F)$ defines a $K$-class in
\[
       K_0(\overline{\cP}^0_{\ff_{\ad}-\phg}(X), \cC^*(T\FCX))\cong K_0( \cC^*(T\FCX)),
\]
where $\cC^*(T\FCX)=\cC^*_r(T\FCX)$ is the reduced $C^*$-algebra of the groupoid $T\FCX$ and $\overline{\cP}^0_{\ff_{\ad}-\phg}(X)$ is the $C^*$-algebra obtained by taking the closure $\Psi^0_{\ff_{\ad}-\phg}(X)$ with respect to the reduced norm of the groupoid $T\FCX$, see for instance \cite[p.641]{LMN2000}. 

This $K$-class only depends on the stable homotopy class of $P$, so that there is in fact a well-defined group homomorphism
\begin{equation}
  \sigma_{\nc}: \fe(X)\to K_0(\cC^*(T\FCX)).
\label{hcep.5b}\end{equation}
\begin{theorem}
The map $\sigma_{\nc}$ in \eqref{hcep.5b} is an isomorphism of abelian groups.
\label{hcep.6}\end{theorem}
\begin{proof}
By Proposition~\ref{hcep.4b},
it suffices to construct a natural identification between $K_0(\cC^*(T\FCX))$ and 
$K_0(\Con_q)$ inducing a commutative diagram
\begin{equation}
\xymatrix{
                &   K_0(\cC^*(T\FCX))\ar[dd]^{\cong} \\
              \fe(X) \ar[ur]^{\sigma_{\nc}} \ar[dr]^{\sigma_{\Con_q}} &  \\
               &   K_0(\Con_q).
}               
\label{hcep.7}\end{equation}
To construct this natural identification, consider the algebra 
\begin{equation}
  (\sigma_0 \oplus \sigma_{\ff_{\ad}})(\Psi^0_{\fS-\ad-\phg}(X))\subset \CI(S(N^* \Delta_{\ad}))\oplus \Psi^0_{\ff_{\ad}-\phg}(X)
\label{hcep.8}\end{equation}
and let $\cB$ be its $C^*$-closure in $\cC(S(N^* \Delta_{\ad}))\oplus \overline{\cP}^0_{\ff_{\ad}-\phg}(X)$.  
The symbol $\sigma_{\epsilon=1}$ restricts to give a map
\begin{equation}
  \sigma_{\epsilon=1}: \cB\to \cC^0(S({}^{\pi}T^* X)) \oplus \left( \bigoplus_{i=1}^k \overline{\cP}^0_{\ff_{\pi_i}-\phg}(H_i) \right),
\label{hcep.9}\end{equation}
where $H_1,\ldots, H_k$ is an exhaustive list of the boundary hypersurfaces of $X$ and $\overline{\cP}^0_{\ff_{\pi_i}-\phg}(H_i)$ is the $C^*$-closure of $\Psi^0_{\ff_{\pi_i}-\phg}(H_i)$ with respect to the reduced norm (see \cite[p.641]{LMN2000}) for the groupoid $\ff_{\pi_i}\cap\overset{\circ}{\ff}_{\pi}$.  There is a natural inclusion
$\iota: \cC(X)\hookrightarrow \cB$.  Let $\cB_0$ be the kernel of the map \eqref{hcep.9} and consider the subalgebra
\[
         \hat{\cB}_0= \{ b\in \cB_0 \quad | \quad \sigma_{\ff_0}(b) \in \cC(X)\}.
\]  
Clearly, there is a natural identification $K_0(\hat{\cB}_0) \cong K_0 (\Con_{\iota})$, where $C_{\iota}$ is the mapping cone of the natural inclusion $\iota: \cC(X)\to \sigma_{\epsilon=1}(\cB)$.   On the other hand, the commutative diagram of short exact sequences
\begin{equation}
\xymatrix{
    0 \ar[r] &  \cS(\sigma_{\epsilon=1}(\cB)) \ar[r]\ar[d]^{\Id} & \hat{\cB}_0 \ar[r] \ar@{^{(}->}[d] & \cC(X) \ar[r] \ar@{^{(}->}[d] & 0 \\
    0 \ar[r]  & \cS(\sigma_{\epsilon=1}(\cB))\ar[r] & \cB_0 \ar[r] & \overline{\cP}^0_{\ff_0-\phg}(X) \ar[r] & 0
 }
 \label{hcep.9a}\end{equation}
 induces a corresponding commutative diagram of six-term exact sequences in $K$-theory.  Since the inclusion
 $\cC(X)\subset \overline{\cP}^0_{\ff_0-\phg}(X)$ induces isomorphisms in $K$-theory, we conclude by the five-lemma that 
the inclusion $\hat{\cB}_0\subset \cB_0$ also induces isomorphisms in $K$-theory.  This means there are natural identifications
\begin{equation}
\begin{aligned}
   K_0(\cB_0) & \cong K_0(\hat{\cB}_0)  \\
                      & \cong K_0(\Con_{\iota}) \\
                      & \cong K_0(\Con_q),
\end{aligned}   
\label{hcep.10}\end{equation}
where we have used Theorem~\ref{fp.15} in the last step.  On the other hand, the principal symbol induces a short exact sequence
\begin{equation}
\xymatrix{
   0\ar[r] & \cC^*(T\FCX) \ar[r] & \cB_0  \ar[r]^-{\sigma_0} & \cC(S({}^{\pi}T^*X)\times[0,1)) \ar[r] &0.
}   
\label{hcep.11}\end{equation}
Since the quotient is contractible, this induces a natural identification 
\[
K_0(\cC^*(T\FCX))\cong K_0(\cB_0), 
\]
so that we obtain the desired identification by combining this with \eqref{hcep.10}.  Thanks to the naturality of our construction, one can readily check it induces a commutative diagram as in \eqref{hcep.7}.

\end{proof}

\section{Poincar\'e duality}\label{pd.0}

This last section will involve some  Kasparov bivariant $K$-theory. The unfamiliar reader may for instance have a look at \cite{Skandalis,Blackadar,DL_Col}.  We are using the notations of \cite{Blackadar} and \cite{DL_Col}.

Let $P\in\Psi_{\fS-\phg}^{0}(X;E,F)$ be a fully elliptic operator and let $Q$ be a parametrix for $P$ as constructed in Proposition \ref{fp.1}. Set 
$\mathbf{H}=L^2_{g_\pi}(X,E)\oplus L^2_{g_\pi}(X,F)$ and $\mathbf{P}=\begin{pmatrix} 0& Q\\P & 0\end{pmatrix}$. 
By Theorem \ref{fp.15}, the operator $\mathbf{P}$ is bounded and Fredholm on $\mathbf{H}$.  Let $\CI_{\pi}(X)\subset \CI(X)$ be the subalgebra of smooth functions on $X$ which are constants along the fibres of the fibration $\pi_i$ for each boundary hypersurface $H_i$ of $X$.  Clearly, we have a dense inclusion 
$\CI_{\pi}(X)\subset \cC({}^\fS X)$.  Denote by $\mathbf{m} : \cC(\SX)\to \cL(\mathbf{H})$ the representation given by multiplication.  

For $f\in \CI_{\pi}(X)$, $\mathbf{m}(f)$ is naturally a $\fS$-operator of order $0$.   
The commutator $[\mathbf{m}(f),\mathbf{P}]$ is a $\fS$-operator of order $-1$ such that $\sigma_{\partial_i}([\mathbf{m}(f),\mathbf{P}])=0$ for all $i$. 
Hence, by Theorem~\ref{compactness}, the commutator $[\mathbf{m}(f),\mathbf{P}]$ is a compact operator.  By the density of $\CI_{\pi}(X)$ in $\cC({}^{\fS} X)$, we conclude more generally that   
$[\mathbf{m}(f),\mathbf{P}]\in\cK(\mathbf{H})$ for all $f \in \cC({}^\fS X)$.  Since $\mathbf{P}^2-\Id\in \cK(\mathbf{H})$, this means $(\mathbf{H},\mathbf{m},\mathbf{P})$ is a Kasparov $(\cC({}^{\fS}X),\bbC)$-module. We denote by\begin{equation}\label{dp.2}
   [P]= [(\mathbf{H},\mathbf{m},\mathbf{P})]\in KK(\cC({}^\fS X),\CC)=K_0({}^\fS X).
\end{equation}
the corresponding Kasparov $(\cC({}^\fS X),\CC)$-cycle.   

It is straightforward to check that this Kasparov cycle only depends on the stable homotopy class of $P$.  This means this procedure defines a homomorphism of abelian groups
\begin{equation}
 \quan: \fe(X)\to K_0({}^{\fS}X).
\label{quan.1}\end{equation}
Using the identification of Theorem~\ref{hcep.6}, this can be seen as defining a homomorphism of abelian groups
\begin{equation}
    \PD:= \quan\circ \sigma_{\nc}^{-1}: K_{0}(\cC^*(T\FCX))\to K_0({}^{\fS}X).
\label{quan.2}\end{equation}
This map establishes a Poincar\'e duality between $T\FCX$ and ${}^{\fS}X$.  This can be described in a systematic way using Kasparov bivariant K-theory.

We first recall that two separable $C^*$-algebras $A$ and $B$ are Poincar\'e dual in $K$-theory if there  exist  $\alpha\in  KK(A\otimes B,\CC)$ and $\beta\in KK(\CC, A\otimes B) $ (minimal tensor products are understood) such that $\beta\underset{A}{\otimes}\alpha=1_{B}$ and
$\beta\underset{B}{\otimes}\alpha=1_{A}$. Once such an $\alpha$ is given, the element $\beta$ completing the Poincar\'e duality is unique. The element $\alpha$ (resp. $\beta$) is called the Dirac (resp. dual-Dirac) element of the Poincar\'e duality.  For any $C^*$-algebras $C,D$,  they provide  isomorphisms
$$ \cdot\underset{A}{\otimes}\alpha : KK(C,A\otimes D) \longrightarrow KK(B\otimes C, D), $$
with inverses given by
$$ \beta \underset{B}{\otimes}\cdot : KK(B\otimes C,D) \longrightarrow KK(C, A\otimes D). $$

We are interested in the special case where $A=\cC^*(T\FCX)$ and $B=\cC({}^{\fS}X)$.  To construct a Dirac element, consider the groupoid $\cG_{\pi-\ad}'=\cG_{\pi-\ad}\setminus \ff_{\pi}\times\{\epsilon=1\}$. 	
It enters in the short exact sequence
\begin{equation}\label{ex-seq-TFCY}
0 \longrightarrow\cC^*(\interior{X}\times \interior{X}\times (0,1])\longrightarrow\cC^*(\cG_{\pi-\ad}')\overset{\ev_{\FC}}{\longrightarrow}\cC^*(T\FCX)\longrightarrow 0,
\end{equation}
where $\ev_{\FC}$ is the obvious evaluation map induced by the inclusion
$T\FCX\subset \cG_{\pi-\ad}'$.  
The ideal is contractible so by classical arguments $[\ev_{\FC}]$ is invertible in $KK$-theory and we set
\begin{equation}\label{pre-dirac-FC}
 \partial^\FC_{X} =[\ev_{\FC}]^{-1}\otimes[\ev_{\epsilon=1}]\otimes[\interior{\mu}]^{-1}\in KK(\cC^*(T\FCX),\CC).
\end{equation}
Here, $\ev_{\epsilon=1}: \cC^*(\cG_{\pi-\ad}')\to \cC^*(\interior{X}\times \interior{X})$ is the obvious evaluation map at $\epsilon=1$ and  the homomorphism $\interior{\mu}$  is defined by $\lambda\mapsto\lambda q$ where $q$ is a rank one self-adjoint projection and $[\interior{\mu}]^{-1}$ is thus the Morita equivalence $\cC^*(\interior{X}\times \interior{X})\sim \CC$. 

The natural inclusion $\CI_{\pi}(X)\subset \Psi^0_{\fS-\ad-\phg}(X)$ extends to an inclusion $\iota: \cC(\SX)\hookrightarrow\overline{\cP}^0_{\fS-\ad-\phg}(X)$, where $\overline{\cP}^0_{\fS-\ad-\phg}(X)$ is the $\cC^*$-closure of 
$\Psi^0_{\fS-\ad-\phg}(X)$ with respect to the reduced norm for the groupoid $\cG_{\pi-\ad}$.   This can be used to define a `zero sections' homomorphism
\begin{equation}\label{zero-section-FC-ad}
\begin{array}{llcl}
\Psi^{\FC}_{\pi-\ad}: & \cC(\SX)\otimes\cC^*(\cG_{\pi-\ad}) &\longrightarrow  &\cC^*(\cG_{\pi-\ad})  \\ 
   &g\otimes a & \longmapsto & \iota(g) a.
\end{array}
\end{equation}
 By restriction to $T\FCX$, we also get a map
\begin{equation}\label{zero-section-FC}
\Psi^{\FC}_{X}:  \cC(\SX)\otimes\cC^*(T\FCX)  \longrightarrow \cC^*(T\FCX). 
\end{equation}
 Consider then the following Kasparov cycle,
\begin{equation}\label{Dirac-FC}
  D^{\FC}_{X}= [\Psi^{\FC}_{X}]\otimes\partial^\FC_{X} \in KK(\cC(\SX)\otimes\cC^*(T\FCX),\CC).
\end{equation}
 
\begin{theorem}\label{dualite-Poincare-FC}
 The Kasparov cycle $ D^{\FC}_{X}$ is the Dirac element of a Poincar\'e duality between $\cC^*(T\FCX)$ and $\cC(\SX)$. 
\end{theorem}
\begin{proof}
The groupoid $T\FCX$ is slightly different, but nevertheless intimately related to the noncommutative tangent space of \cite{Debord-Lescure} (see Corollary~\ref{idFCS.1} below).  At the cost of clarifying this relationship, it is therefore possible to transfer the Poincar\'e duality result of \cite{Debord-Lescure} to our context.    To have instead a more self-contained approach, we will adapt the proof of \cite{Debord-Lescure} to our context.  Really, this should be thought as a hybrid of the groupoid approach of \cite{Debord-Lescure} and the operator theoretic approach of \cite{NSS07c} (see also \cite{Melrose-Rochon06}).  

Let $H_1, \ldots, H_k$ be an exhaustive list of the boundary hypersurfaces of $X$ such that 
\[
          i<j, \; H_i\cap H_j\ne 0 \; \Longrightarrow \; H_i<H_j.  
\]  
Set $X_0=X$ and consider the non-compact manifolds with fibred corners 
\begin{equation}
        X_j := X\setminus \left( \bigcup_{i=1}^{j} H_i\right), \quad \mbox{for} \; j\in \{1,\ldots,k\}.
\label{xj.1}\end{equation}
Let 
\[
      \cC({}^{\fS}X_j)= \{  f\in \cC({}^{\fS}X);  \; \left. f\right|_{q(H_i)}=0 \; \mbox{for} \; i\le j\}
\]  
be the corresponding space of continuous functions on the associated stratified pseudomanifold, where $q:X\to {}^{\fS}X$ is the natural quotient map.  Finally, set $T_0\FCX= T\FCX$ and consider the subgroupoid
\[
    T \FCX_j := T\FCX \setminus \overset{\circ}{\left( \bigcup_{i=1}^j \ff_{\pi_i-\ad} \right)},
\]
where the interior is taken as a subset of $\pa X^2_{\pi-\ad}$.  Clearly, the morphism $\Psi_X^{\FC}$ restricts to give a morphism
\[
  \Psi^{\FC}_{X_j}:  \cC({}^{\fS}X_j)\otimes \cC^*(T\FCX_j) \to \cC^*(T\FCX),
\]
allowing us to define the following Kasparov cycle,
\[
      D^{\FC}_{X_j}= [\Psi^{\FC}_{X_j}] \otimes \pa^{\FC}_X \in KK(\cC({}^{\fS}X_j)\otimes \cC^*(T\FCX_j),\bbC).
\]
Now, for $j\in \{1,\ldots,k\}$, we have two natural short exact sequences of $C^*$-algebras,
\begin{gather}
\label{fl.1} \xymatrix{
        0 \ar[r] & \cC({}^{\fS}X_j) \ar[r] & \cC({}^{\fS}X_{j-1}) \ar[r]^-{\alpha} & \cC(S_j\setminus \pa S_j) \ar[r] & 0, 
}  \\
\label{fl.2}
\xymatrix{
  0\ar[r] & \cC^*(\cH_j) \ar[r] & \cC^*(T\FCX_{j-1}) \ar[r]^-{\beta} & \cC^*(T\FCX_j) \ar[r] & 0,
}
\end{gather}
where $\cH_j\subset \ff_{\pi_j-\ad}$ is the subgroupoid given by
\[
   \cH_j= T\FCX_{j-1}\setminus T\FCX_{j}.
\]
It is naturally Morita equivalent to the groupoid ${}^{\pi}TS_i$.  For this latter groupoid, we have a natural Kasparov cycle given by
\[
       D_{S_j}^{\MC}= [\Psi^{\MC}_{S_j}]\otimes \pa^{\FC}_{S_j} \in KK(\cC(S_j\setminus \pa S_j) \otimes \cC^*({}^{\pi}TS_j),\bbC),
\]
where 
\[
  \Psi^{\MC}_{S_j}: \cC(S_j\setminus \pa S_j)\otimes \cC^*({}^{\pi}TS_j)\to \cC^*( T^{\FC}S_j)
\]
is the morphism obtained by restriction of $\Psi^{\FC}_{S_j}$.  Using the Morita equivalence between $\cH_j$ and ${}^{\pi}TS_j$, this gives a corresponding Kasparov cycle $D^{\MC}_{\cH_j}\in KK(\cC(S_j\setminus \pa S_j)\otimes \cC^*(\cH_j),\bbC)$.  This cycle can be defined alternatively by $D^{\MC}_{\cH_j}= [\Psi^{\MC}_{\cH_j}]\otimes \pa^{\FC}_X$, where 
\[
            \Psi^{\MC}_{\cH_j}: \cC(S_j\setminus \pa S_j) \otimes \cC^*(\cH_j)\to \cC^* (T\FCX)
\]
is the morphism obtained by restriction of $\Psi^{\FC}_X$.  

Now, the cycle $D^{\FC}_{j-1}$, $D^{\FC}_{j}$ and $D^{\MC}_{\cH_j}$ can be used to obtain a diagram intertwining the six-term exact sequences in KK-theory associated to the short exact sequences \eqref{fl.1} and \eqref{fl.2},
\begin{equation}
\xymatrix@C=2.5cm{
\vdots \ar[d] &   \vdots \ar[d] \\
KK_q(A,B\otimes\cC({}^{\fS}X_j)) \ar[r]^{\underset{\cC({}^{\fS}X_j)}{\otimes} D^{\FC}_{X_j}} \ar[d]  & KK_q(A\otimes \cC^*(T\FCX_j),B)   \ar[d] \\
KK_q(A,B\otimes \cC({}^{\fS}X_{j-1})) \ar[d] \ar[r]^{\underset{\cC({}^{\fS}X_{j-1})}{\otimes} D^{\FC}_{X_{j-1}}} &  KK_q(A\otimes \cC^*(T\FCX_{j-1}),B)   \ar[d]    \\
KK_q(A,B\otimes \cC(S_j\setminus \pa S_j)) \ar[d]^{\pa_{\alpha}}\ar[r]^{\underset{\cC(S_j\setminus\pa S_j)}{\otimes}D^{\MC}_{\cH_j}}  &  KK_q(A\otimes\cC^{*}(\cH_j),B) \ar[d]^{\pa_{\beta}}  \\
 \vdots  & \vdots, 
}
\label{fl.3}\end{equation} 
where $A$ and $B$ are $C^*$-algebras.  

The result then follows from the following two claims.
\begin{claim}
The diagram \eqref{fl.3} is commutative up to sign.
\label{fl.5}\end{claim}
\begin{claim}
The Kasparov cycles $D_{X_k}^{\FC}= D^{\MC}_X$ and $D^{MC}_{S_j}$ for $j\in \{1,\ldots,k\}$ are Dirac elements.  
\label{fl.4}\end{claim}

Indeed, using the Morita equivalence between ${}^{\pi}TS_j$ and $\cH_j$, we see that $D^{\MC}_{\cH_j}$ is also a Dirac element.  Thus, starting with $j=k$ and applying the five-lemma to \eqref{fl.3}, we find that the map 
\[
\xymatrix@C=3cm{
     KK_q(A,B\otimes \cC({}^{\fS}X_{k-1})) \ar[r]^-{\underset{\cC({}^{\fS}X_{k-1})}{\otimes} D^{\FC}_{X_{k-1}}} &
     KK_q(A\otimes \cC^*(T\FCX_{k-1} ),B)
     }
\] 
is an isomorphism.  By \cite[Lemma~2]{Debord-Lescure}, this implies $D^{\FC}_{X_{k-1}}$ is a Dirac element.  Repeating this argument for $j=k-1,k-2,\ldots, 1$, we find more generally that $D^{\FC}_{X_j}$ is a Dirac element for all $j\in \{0,1,\ldots,k\}$.  In particular, $D^{\FC}_{X}= D^{\FC}_{X_0}$ is a Dirac element.  

Thus, it remains to prove the two claims, which we do below.  
\end{proof}

\begin{proof}[Proof of Claim~\ref{fl.5}]
The proof of the commutativity of the squares not involving boundary homomorphisms is straightforward and left to the reader.  To obtain the commutativity of the remaining squares, we need to show that 
\begin{equation}
    \pa_{\alpha}\underset{\cC({}^{\fS}X_j)}{\otimes} D^{\FC}_{X_j}=  \pa_{\beta} \underset{\cC^*(\cH_j)}{\otimes} D^{\MC}_{\cH_j},
\label{fl.5b}\end{equation}
where  $\pa_{\alpha}\in KK_1(\cC(S_j\setminus\pa S_j), \cC({}^{\fS}X_j))$ and  $\pa_{\beta}\in KK_1(\cC^*(T\FCX_j), \cC^*(\cH_j))$ are the boundary homomorphisms associated to the short exact sequences \eqref{fl.1} and \eqref{fl.2}.  From the definition of $D^{\FC}_j $ and $D^{\MC}_{\cH_j}$, this means we need to show that
\begin{equation}
\pa_{\alpha}\underset{\cC({}^{\fS}X_j)}{\otimes} [\Psi^{\FC}_{X_j}]= \pa_{\beta} \underset{\cC^*(\cH_j)}{\otimes} [\Psi^{\MC}_{\cH_j}] 
\label{fl.6}\end{equation}
in $ KK_1(\cC(S_j\setminus\pa S_j) \otimes \cC^*(T\FCX_j), \cC^*(T\FCX))$.  To see this, consider the subgroupoid
$L_j:= T\FCX_{j}\cap \ff_0\cap \ff_{\pi_j-\ad} \subset \left. {}^{\pi}TX \right|_{H_j}$.  Thus, there is a natural restriction homomorphism $\cC^*(T\FCX_j) \to\cC^*(L_j)$.  There is also an obvious multiplication homomorphism 
\[
           \cC(S_j\setminus \pa S_j)\otimes \cC(L_j) \to \cC( \left. {}^{\pi} TX\right|_{\overset{\circ}{H}_{j}}).
\]
Let also $\cN_{H_j}$ be a tubular neighborhood of $H_j$ coming from an iterated fibred tube system and set $W= \overset{\circ}{\cN}_{H_j}$.  The tube system of $H_j$ induces an identification
\begin{equation}
       \cC(\bbR)\otimes\cC^*(\left.{}^{\pi}TX\right|_{\overset{\circ}{H}_j}) \to \cC^*(TW).
\label{fl.7}\end{equation}
On the other hand, the short exact sequence
\[
\xymatrix{
0 \ar[r]& \cC^*(\overset{\circ}{\ff}_{\pi_j-\ad}) \ar[r] & \cC^*( \overset{\circ}{\ff}_{\pi_j-\ad}\cup \left.{}^{\pi}TX\right|_{\overset{\circ}{H}_j}) \ar[r] & \cC^*( \left.{}^{\pi}TX\right|_{\overset{\circ}{H}_j}) \ar[r] & 0,
}
\]
induces a boundary homomorphism in $KK_1(\cC^*(\left.{}^{\pi}TX\right|_{\overset{\circ}{H}_j}), \cC^*(\overset{\circ}{\ff}_{\pi_j-\ad}))$.  By composing with the inclusion $\cC^*(\overset{\circ}{\ff}_{\pi_i-\ad})\subset \cC^*(\cH_j)$, this induces a morphism $\pa\in KK_0(\cC(\bbR)\otimes \cC^*( \left.{}^{\pi}TX\right|_{\overset{\circ}{H}_j}), \cC^*(\cH_j))$.  Using the identification \eqref{fl.7}, this gives a corresponding element in $\pa'\in KK_0(\cC^*(TW),\cC^*(\cH_j))$ inducing a commutative diagram of Kasparov cycles
\[
\xymatrix{     \cC(\bbR)\otimes \cC( \left.{}^{\pi}TX\right|_{\overset{\circ}{H}_j}) \ar[r] \ar[dr]^{\pa} & \cC^*(TW) \ar[d]^{\pa'} \\
                       & \cC^*(\cH_j).
}
\]  
The result then follows by noticing this fits into a bigger diagram commutative up to sign involving the Kasparov cycles of \eqref{fl.6}, 
\begin{equation}
\xymatrix{
        & \cC(\bbR)\otimes \cC(S_j\setminus\pa S_j) \otimes \cC(T\FCX_j)\ar[d] \ar@/_6cm/[ddd]^{\pa_{\alpha}\underset{\cC({}^{\fS}X_j)}{\otimes} [\Psi^{\FC}_{X_j}]} \ar@/^6cm/[ddd]_{\pa_{\beta} \underset{\cC^*(\cH_j)}{\otimes} [\Psi^{\MC}_{\cH_j}] }   &      \\
        &   \cC(\bbR) \otimes \cC^*(  \left.{}^{\pi}TX\right|_{\overset{\circ}{H}_j}) \ar[dl]\ar[dr]^{\pa} &     \\
      \cC^*(TW)\ar[rr]^{\pa'}\ar[dr]  &  &  \ar[dl] \cC^*(\cH_j)     \\
        &    \cC^*(T\FCX).  &   
    }    
\label{fl.8}\end{equation}

\end{proof}

For Claim~\ref{fl.4}, this is the Poincar\'e duality for manifolds with corners obtained in \cite{Melrose-Piazza1}.  The result of \cite{Melrose-Piazza1} is not formulated in terms of Dirac elements,  but this can be remedied easily by using the semiclassical $b$-double space (or the semiclassical cusp double space).  For the convenience of the reader, we will provide a brief outline.  First, the semiclassical $b$-double space is defined by
\[
      X^2_{b-\ad}=[X^2_b\times[0,1]_{\epsilon}; \Delta_b\times\{0\}],
\]  
where $\Delta_b\subset X^2_b$ is the lifted diagonal.  Denote the new face obtained by this blow-up by $\ff_{0,b}$.  Notice that the $b$-tangent bundle is naturally included in $\ff_{0,b}$.  If $\ff_{b-\ad}$ is the union of all the boundary hypersurfaces intersecting the lift of $\Delta_b\times [0,1]$ in $X^2_{b-\ad}$, we get a corresponding groupoid 
\[
     T^bX:= \overset{\circ}{\ff}_{b-\ad}\setminus (\overset{\circ}{\ff}_{b-\ad}\cap X^2_b\times\{1\}).
\]
Using evaluation maps as in the fibred corners case, one can define a natural Kasparov cycle $\pa^b_X\in KK(\cC^*(T^bX),\bbC)$.  There is also a `zero sections' morphism $\Psi^b_X: \cC(X)\otimes \cC^*(T^bX)\to \cC^*(T^bX)$, and so a corresponding Kasparov cycle $D^b_X= [\Psi^b_X]\otimes \pa^b_X$ in $KK(\cC(X)\otimes \cC^*(T^bX),\bbC)$.  

Let $H_1,\ldots, H_k$ be an exhaustive list of boundary hypersurfaces of $X$ and set
\[
       X_j= X\setminus \bigcup_{i=1}^j H_i, \quad X_j'= X\setminus \bigcup_{i=j+1}^{k} H_i,
\]
with the convention that $X_0=X=X_k'$.  Then, by restriction of $D^b_X$, we obtain corresponding cycles
\[
              D^b_{X_j}\in KK(\cC(X_j)\otimes \cC^*({}^{b}TX_j'),\bbC), \quad \mbox{where} \, {}^bTX_j' = \left. {}^bTX \right|_{X_j'}.
\]
Since $D^{\MC}_X=D^b_{X_k}$, Claim~\ref{fl.4} is a consequence of the following proposition.
\begin{proposition}
If $X$ is a compact manifold with corners and $H_1,\ldots, H_k$ is an exhaustive list of its boundary hypersurfaces, then the Kasparov cycle $D^b_{X_j}$ is a Dirac element for all $j\in \{0,1,\ldots,k\}$.
\label{mepi.1}\end{proposition}
\begin{proof}
From \cite{Kasparov}, we know that $D^b_{X_0}$ is a Dirac element.  This suggests to proceed by induction on the depth of $X$.  Thus, assume the proposition is true for all manifolds with corners of depth less than the one of $X$.  If $N_{H_j}= H_j\times [0,1]$ is a collar neighborhood of $H_j$ in $X$, then, after making obvious identifications, the inclusion $N_{H_j}\subset X$ induces two natural short exact sequences of $C^*$-algebras,
\begin{gather}
\label{mepi.2}\xymatrix{ 0 \ar[r] & \cC(X_j)\ar[r] & \cC(X_{j-1}) \ar[r] & \cC(\hat{H}_j\times [0,1])  \ar[r] & 0,
} \\
\label{mepi.3}\xymatrix{
0 \ar[r] & \cC^*({}^bT\check{H}_j\times T(0,1))\ar[r] & \cC^*({}^bTX_{j-1}') \ar[r] & \cC^*({}^bTX_j') \ar[r] & 0,
}
\end{gather} 
where 
\[
\hat{H}_j= H_j\setminus \left(\bigcup_{i=1}^{j-1} (H_i\cap H_j) \right), \quad   \check{H}_{j}= H_j \setminus \left(\bigcup_{i=j+1}^{k} (H_i\cap H_j) \right).
\]
By our inductive assumption, the cycle $D^b_{\hat{H}_j}\in KK(\cC(\hat{H}_j)\otimes \cC^*({}^bT\check{H}_j),\bbC)$ is a Dirac element.  On the other hand, $D^b_{[0,1]}\in KK(\cC([0,1])\otimes \cC^*(T(0,1)),\bbC)$ is a Dirac element by the result of \cite{Kasparov}.  This means the corresponding cycle
\[
       D^b_{\hat{H}_j\times[0,1]}= D^b_{\hat{H}_j}\otimes D^b_{[0,1]} \in KK(\cC(\hat{H}_j\times [0,1])\otimes\cC^*({}^bT\check{H}_j\times T(0,1)),\bbC)
\] 
is a Dirac element.  Now, this Dirac element combines with $D^b_{X_j}$ and $D^b_{X_{j-1}}$ to give a diagram intertwining the six-term exact sequences in $KK$-theory associated to \eqref{mepi.2} and \eqref{mepi.3},  
\begin{equation}
\xymatrix@C=2.5cm{
\vdots \ar[d] &   \vdots \ar[d] \\
KK_q(A,B\otimes\cC(X_j)) \ar[r]^{\underset{\cC(X_j)}{\otimes} D^{b}_{X_j}} \ar[d]  & KK_q(A\otimes \cC^*({}^bT X_j'),B)   \ar[d] \\
KK_q(A,B\otimes \cC(X_{j-1})) \ar[d] \ar[r]^{\underset{\cC(X_{j-1})}{\otimes} D^{b}_{X_{j-1}}} &  KK_q(A\otimes \cC({}^bTX_{j-1}'),B)   \ar[d]    \\
KK_q(A,B\otimes \cC(\hat{H}_j\times[0,1])) \ar[d]\ar[r]^{\underset{\cC(\hat{H}_j\times[0,1])}{\otimes}D^{b}_{\hat{H}_j\times[0,1]}}  &  KK_q(A\otimes\cC^{*}({}^bT\check{H}_j\times T(0,1)),B)\ar[d]   \\
 \vdots  & \vdots, 
}
\label{mepi.4}\end{equation}
where $A$ and $B$ are $C^*$-algebras.  
Using a similar method as for \eqref{fl.3}, it can be shown that this diagram is commutative up to sign.  Thus, starting with $j=0$ and applying the five-lemma recursively to \eqref{mepi.4} as well as \cite[Lemma 2]{Debord-Lescure}, we conclude that $D^b_{X_j}$ is a Dirac element for all $j\in \{0,1,\ldots,k\}$.   

\end{proof}

Since the noncommutative tangent space $T^{\fS}X$ of \cite{Debord-Lescure} is also Poincar\'e dual to the stratified pseudomanifold ${}^{\fS}X$, Theorem~\ref{dualite-Poincare-FC} has the following consequence.

\begin{corollary}
The $C^*$-algebras $\cC^*(T^{\fS}X)$ and $\cC^*(T\FCX)$ are $KK$-equivalent.
\label{idFCS.1}\end{corollary}
\begin{proof}
Let $D^{\fS}_X\in KK(\cC^*(T^{\fS}X)\otimes \cC({}^{\fS}X),\bbC)$ be the Dirac element of \cite{Debord-Lescure} that provides the Poincar\'e duality between $\cC^*(T^{\fS}X)$ and $\cC({}^{\fS}X)$.  Denote by
\[
    (D^{\FC}_X)^{-1}\in KK(\bbC, \cC^*(T\FCX)\otimes\cC({}^{\fS}X)), \quad (D^{\fS}_X)^{-1}\in KK(\bbC, \cC^*(T^{\fS}X)\otimes\cC({}^{\fS}X)),
\]
the dual-Dirac elements of $D^{\FC}_X$ and $D^{\fS}_{X}$ respectively.  Then the element
\[
  \alpha= (D^{\fS}_X)^{-1} \otimes_{\cC({}^{\fS}X)} D^{\FC}_X \in KK(\cC^*(T\FCX), \cC^*(T^{\fS}X))
\]
is a $KK$-equivalence between $\cC^{*}(T\FCX)$ and $\cC^{*}(T^{\fS}X)$ with inverse
\[
   \alpha^{-1}= (D^{\FC}_X)^{-1}\otimes_{\cC({}^{\fS}X)} D^{\fS}_X \in KK(\cC^*(T^{\fS}X), \cC^*(T\FCX)).
\]
\end{proof}

The map $\PD$ in \eqref{quan.2} can be described in terms of the Dirac element $D^{\FC}_X$.

\begin{theorem}\label{interpretation-dp}
If $P\in \Psi^0_{\fS-\phg}(X;E,F)$ is a fully elliptic operator, then
\begin{equation}\label{dp.with.S-calcul}
 \sigma_{nc}(P) \underset{C^*(T\FCX) }{\otimes}D_{X}^{\FC} = [P].
\end{equation}
In particular, the map $\PD$ in \eqref{quan.2} is an isomorphism of abelian groups.  
\end{theorem}

\begin{proof}

Let $P\in\Psi_{\fS-\phg}^{0}(Y;E,F)$ be a fully elliptic operator and let $Q$ be a parametrix for $P$ as constructed in Proposition~\ref{fp.1}.   Let $\cP\in \Psi^0_{\fS-\ad-\phg}(X;E,F)$ and $\cQ\in \Psi^0_{\fS-\ad-\phg}(X;F,E)$ be fully elliptic semiclassical $\fS$-operators such that $\sigma_{\epsilon=1}(\cP)=P$ and $\sigma_{\epsilon=1}(\cQ)=Q$.  Without loss of generality, we can choose $\cQ$ such that
\begin{equation}\label{full-ell_nc-symbol}
\cP\cQ-1\in \Psi^{-\infty}_{\fS-\ad}(X;F),\quad \cQ\cP-1\in \Psi^{-\infty}_{\fS-\ad}(X;E).
\end{equation}
By construction,   $a:=\cP|_{T\FCX}$  
 is a pseudodifferential operator on the groupoid $T\FCX$ of order
$0$, so it gives a (bounded) morphism between the $\cC^*(T\FCX)$-Hilbert modules $\cC^*(T\FCX, E)$ and $\cC^*(T\FCX,F)$. Reverting $E$ and $F$, the same is true for $b:= \cQ|_{T\FCX}$ so we get a bounded morphism 
 $$ \mathbf{a}=\begin{pmatrix}0 & b \\ a& 0\end{pmatrix} \in \cL(C^*(T\FCX, E\oplus F)). $$
Since $\sigma_{\epsilon=1}(\cP)=P$ and $\sigma_{\epsilon=1}(\cQ)=Q$, we have that $a|_{\epsilon=1}$ is invertible with inverse $b|_{\epsilon=1}$ so that 
$\mathbf{a}^2-\mathrm{Id}\in \cK(C^*(T\FCX, E\oplus F))$. This means 
\begin{equation}\label{nc-symbol-class}
 \left(\cC^*(T\FCX, E\oplus F) , \mathbf{a} \right)
\end{equation}
is a Kasparov $(\CC,C^*(T\FCX))$-cycle. Its class in $K_0(\cC^*(T\FCX))$ is the  element $\sigma_{\nc}(P)$ defined in \eqref{hcep.5b}.

Similarly, we get a $K$-theory class associated with $\cP$. As before, 
 $$  \mathcal{T}:= \begin{pmatrix}0 &  \cQ  \\  \cP  & 0\end{pmatrix} \in \cL(\cC^*(\cG_{\pi-\ad}',E\oplus F))$$
and $  \mathcal{T}^2-1 \in \cK(\cC^*(\cG_{\pi-\ad}',E\oplus F))$, so that 
 $$ [\cT]=\left(\cC^*(\cG_{\pi-\ad}',E\oplus F),  \mathcal{T}\right) \in K_0(\cC^*(\cG_{\pi-\ad}')).$$
The cycle $[\cT]$ is such that 
\begin{equation}\label{lift.nc.symbol}
[\cT]\otimes[\ev_{\FC}]=(\ev_{\FC})_*[\cT] = \sigma_{\nc}(P).
\end{equation}
In order to achieve the computation proving \eqref{dp.with.S-calcul}, we observe that the homomorphism (\ref{zero-section-FC-ad}) naturally induces a map
\begin{equation}\label{zero-section-FC-ad-prime}
\Psi^{\FC}_{\pi-\ad'} : \cC(\SX)\otimes\cC^*(\cG_{\pi-\ad}')\longrightarrow\cC^*(\cG_{\pi-\ad}')
\end{equation}
leading to the equality of homomorphisms
\begin{equation}\label{ev.commutes.psi}
 \ev_{\FC}\circ\Psi^{\FC}_{\pi-\ad'} = \Psi^{\FC}_{X}\circ(\id_{\cC(\SX)}\otimes\ev_{\FC}).
\end{equation}  
Now, using the basic properties of the Kasparov product, we have,
\begin{alignat*}{2}\label{quant.first.part}
    \sigma_{nc}(P) \underset{\cC^*(T\FCX)}{\otimes}D^{\FC}_X &= 
    ([\cT]\otimes [\ev_{\FC}])\underset{\cC^*(T\FCX)}{\otimes} D^{\FC}_X,&&  \hbox{by \eqref{lift.nc.symbol}, } \\
  &= \tau_{\cC(\SX)}([\cT]\otimes [\ev_{\FC}])\otimes D^{\FC}_X  &&  \\
  &= \tau_{\cC(\SX)}([\cT])\otimes [\id_{\cC(\SX)}\otimes\ev_{\FC}]\otimes D^{\FC}_X  && \\
  &= \tau_{\cC(\SX)}([\cT])\otimes [\Psi^{\FC}_{\pi-\ad'}]\otimes[\ev_{\epsilon=1}]\otimes [\interior{\mu}]^{-1},
              && \ \hbox{by \eqref{Dirac-FC}, \eqref{ev.commutes.psi}}.
\end{alignat*}
The next step requires some details. We have 
 $$ \tau_{\cC(\SX)}(\cT) =\left(\cC(\SX)\otimes\cC^*(\cG_{\pi-\ad}',E\oplus F),l, \id \otimes \begin{pmatrix}0 &  \cQ  \\  \cP  & 0\end{pmatrix}\right)$$
where $\cC(\SX)\otimes\cC^*(\cG_{\pi-\ad}',E\oplus F)$ has the obvious right $\cC(\SX)\otimes\cC^*(\cG_{\pi-\ad}')$-module structure and the representation $l$ is defined by: $l(f)(g\otimes\xi)=(fg)\otimes\xi$. We then have,
\begin{equation}
\begin{split}
 & \tau_{\cC(\SX)}(\cT)\otimes \Psi^{\FC}_{\pi-\ad'} = \\
  & \left( [\cC(\SX)\otimes\cC^*(\cG_{\pi-\ad}',E\oplus F)]\underset{\Psi^{\FC}_{\pi-\ad'}}{\otimes}\cC^*(\cG_{\pi-\ad}'),l\otimes\id, (\id \otimes \begin{pmatrix}0 &  \cQ  \\  \cP  & 0\end{pmatrix})\otimes \id\right).
\end{split}
\end{equation}
By construction, $\cC^*(\cG_{\pi-\ad}',E\oplus F)$ is a finitely generated projective Hilbert $\cC^*(\cG_{\pi-\ad}')$-module, so we can choose a self-adjoint idempotent $e\in M_r(\cC^*(\cG_{\pi-\ad}'))$ such that $\cC^*(\cG_{\pi-\ad}',E\oplus F)=e\cC^*(\cG_{\pi-\ad}')^r$.  This choice provides a Hilbert  $\cC^*(\cG_{\pi-\ad}')$-module isomorphism
 $$ [\cC(\SX)\otimes\cC^*(\cG_{\pi-\ad}',E\oplus F)]\underset{\Psi^{\FC}_{Y-\ad'}}{\otimes}\cC^*(\cG_{\pi-\ad}')\simeq\cC^*(\cG_{\pi-\ad}',E\oplus F) $$
under which the representation $l\otimes\id$ corresponds to $ \nu : \cC(\SX)\to\cL(\cC^*(\cG_{\pi-\ad}',E\oplus F))$ defined by
 $$ \nu(f)(e(b_1,\ldots,b_r)):=e(\Psi^{\FC}_{\pi-\ad'}(f,b_1),\ldots,\Psi^{\FC}_{\pi-\ad'}(f,b_r)),$$
and the operator $ (\id\otimes \begin{pmatrix}0 &  \cQ  \\  \cP  & 0\end{pmatrix})\otimes \id$ simply corresponds to $\begin{pmatrix}0 &  \cQ  \\  \cP  & 0\end{pmatrix}$. 
In other words, we have the equality
\begin{equation}
 \tau_{\cC(\SX)}([\cT])\otimes [\Psi^{\FC}_{Y-\ad'}] =\left[\left(\cC^*(\cG_{\pi-\ad}',E\oplus F),\nu, \begin{pmatrix}0 &  \cQ  \\  \cP  & 0\end{pmatrix}\right)\right] 
\end{equation}
in $KK(\cC(\SX),\cC^*(\cG_{\pi-\ad}'))$.  It follows that,
\begin{equation}\label{quant.final.part} 
\begin{aligned}
 &   \sigma_{nc}(P) \underset{\cC^*(T\FCX)}{\otimes}D^{\FC}_X  \\
&= \left[\left(\cC^*(\cG_{\pi-\ad}',E\oplus F),\nu, \begin{pmatrix}0 &  \cQ  \\  \cP  & 0\end{pmatrix}\right)\right]\otimes[\ev_{\epsilon=1}]\otimes [\interior{\mu}]^{-1} \\
 &= \left[\left(\cC^*(\interior{X}\times\interior{X},E\oplus F),\nu_{\epsilon=1},\begin{pmatrix}0 & Q \\ P & 0\end{pmatrix}\right)\right]\otimes [\interior{\mu}]^{-1}  \\
 &= \left[\left(L^2_\pi(X,E\oplus F),\mathbf{m},\begin{pmatrix}0 & Q\\ P & 0\end{pmatrix}\right)\right] =[P].
\end{aligned}
\end{equation}

\end{proof}

\bibliography{CoinsFibres} 
\bibliographystyle{amsplain}

\end{document}